\documentclass[letterpaper, 11pt]{amsart}
\usepackage{mathtools}
\usepackage{amsmath}
\usepackage{amssymb}
\usepackage{yhmath}
\usepackage{graphicx}
\usepackage{mathrsfs}
\usepackage{bbm}
\usepackage{xcolor}
\usepackage{tikz-cd}
\usepackage{tikz}
\usetikzlibrary{patterns}
\usepackage{hyperref}

\usepackage{verbatim}
\usepackage{float}

\DeclareMathAlphabet{\mathpzc}{OT1}{pzc}{m}{it}

\usepackage{thmtools}
\usepackage{thm-restate}

\usepackage{caption}

\newtheorem{theorem}{Theorem}[section]

\newtheorem*{claim*}{Claim}

\newtheorem{lemma}[theorem]{Lemma}
\newtheorem{lem}[theorem]{Lemma}
\newtheorem{corollary}[theorem]{Corollary}

\newtheorem{cor}[theorem]{Corollary}

\newtheorem{proposition}[theorem]{Proposition}

\newtheorem{prop}[theorem]{Proposition}

\theoremstyle{definition}
\newtheorem{definition}[theorem]{Definition}

\theoremstyle{remark}

\newtheorem{rmk}[theorem]{Remark}
\newtheorem{Rmk}[theorem]{Remark}

\numberwithin{equation}{section}

\newcommand{\norm}[1]{\lVert#1\rVert}
\newcommand{\op}{\operatorname}

\newcommand{\be}{\begin{equation}}
\newcommand{\ee}{\end{equation}}
\newcommand{\Ga}{\Gamma}

\newcommand{\R}{\mathbb R}

\newcommand{\Z}{\mathbb Z}

\newcommand{\ga}{\gamma}

\newcommand{\la}{\lambda}
\newcommand{\La}{\Lambda}
\newcommand{\inte}{\op{int}}
\newcommand{\ba}{\backslash}

\newcommand{\cal}{\mathcal}
\newcommand{\br}{\mathbb R}

\newcommand{\F}{\cal F}

\newcommand{\stab}{\op{Stab}}

\newcommand{\diam}{\op{diam}}
\newcommand{\vol}{\op{Vol}}

\newcommand{\G}{\Gamma}
\newcommand{\m}{\mathsf{m}}

\renewcommand{\frak}{\mathfrak}

\newcommand{\e}{\varepsilon}

\renewcommand{\L}{\mathcal L}
\newcommand{\fa}{\mathfrak a}

\renewcommand{\i}{\op{i}}

\renewcommand{\P}{\mathbb{P}}

\newcommand{\fg}{\frak g}

\newcommand{\Leb}{\op{Leb}}

\newcommand{\lat}{\La_{\theta}}
\newcommand{\ft}{\F_\theta}

\newcommand{\lac}{\lat^{\mathsf{con}}}

\renewcommand{\epsilon}{\e}

\renewcommand{\d}{\mathsf{d}}
\newcommand{\sq}{\mathsf{Q}}
\newcommand{\SL}{\op{SL}}
\newcommand{\codim}{\op{codim}}

\title[Ergodic dichotomy]{Ergodic dichotomy for subspace flows \\ in higher rank}

\author{Dongryul M. Kim}
\address{Department of Mathematics, Yale University, New Haven, CT 06511}
\email{dongryul.kim@yale.edu}

\author{Hee Oh}
\address{Department of Mathematics, Yale University, New Haven, CT 06511.}
\email{hee.oh@yale.edu}

\thanks{Oh is partially supported by the NSF grant No. DMS-1900101.}

\author{Yahui Wang}
\address{Department of Mathematics, Yale University, New Haven, CT 06511}
\email{amy.wang.yw735@yale.edu}

\begin{document}

\maketitle

\begin{abstract}
In this paper, we study the ergodicity of a one-parameter diagonalizable subgroup of a connected semisimple real algebraic group $G$ acting on a homogeneous space or, more generally, a homogeneous-like space, equipped with a Bowen-Margulis-Sullivan type measure. These flow spaces are associated with Anosov subgroups of $G$, or more generally, with transverse subgroups of $G$.

 We obtain an ergodicity criterion similar to the Hopf-Tsuji-Sullivan dichotomy for the ergodicity of the geodesic flow on hyperbolic manifolds. In addition, we  extend this criterion to the action of any connected diagonal subgroup of arbitrary dimension.
Our criterion provides a codimension dichotomy on the ergodicity of a connected diagonalizable subgroup for general Anosov subgroups. This generalizes an earlier work by Burger-Landesberg-Lee-Oh for {\it Borel} Anosov subgroups.
\end{abstract}

\tableofcontents

\section{Introduction}
A continuous flow $\phi_t$ on a phase space $X$ with an invariant measure $\m$ is called {\it ergodic} if any invariant measurable subset has measure zero or co-measure zero.

If $\m$ is a probability measure on $X$, the Birkhoff ergodic theorem states that the ergodicity of the dynamical system $(X,\phi_t, \m)$ is equivalent to the condition that for any $f\in L^1(X, \m)$, the time average along the trajectory of almost every point $x\in X$ is equal to the space average with respect to $\m$:
$$\lim_{T\to \infty}\frac{1}{T} \int_0^T f(\phi_t(x)) dt =\int_X f\; d\m \quad\text{ for $\m$-almost all $x\in X$}. $$
Hence, ergodicity ensures that the system's behavior, when observed over a long time, reflects the statistical properties of the entire phase space. This property is crucial in understanding the long-term behavior of dynamical systems.

The main dynamical system of interest in homogeneous dynamics arises from  the quotient $\Ga\ba G$  of a connected semisimple real algebraic group $G$ by a discrete subgroup $\Ga$ of $G$. Any one-parameter subgroup $H=\{\phi_t:t\in \br\}$ of $G$ acts on $\Ga\ba G$ by translations on the right, giving rise to a continuous dynamical system $(\Ga\ba G, H)$. The ergodicity in {\it finite-volume} homogeneous dynamics is well understood thanks to the Moore ergodicity theorem: if  $\Ga$ is an irreducible lattice, 
any non-compact closed subgroup $H$ acts ergodically on $(\Gamma\ba G, \m_G)$
where  $\m_G$ denotes a $G$-invariant (finite) measure on $\Gamma\ba G$.

The concept of ergodicity becomes much more delicate and challenging to prove for an {\it infinite} measure system.
Firstly, the Birkhoff ergodic theorem no longer holds, but we have the Hopf ratio ergodic theorem:
 for a conservative and ergodic action of a continuous flow $\phi_t$ on a $\sigma$-finite measure space $(X,\m)$, for any $f, g\in L^1(X,\m)$ with $g>0$, we have
 $$\lim_{T\to \infty}\frac{\int_0^T f(\phi_t(x))dt}{\int_0^T g(\phi_t(x) )dt} =\frac{\int_X f\; d\m}{\int_X g\; d\m } \quad\text{ for $\m$-almost all $x\in X$}. $$
 When $\m(X)=\infty$, while the denominator ${\int_0^T g(\phi_t(x) )dt} $ depends on the initial position $x$, unlike in the finite measure case, the ratio of time averages
 still converges to the ratio of space averages. In particular, almost all trajectories are dense in the phase space.
 
One of the first significant results on ergodicity in infinite measure systems is the Hopf-Tsuji-Sullivan dichotomy for the geodesic flow on hyperbolic manifolds with respect to Bowen-Margulis-Sullivan measures. The hyperbolic nature of the geodesic flow  and the quasi-product structure of Bowen-Margulis-Sullivan measures, with respect to the stable and unstable foliations for the geodesic flow, have been crucial in their approach. This methodology was successfully extended to a one-parameter diagonalizable flow on higher rank homogeneous spaces, which are quotients of a connected semisimple real algebraic group of higher rank by Borel Anosov subgroups, by Burger-Landesberg-Lee-Oh \cite{BLLO}. The main aim of this paper is to generalize this result for general Anosov subgroups that are not necessarily Borel Anosov, as well as to address the action of any connected diagonalizable subgroup.

\subsection*{Background and Motivation} To give some background, let $G$ be a connected semisimple real algebraic group. 
Fix a Cartan decomposition $$G=KA^+K$$ where
$K$ is a maximal compact subgroup of $G$ and $A^+=\exp \fa^+$ is a positive Weyl chamber of a maximal split torus $A$ of $G$.  
 Let $M$ be the centralizer of $A$ in $K$. For any Zariski dense discrete subgroup $\Ga$ of $G$, we have a natural locally compact Hausdorff space $\Ga\ba G/M$ on which $A$ acts  by translations on the right. 
For any non-zero vector $u\in \fa^+$, consider the one-parameter subgroup $A_u=\{\exp tu:t\in \br\}$. The ergodicity criterion for the $A_u$-action on $\Ga\ba G/M$ with respect to a Bowen-Margulis-Sullivan measure was obtained by Burger-Landesberg-Lee-Oh \cite{BLLO} 
in terms of the divergence of an appropriate directional Poincar\'e series. In particular,
it was shown that for Borel Anosov subgroups of $G$ (in other words, Anosov subgroups with respect to a minimal parabolic subgroup), the ergodicity of the $A_u$-action is completely determined by the rank of $G$. Soon after \cite{BLLO} appeared,  Sambarino \cite{sambarino2022report} gave a different proof for this rank dichotomy, but it applied only when $\op{rank} G \neq 3$.

On the other hand, the ergodicity criterion of \cite{BLLO} is not very useful in practice for Anosov subgroups that are not Borel Anosov, since there seems to be no way to decide whether the relevant directional Poincar\'e series diverges or not.

In this paper, we consider a different dynamical space than $\Ga\ba G/M$ depending on the Anosov type of $\Ga$. Let $\theta$ be a non-empty subset of simple roots. Consider the standard parabolic subgroup $P_\theta=A_\theta S_\theta N_\theta$ where $A_\theta S_\theta$ is a Levi-subgroup with $A_\theta$ being the central real split torus and $N_\theta$ is the unipotent radical. The double quotient space $\Ga\ba G/S_\theta$ is
precisely $\Ga\ba G/M$ when $P_\theta$ is a minimal parabolic subgroup, but it is not Hausdorff for a  general  $P_\theta$.
If $\Ga$ is a $\theta$-Anosov subgroup,  there exists a locally compact Hausdorff subspace $\Omega_\theta\subset \Ga\ba G/S_\theta$ on which  $A_\theta$ acts by translations. 
We will  obtain the ergodicity criterion for the action of a one-parameter subgroup of $A_\theta$ on $\Omega_\theta$ in terms of the associated directional Poincar\'e series.

In fact, this viewpoint and our criterion can be applied to a much more general class of
discrete subgroups, called $\theta$-transverse subgroups. For $\theta$-Anosov subgroups, our criterion provides the dichotomy for the ergodicity of $A_u$ on $\Omega_\theta$ with respect to a Bowen-Margulis-Sullivan measure in terms of the cardinality of $\theta$. When $\#\theta =3$, this was an open question while the other cases were obtained by Sambarino \cite{sambarino2022report}.
Although our proof closely follows the general strategy of \cite{BLLO}, a major difficulty arises from the non-compactness of $S_\theta$ which requires new ideas and new technical arguments to overcome.

\subsection*{Flow space} To discuss $\theta$-transverse subgroups and the associated flow space $\Omega_\theta$, we need to introduce some notation and definitions.  We denote by $\mu : G \to \fa^+$ the Cartan projection defined by the condition $g\in K\exp \mu(g) K$ for $g \in G$.
Let $\Pi$ be the set of all simple roots for $(\op{Lie} G, \frak a^+)$. 
Let  $\i :\Pi \to \Pi $ denote the opposition involution (see \eqref{opp0}).
Fix a non-empty subset $$\theta\subset \Pi.$$
Consider the $\theta$-boundary: $$\F_\theta=G/P_\theta,$$
where $P_\theta$ is the standard parabolic subgroup associated with $\theta$.
  We say that two points $\xi\in \F_\theta$ and $\eta\in \F_{\i(\theta)}$ are in general position if the pair
$(\xi,\eta)$ belongs to the unique open $G$-orbit in $\F_\theta\times \F_{\i(\theta)}$ under the diagonal action of $G$.

 Let $\Ga<G$ be a Zariski dense discrete subgroup. Let $\La_\theta$ denote the $\theta$-limit set of $\Ga $, which is the unique $\Ga$-minimal subset of $\F_\theta$ (Definition \ref{def.limitset}). 
 We say that $\Ga$ is {\it $\theta$-transverse} if it satisfies
\begin{itemize}
    \item ({\it $\theta$-regularity}):
$ \liminf_{\ga\in \Ga} \alpha(\mu({\ga}))=\infty $ for all $\alpha\in \theta$; 
\item ({\it $\theta$-antipodality}):
any distinct $\xi, \eta\in \La_{\theta\cup \i(\theta)}$ 
are in general position.   \end{itemize}

The class of $\theta$-transverse subgroups includes all discrete subgroups of rank one Lie groups, $\theta$-Anosov subgroups and their relative versions. Note also that every subgroup of a $\theta$-transverse subgroup is again $\theta$-transverse. The class of transverse subgroups is regarded as a generalization of all rank one discrete subgroups, while the class of Anosov subgroups is regarded as a generalization of rank one convex cocompact subgroups.

In the rest of the introduction, we assume that
$$\text{$\Ga$ is
a Zariski dense $\theta$-transverse subgroup of $G$. }$$
In order to introduce an appropriate substitute of $\Ga\ba G/M$ for a $\theta$-transverse subgroup $\Ga$, recall the Langlands decomposition $P_\theta=A_\theta S_\theta N_\theta$ where $A_\theta$ is the maximal split central torus, $S_\theta$ is an almost direct product of a semisimple algebraic subgroup and a compact central torus and $N_\theta$ is the unipotent radical of $P_\theta$. The diagonalizable subgroup $A_\theta$ acts on the quotient space $G/S_\theta$  by translations on the right.
The left translation action of $\Ga$ on $G/S_\theta$ is in general not properly discontinuous (cf. \cite{Benoist_Actions}, \cite{Koba}) unless $\theta=\Pi$, in which case $S_\theta$ is compact.
However the action of $\Ga$ is properly discontinuous on the following
closed $A_\theta$-invariant subspace (\cite[Thm. 9.1]{KOW_indicators}):
$$\tilde{\Omega}_\theta:=\{ [g] \in G/S_\theta: gP_\theta\in \La_\theta, g w_0 P_{\i(\theta)}\in \La_{\i(\theta)}\} $$ 
where $w_0$ is the longest Weyl element. Therefore the quotient space $$\Omega_\theta:=\Ga\ba \tilde \Omega_\theta$$
is a second countable locally compact Hausdorff space equipped with the right translation action of $A_\theta$ which is non-wandering. 
Denoting by $\La_\theta^{(2)}$ the set of all pairs
$(\xi, \eta)\in \La_\theta\times \La_{\i(\theta)}$ in general position, we have
(see \eqref{odd}):
$${\Omega}_\theta \simeq \Gamma\ba \left(\La_\theta^{(2)}\times \fa_\theta\right).$$ 
By a subspace flow on $\Omega_\theta$, we mean the action of the subgroup $A_{W}=\exp W$ for a non-zero linear subspace $W<\fa_\theta$. 

The main goal of this paper is to study the ergodic properties of the subspace flows on $\Omega_\theta$ with respect to Bowen-Margulis-Sullivan measures. The most essential case turns out to be the action of one-parameter subgroups of $A_\theta$ which we call directional flows. We first present the ergodic dichotomy for directional flows. 

\subsection*{Directional flows} Fixing a non-zero vector $u\in \fa_\theta^+$, we are interested in ergodic properties of
the action of the one-parameter subgroup $$A_{u}=\{a_{tu}=\exp t u:t\in \br\}$$
on the space $\Omega_\theta$.
We say that $\xi\in \La_\theta$ is a $u$-directional conical point if there exists $g\in G$ such that $\xi=gP_\theta$ and $[g]a_{t_iu}\in \Omega_\theta$
 belongs to a compact subset for some sequence $t_i\to +\infty$. We denote by $\La_\theta^u$ the set of all $u$-directional conical points, that is{\footnote{The set $\limsup_{t \to + \infty} [g]a_{tu}$ consists of all limits $\lim_{t_i\to +\infty} [g]a_{t_iu}$.}},
$$
\La_{\theta}^u:= \{gP_{\theta} \in\La_\theta : [g]\in \Omega_\theta,
\limsup_{t\to +\infty} [g] a_{tu} \ne \emptyset\}.$$
See Definition \ref{def.directionalconicalshadow} and Lemma \ref{defdir} for an equivalent
definition of $\La_\theta^u$ given in terms of shadows. It is clear from the definition that $\La_\theta^u$ is an important object in the study of the recurrence of $A_u$-orbits. Another important player in our ergodic dichotomy is the directional $\psi$-Poincar\'e series for a linear form $\psi\in \fa_\theta^*$. To define it, we set
$\mu_\theta:=p_\theta\circ \mu$ to be the $\fa_\theta$-valued Cartan projection where $p_{\theta} : \fa \to \fa_{\theta}$ is the unique projection, invariant under all Weyl elements fixing $\fa_{\theta}$ pointwise. The $u$-directional  $\psi$-Poincar\'e series is of the form
\be \label{eqn.directionalpoincaredef}
 \sum_{\ga \in \Ga_{u, R}} e^{-\psi(\mu_{\theta}(\ga))}
\ee where $\Ga_{{u},R}:=\{\ga\in\Ga : \norm{\mu_\theta(\ga)-\br {u}}<R\}$ for a Euclidean norm $\|\cdot\|$ on $\fa_{\theta}$ and $R>0$. 
In considering these objects, it is natural to restrict 
to those linear forms $\psi$ such that $\psi \circ \mu_{\theta} : \Ga \to [-\varepsilon, \infty)$ is a proper map for some $\e>0$, which we call $(\Ga, \theta)$-proper linear forms. 
A Borel probability measure $\nu$ on $\mathcal{F}_\theta$ is called a $(\Gamma, \psi)$-conformal measure if $$\frac{d \gamma_*\nu}{d\nu}(\xi)=e^{\psi(\beta_\xi^\theta(e,\gamma))} \quad \text{for all $\gamma \in \Gamma$ and $ \xi \in \mathcal{F}_\theta$} $$
 where ${\ga}_* \nu(D) = \nu(\ga^{-1}D)$ for any Borel subset $D\subset \F_\theta$ and $\beta_\xi^\theta$ denotes the $\fa_\theta$-valued Busemann map defined in \eqref{Bu}.
For a $(\Ga, \theta)$-proper $\psi \in \fa_{\theta}^*$, a $(\Ga, \psi)$-conformal measure  can exist only when $\psi\ge \psi_\Ga^\theta$ where $\psi_\Ga^\theta$ is the $\theta$-growth indicator of $\Ga$ \cite[Thm. 7.1]{KOW_indicators}.

Here is our main theorem for directional flows, relating the ergodicity of $A_u$, the divergence of the $u$-directional Poincar\'e series, and the
size of conformal measures on $u$-directional conical sets:
\begin{theorem}[Ergodic dichotomy for directional flows] \label{main}
    Let $\Ga$ be a Zariski dense $\theta$-transverse subgroup of $G$.
    Fix a non-zero vector $u\in \fa_\theta^+$ and a $(\Ga, \theta)$-proper linear form $\psi \in \fa_{\theta}^*$. 
    Suppose that there exists a pair $(\nu, \nu_{\i})$ of $(\Ga, \psi)$ and $(\Ga, \psi\circ\i)$-conformal measures on $\La_\theta$ and $\La_{\i(\theta)}$ respectively. Let
    $\m=\m(\nu, \nu_{ \i})$ denote the  associated Bowen-Margulis-Sullivan measure on $\Omega_\theta$ (see \eqref{bdef}).

    In each of the following complementary cases, claims (1)-(3) are equivalent to each other. If ${\mathsf m}$ is $u$-balanced (Definition \ref{def.ubalanced}), then (1)-(5) are all equivalent.

    The first case:
\begin{enumerate}
\item  $\max \left(\nu(\La_{\theta}^u),\nu_{\i} (\La_{\i(\theta)}^{\i({u})})\right) >0$;
\item $(\Omega_\theta, A_u,  \mathsf m)$ is completely conservative;
\item $(\Omega_\theta, A_u,  \mathsf m)$ is ergodic;
\item 
$\sum_{\ga\in\Ga_{{u},R}}e^{-\psi(\mu_\theta(\ga))}=\infty$ for some $R>0$;
\item $\nu(\La_{\theta}^{u})=1=\nu_{\i} (\La_{\i(\theta)}^{\i({u})})$.
\end{enumerate} 

    The second case:
    \begin{enumerate}
\item  $\nu(\La_{\theta}^{u})=0=\nu_{ \i} (\La_{\i(\theta)}^{\i({u})})$;
\item $(\Omega_\theta, A_u,  \mathsf m)$ is completely dissipative;
\item $(\Omega_\theta, A_u,  \mathsf m)$ is non-ergodic;
\item  $\sum_{\ga\in\Ga_{{u},R}}e^{-\psi(\mu_\theta(\ga))}<\infty$ for all $R>0$;
\item $\nu(\La_{\theta}^{u})=0=\nu_{ \i} (\La_{\i(\theta)}^{\i({u})})$.
\end{enumerate} 
    \end{theorem}

We remark that in the first case,
(1) is again equivalent to the condition $\max \left(\nu(\La_{\theta}^u),\nu_{ \i} (\La_{\i(\theta)}^{\i({u})})\right) =1$.
\begin{rmk} 
\begin{enumerate}
  \item When $\theta=\Pi$, or equivalently when $S_\theta$ is compact, 
  Theorem \ref{main} was obtained for a general Zariski dense discrete subgroup $\Ga<G$ by Burger-Landesberg-Lee-Oh \cite[Thm. 1.4]{BLLO}.
 
\item  The $u$-balanced condition is required only for the implication $(4)\Rightarrow (5)$ in the first case, which takes up the most significant portion of our proof. This condition can be verified for Anosov subgroups, as we will discuss later (Theorem \ref{main3}, Corollary \ref{cdi}).

\item By a recent work \cite[Prop. 10.1]{BCZZ2}, the existence of a $(\Ga, \psi)$-conformal measure on $\La_\theta$
implies that $\psi$ is $(\Ga, \theta)$-proper. Therefore the hypothesis that $\psi$ is $(\Ga, \theta)$-proper is unnecessary.
\item When $G$ is of rank one, this is precisely the classical Hopf-Tsuji-Sullivan dichotomy
 (see \cite{Sullivan1979density}, \cite{Hopf1971ergodic}, \cite{Tsuji1959potential}, \cite[Thm. 1.7]{Roblin2003ergodicite}, etc.).
\end{enumerate}
\end{rmk}

 Our proof of Theorem \ref{main} is a generalization of the approach of \cite{BLLO} to a general $\theta$. The main difficulties  arise from 
 the non-compactness of $S_\theta$ which we overcome using special properties of $\theta$-transverse subgroups such as regularity, anitipodality and the convergence group actions on the limit sets.

\subsection*{Subspace flows} We now turn to the ergodic dichotomy for general subspace flows.
Let $W$ be a non-zero linear subspace of $\fa_\theta$ and set $A_W=\{\exp w: w\in W\}$. 
The $W$-conical set of $\Ga$ is defined as
\be\label{laW}\La_\theta^W=\{gP_\theta\in \F_\theta: [g]\in \Omega_\theta, \; \limsup [g](A_W\cap A^+) \ne \emptyset \}; \ee 
see Definition \ref{def.subspaceconical} and Lemma \ref{lem.equivdefsubspace} for an equivalent
definition of $\La_\theta^W$ given in terms of shadows.
For $R>0$, we set
\be\label{gaW} \Ga_{W,R}=\{\gamma\in \Ga: \|\mu_\theta (\ga)-W\| < R\}.\ee 

\begin{theorem}[Ergodic dichotomy for subspace flows] 
 \label{thm.subspacedichotomy}
   Let $\psi, \nu, \nu_{\i}$ and $ \m$ be as in Theorem \ref{main}. Let $W<\fa_\theta$ be a non-zero linear subspace.
In the following complementary cases, claims (1)-(3) are equivalent to each other. If $\m$ is $W$-balanced 
    as in Definition \ref{Wb},
then (1)-(5) are all equivalent.

The first case:
\begin{enumerate}
\item  $\max \left(\nu(\La_{\theta}^W),\nu_{\i} (\La_{\i(\theta)}^{\i({W})})\right) >0$;
\item $(\Omega_\theta, A_W,  \mathsf m)$ is completely conservative;
\item $(\Omega_\theta, A_W,  \mathsf m)$ is ergodic;
\item 
$\sum_{\ga\in\Ga_{{W},R}}e^{-\psi(\mu_\theta(\ga))}=\infty$ for some $R>0$;
\item $\nu(\La_{\theta}^{W})=1=\nu_{ \i} (\La_{\i(\theta)}^{\i({W})})$.
\end{enumerate} 

The second case:
\begin{enumerate}
\item  $\nu(\La_{\theta}^{W})=0=\nu_{\i} (\La_{\i(\theta)}^{\i({W})})$;
\item $(\Omega_\theta, A_W,  \mathsf m)$ is  completely dissipative;
\item $(\Omega_\theta, A_W,  \mathsf m)$ is non-ergodic;
\item $\sum_{\ga\in\Ga_{{W},R}}e^{-\psi(\mu_\theta(\ga))}<\infty$ for all $R>0$;
\item  $\nu(\La_{\theta}^{W})=0=\nu_{ \i} (\La_{\i(\theta)}^{\i({W})})$.
\end{enumerate} 

\end{theorem}

\begin{rmk}\label{dif} When $W$ is all of $\fa_\theta$,
a similar dichotomy  was obtained in (\cite{LO_dichotomy}, \cite{CZZ_transverse}, \cite{KOW_indicators}). In this case, 
  the $W$-balanced condition of $\m$ is not required in our proof; see Remark \ref{fine}. Hence we give a different proof of the ergodicity criterion for the $A_\theta$-action \cite[Thm. 1.8]{KOW_indicators}.
\end{rmk}

A special feature of a transverse subgroup is that
for any $(\Ga, \theta)$-proper form $\psi$, the projection $\tilde \Omega_\theta \to \La_\theta^{(2)}\times \br$ given by
$(\xi, \eta, v)\mapsto (\xi, \eta, \psi(v))$ 
induces  a $\ker \psi$-bundle structure of $\Omega_\theta$
over the base space $\Omega_\psi:=\Ga\ba \La_\theta^{(2)}\times \br $ with the $\Ga$-action given in \eqref{linearformhopf}. In particular, we have a vector bundle isomorphism
$$\Omega_\theta \simeq \Omega_\psi \times \ker \psi.$$
The $\ker \psi$-bundle  $\Omega_\theta\to \Omega_{\psi}$ plays an important role in our proof of Theorem \ref{thm.subspacedichotomy}. Indeed, such a vector bundle $\Omega_{\theta} \to \Omega_{\psi}$ factors through the space $\Omega_{W^{\diamond}} := \Ga \ba \La_{\theta}^{(2)} \times \fa_{\theta}/(W \cap \ker \psi)$.
 Denoting by $\m'$ the Radon measure on $\Omega_{W^{\diamond}}$ such that $\m = \m' \otimes \Leb_{W \cap \ker \psi}$, 
the $W\cap \ker \psi$-bundle  $(\Omega_\theta, \m) \to (\Omega_{W^\diamond},\m')$ 
enables us to adapt arguments of Pozzetti-Sambarino \cite{PS_metric} in obtaining 
Theorem \ref{thm.subspacedichotomy} from the ergodic dichotomy of the directional flow $A_u$ on $\Omega_{W^{\diamond}}$ for any $u\in W$ such that $\psi(u)>0$.

\begin{Rmk} 
We remark that the Zariski dense hypothesis on $\Ga$ is used to ensure the non-arithmeticity of the Jordan projection of $\Ga$, which implies that the subgroup generated by $p_{\theta}(\la(\Ga))$ is dense in $\fa_{\theta}$ \cite{Benoist_properties2}. This is a key ingredient in the discussion of transitivity subgroup (Proposition \ref{prop.transitivity}). In fact, Theorem \ref{thm.subspacedichotomy} (and hence Theorem
    \ref{main}) works for a non-Zariski dense
    $\theta$-transverse subgroup $\Ga$ as well, provided that
     $p_\theta(\lambda(\Ga))$ generates a dense subgroup of $\fa_\theta$.
\end{Rmk}

\subsection*{The case of $\theta$-Anosov subgroups}
A finitely generated subgroup $\Gamma<G$ is called $\theta$-Anosov if
 there exist constants $C, C'>0$ such that for all $\alpha \in \theta$ and $\ga \in \Ga$, $$\alpha(\mu(\ga)) \ge C|\ga| - C'$$ where $|\ga |$ is the word length of $\ga$ with respect to a fixed finite generating set of $\Ga$ (\cite{Labourie2006anosov}, \cite{GW_anosov}, \cite{KLP_Anosov}, \cite{GGKW_gt}, \cite{BPS_anosov}). 
By the work of Kapovich-Leeb-Porti \cite{KLP_Anosov}, a $\theta$-transverse subgroup $\Ga<G$ is $\theta$-Anosov if $\La_\theta$ is equal to the $\theta$-conical set $\lac$ of $\Ga$ (see \eqref{cd} for definition).
If $\Ga$ is a $\theta$-Anosov subgroup, then for each unit vector $u$
in the interior of the limit cone $\L_\theta$, there exists a unique linear form $\psi_u
\in \fa_\theta^*$ tangent to the growth indicator $\psi_\Ga^\theta$ at $u$ and a unique $(\Ga, \psi_u)$-conformal measure $\nu_u$ on $\La_\theta$.
Moreover $u\mapsto \psi_u$ and $u \mapsto \nu_u$ give bijections among the directions in $\inte\L_\theta$, the space of tangent linear forms  to $\psi_\Ga^\theta$, and
the space of $\Ga$-conformal measures supported on $\La_\theta$ (\cite{LO_invariant}, \cite{sambarino2022report}, \cite{KOW_indicators}).   Let 
\be\label{mu} \mathsf m_u=\mathsf m (\nu_u, \nu_{\i(u)})\ee  denote the Bowen-Margulis-Sullivan measure on $\Omega_\theta$
associated with the pair $(\nu_u, \nu_{\i(u)})$. We deduce the following codimension dichotomy from Theorem \ref{thm.subspacedichotomy}:
\begin{theorem} [Codimension dichotomy] \label{main3} Let $\Ga<G$ be a Zariski dense $\theta$-Anosov subgroup. Let $u \in \inte \L_{\theta}$ and $W < \fa_\theta$ be a linear subspace containing $u$. 
The following are equivalent:
    \begin{enumerate}
        \item $\codim W \le  2$ (resp. $\codim W  \ge 3$);
        \item $\nu_u(\La_\theta^W)=1$ (resp.  $\nu_{u}(\La_{\theta}^W) = 0 $);
       \item $(\Omega_\theta, A_W,  \mathsf m_u)$ is ergodic and completely conservative (resp. non-ergodic and completely dissipative);
        \item  $\sum_{\ga\in\Ga_{W,R}}e^{-\psi_u(\mu_\theta(\ga))}=\infty$ for some $R>0$ (resp. $\sum_{\ga\in\Ga_{W,R}}e^{-\psi_u(\mu_\theta(\ga))} < \infty$ for all $R > 0$).
 \end{enumerate}
\end{theorem}
We can view this dichotomy phenomenon depending on $\codim W$ as consistent with a classical theorem about random walks in $\Z^d$ (or Brownian motions in $\R^d$), which are transient if and only if $d \ge 3$.
Since $\codim W=\#\theta -\dim W$, we have the following corollary:
\begin{corollary}[$\theta$-rank dichotomy]\label{cdi}
    Let $\Ga<G$ be a Zariski dense $\theta$-Anosov subgroup and let $u \in \inte \L_{\theta}$. Then
     $\# \theta\le 3$ if and only if   the directional flow $A_u$ on $(\Omega_\theta,  \mathsf m_u)$ is ergodic. 
\end{corollary}
For a $\theta$-Anosov subgroup $\Ga$,
$\Omega_{\psi_u}$ is a {\it compact} metric space (\cite{sambarino_orbit} and \cite[Appendix]{Carvajales}), and hence $\Omega_{W^\diamond}$ is a vector bundle over a {\it compact} space $\Omega_{\psi_u}$ with fiber $\br^{\codim W}$. Moreover, 
we have the following local mixing result due to Sambarino \cite[Thm. 2.5.2]{sambarino2022report} (see also \cite{CS_local}) that for any $f_1, f_2 \in C_c(\Omega_{W^{\diamond}})$,\footnote{The notation $C_c(X)$
for a topological space $X$ means the space of all continuous functions on $X$ with compact supports.}
\be\label{sa2}
\lim_{t \to \infty} t^{\frac{\codim W}{2}}\int_{\Omega_{W^{\diamond}}} f_1(x) f_2(xa_{tu}) d\m_u'(x) = \kappa_u \m_u'(f_1) \m_u'(f_2)\ee 
where $\kappa_u>0$ is a constant depending only on $u$.
In particular, $\mathsf m'_u$ satisfies the $u$-balanced hypothesis. The key part of our proof lies in establishing the inequalities (Propositions \ref{prop.poincareintegral} and \ref{prop.poincareintegral2}) that for all large enough $R>0$,\footnote{The notation $f(T)\ll g(T) $ means that there is a constant $c>0$ such that
$f(T)\le c g(T)$ for all $T$ in a given range.}
$$ \left( \int_1 ^T  t^{- \frac{\codim W}{2}} dt \right)^{1/2} \ll \sum_{ \substack{\ga \in \Ga_{W, R} \\ \psi_u(\mu_{\theta}(\ga)) \le \delta T}} e^{-\psi_u(\mu_{\theta}(\ga))} \ll \int_1^T t^{- \frac{\codim W}{2}} dt$$ for $T > 2$ where $\delta = \psi_u(u) > 0$.
Therefore, 
 $\sum_{\ga\in\Ga_{{W},R}}e^{-\psi_u(\mu_\theta(\ga))}=\infty$ if and only if $\codim W \le 2$.

\begin{rmk}
\begin{enumerate}
\item When $\theta=\Pi$ and $\dim W=1$, Theorem \ref{main3} and hence Corollary \ref{cdi} were
obtained in \cite{BLLO}; in this case, $\codim W\le 2$
translates into $\op{rank }G\le 3$.
 
 \item For a general $\theta$, when $\dim W=1$ and  $\codim W \ne 2$,  Sambarino proved the equivalence (1)-(3) of Theorem \ref{main3}  using a different approach \cite{sambarino2022report}; for instance, the directional Poincar\'e series was not discussed in his work. 
This was extended by
Pozzetti-Sambarino \cite{PS_metric} for subspace flows, but still under the hypothesis $\codim W \ne 2$,  using an approach similar to \cite{sambarino2022report}.
 Thus, Theorem \ref{main3} settles the open case of $\codim W = 2$.

\item 
We mention that in (\cite{KMO_HD}, \cite{KO_rigiditymeasure}, \cite{PS_metric}), the sizes of directional/subspace conical limit sets were used as a key input in estimating Hausdorff dimensions of certain subsets of the limit sets.
 \item Theorem \ref{main3} and Corollary \ref{cdi} are not true
 for a general $\theta$-transverse subgroup, e.g., there are discrete subgroups
 in a rank one Lie group which are not of divergence type. 
Consider a normal subgroup $\Ga$ of a non-elementary convex cocompact subgroup $\Ga_0$ of a rank one Lie group $G$ with $\G_0/\Ga\simeq \mathbb Z^d$ for $d\ge 0$.
  In this case, by a theorem of Rees \cite[Thm. 4.7]{Rees1981checking},
 $d\le 2$ if and only if $\Ga$ is of divergence type, i.e., its Poincar\'e series
diverges at the critical exponent of $\Ga$.
 Using  the local mixing result \cite[Thm. 4.7]{OP_local} which is of the form as \eqref{sa2} with $t^{\codim W/2}$
 replaced by $t^{d/2}$ and Corollary \ref{useful}, the approach of our paper gives an alternative proof of Rees' theorem.

\item Corollaries \ref{useful} and \ref{useful2} reduce the divergence of the $u$-directional Poincar\'e series to the local mixing rate for the $A_u$-flow.
For example, we expect the local mixing rate
of relatively $\theta$-Anosov subgroups to be same as that of Anosov subgroups,  which would then imply
Theorem \ref{main3} and Corollary \ref{cdi} for those subgroups.
 \end{enumerate}
\end{rmk}

\subsection*{Examples of ergodic actions on $\Gamma\ba G/S_\theta$} By the work of Gu\'eritaud-Guichard-Kassel-Wienhard \cite{GGKW_gt}, there are examples of
Borel Anosov subgroups which act properly discontinuously on $G/S_\theta$ for some $\theta\ne \Pi$
(\cite[Coro. 1.10, Coro. 1.11]{GGKW_gt}), in which case
our rank dichotomy theorem can be stated for the one-parameter subgroup action on $\Gamma\ba G/S_\theta$. We discuss one example where
 $G=\SL_d(\br)$, $d\ge 3$. For $2\le k\le d-2$, let $H_k$ be the block diagonal subgroup $\begin{pmatrix} I_{k} & \\ &  \SL_{d-k}(\br)\end{pmatrix} \simeq  \SL_{d-k}(\br)$ where $I_k$ denotes the $(k\times k)$-identity matrix. 
Set $\alpha_i(\text{diag}(v_1, \cdots, v_d))=v_i-v_{i+1}$ for $1\le i\le d-1$; so $\Pi=\{\alpha_i:1\le i\le d-1\}$ is the set of all simple roots for $G$.
For $\theta=\{\alpha_1, \cdots, \alpha_{k}\}$, we have
$S_\theta=H_k$. 
Let $\Ga<G$ be a $\Pi$-Anosov subgroup. Then $\Ga$ acts properly discontinuously on $\SL_d(\br)/\SL_{d-k}(\br)$ by \cite[Coro. 1.9, Coro. 1.10]{GGKW_gt}. Hence any Radon measure on $\Omega_\theta$ can be considered as a Radon measure on $\Ga\ba \SL_d(\br)/\SL_{d-k}(\br)$. Then Theorem \ref{main3} implies the following:
\begin{corollary} Let $\Ga<\SL_d(\br) $ be a Zariski dense $\Pi$-Anosov subgroup (e.g., Hitchin subgroups), $2\le k\le d-2$ and
$\theta=\{\alpha_1, \cdots, \alpha_{k}\}$.
Let
$u\in \inte \L_\theta$ and $\m_u$ be as in \eqref{mu}. 
    We have $k=2,3$ if and only if the $A_u$-action on $(\Gamma\ba \SL_d(\br)/\SL_{d-k}(\br), \m_u)$ is ergodic. 
\end{corollary}

\medskip

\noindent{\bf Acknowledgements.} 
We would like to thank Blayac-Canary-Zhu-Zimmer for telling us that an admissible metric can be used to prove Proposition \ref{prop.foliation}.

\section{Preliminaries}
Throughout the paper, let $G$ be a connected semisimple real algebraic group. In this section, we review some basic facts about the Lie group structure of $G$ and the notion of convergence of elements of $G$ to boundaries, following \cite[Sec. 2]{KOW_indicators} to which we refer for more details.

\medskip

 Let $P<G$ be a minimal parabolic subgroup with a fixed Langlands decomposition $P=MAN$ where $A$ is a maximal real split torus of $G$, $M$ is the maximal compact subgroup of $P$ commuting with $A$ and $N$ is the unipotent radical of $P$.
Let $\fg$ and $\fa$ respectively denote the Lie algebras of $G$
and $A$. Fix a positive Weyl chamber $\fa^+<\fa$ so that
$\log N$ consists of positive root subspaces and
set $A^+=\exp \fa^+$. We fix a maximal compact subgroup $K< G$ such that the Cartan decomposition $G=K A^+ K$ holds. We denote by 
$$\mu : G \to \fa^+$$ the Cartan projection defined by the condition $g\in K\exp \mu(g) K$ for $g \in G$. Let $X = G/K$ be the associated Riemannian symmetric space, and set $o = [K] \in X$.  Fix a $K$-invariant norm $\| \cdot \|$ on $\fg$ 
and a Riemanian metric $d$ on $X$, induced from the Killing form on $\fg$.
The Weyl group $\cal W$ is given by $N_K(A)/C_K(A)$, where $N_K(A)$ and $C_K(A)$ denote
the normalizer and the centralizer of $A$ in $K$ respectively. 
Oftentimes,we will identify $\cal W$ with the chosen set of representatives from $N_K(A)$, and hence treat $\cal W$ as a subset of $G$.

\begin{lemma} \cite[Lem. 4.6]{Benoist1997proprietes} \label{lem.cptcartan}
For any compact subset $Q \subset G$, there exists $C=C(Q)>0$ such that for all $g \in G$, $$\sup_{q_1, q_2\in Q} \| \mu(q_1gq_2) -\mu(g)\| \le C .$$  
\end{lemma}

Let $\Phi=\Phi(\fg, \fa)$ denote the set of all roots,  $\Phi^{+}\subset \Phi$
the set of all positive roots, and $\Pi \subset \Phi^+$  the set of all simple roots.
Fix an element $w_0\in K$ of order $2$ in the normalizer of $A$ representing the longest Weyl element so that $\op{Ad}_{w_0}\mathfrak a^+= -\mathfrak a^+$. 
The map \be\label{opp0} \i= -\op{Ad}_{w_0}:\fa\to \fa\ee  is called the opposition involution.
It induces an involution $\Phi\to \Phi$ preserving $\Pi$, for which we use the same notation $\i$, such that $\i (\alpha ) \circ  \op{Ad}_{w_0}  =-\alpha $ for all $\alpha\in \Phi$. We have 
$\mu(g^{-1})=\i (\mu(g))$ for all $g\in G$. 

Henceforth, we fix a  non-empty subset $\theta\subset \Pi$. 
Let $ P_\theta$ denote a standard parabolic subgroup of $G$ corresponding to $\theta$; that is,
$P_{\theta}$ is generated by $MA$ and all root subgroups $U_\alpha$,
where $\alpha$ ranges over all positive roots and any negative root which is a $\mathbb Z$-linear combination of $\Pi-\theta$. Hence $P_\Pi=P$. Let 
\begin{equation*}
\mathfrak{a}_\theta =\bigcap_{\alpha \in \Pi-\theta} \ker \alpha, \quad  \quad \fa_\theta^+  =\fa_\theta\cap \fa^+, \end{equation*}
\begin{equation*} A_{\theta}  = \exp \fa_{\theta}, \quad \text{and} \quad     A_{\theta}^+  = \exp \fa_{\theta}^+. \end{equation*}

Let $ p_\theta:\mathfrak{a}\to\mathfrak{a}_\theta$ denote  the projection invariant under $w\in \cal W$ fixing $\fa_\theta$ pointwise. We also write $$\mu_{\theta} := p_{\theta} \circ \mu :G\to \fa_\theta^+.$$

\begin{definition} For a discrete subgroup $\G<G$,
its $\theta$-limit cone $\L_{\theta}=\L_\theta(\Ga)$ is defined as the  the asymptotic cone of $\mu_{\theta}(\Ga)$ in $\fa_{\theta}$, that is, $u\in \L_\theta$ if and only if $u=\lim_{i \to \infty} t_i \mu_\theta(\ga_i)$ for some sequences $t_i\to 0$ and $\ga_i\in \Ga$. If $\Ga$ is Zariski dense, $\L_\theta$ is a convex cone  with non-empty interior by \cite{Benoist1997proprietes}. Setting $\L=\L_\Pi$, we have
$p_\theta(\L)=\L_\theta$.
\end{definition}
 We have the Levi-decomposition
 $P_\theta=L_\theta N_\theta$ where  $L_\theta$ is the centralizer of $A_{\theta}$ and $N_\theta=R_u(P_\theta)$ is the unipotent radical of $P_\theta$. We set $M_{\theta} = K \cap P_{\theta}\subset L_\theta$.
We may then write $L_{\theta} = A_{\theta}S_{\theta}$ where $S_{\theta}$ is an almost direct product of
 a connected semisimple real algebraic subgroup and a compact center.
 Letting $B_\theta=S_\theta \cap A$ and $B_\theta^+=\{b\in B_\theta: \alpha (\log b)\ge 0
 \text{ for all $\alpha\in \Pi-\theta$}\},$
 we have the Cartan decomposition of $S_\theta$: 
$$S_\theta = M_{\theta} B_\theta^+ M_{\theta}.$$
Note that  $A=A_\theta B_\theta$ and $ A^+\subset A_\theta^+ B_\theta^+.$
The space $\fa_\theta^*=\op{Hom}(\fa_\theta, \br)$ can be identified with the subspace of $\fa^*$ which is $p_\theta$-invariant: $\fa_\theta^*=\{\psi\in \fa^*: \psi\circ p_\theta=\psi\}$; so for $\theta_1\subset \theta_2$, we have $\fa_{\theta_1}^*\subset \fa_{\theta_2}^*$.

\subsection*{The $\theta$-boundary $\F_\theta$ and
convergence to $\F_\theta$} 
We set $$\F_\theta=G/P_{\theta} \quad\text{and}\quad \F=G/P.$$
Let $$ \pi_\theta:\F\to \F_\theta$$ denote
 the canonical projection map given by $gP\mapsto gP_\theta$, $g\in G$. 
 We set \be\label{xit} \xi_\theta=[P_\theta] \in \F_{\theta}.\ee 
By the Iwasawa decomposition $G=KP=KAN$, the subgroup $K$ acts transitively on $\F_\theta$, and hence
 $\F_\theta\simeq K/ M_\theta.$

We consider the following notion of convergence of a sequence in $G$ to an element of $\F_\theta$. For a sequence $g_i \in G$, we say $g_i \to \infty$ $\theta$-regularly if $\min_{\alpha\in \theta} \alpha(\mu(g_i)) \to \infty$ as $i \to \infty$.

\begin{definition} \label{fc} For a sequence $g_i\in G$  and $\xi\in \ft$, we write $\lim_{i\to \infty} g_i=\lim_{i \to \infty} g_i o =\xi$ and
 say $g_i $ (or $g_io \in X$) converges to $\xi$ if \begin{itemize}
     \item $g_i \to \infty$ $\theta$-regularly; and
\item $\lim_{i\to\infty} \kappa_{i}\xi_\theta= \xi$ in $\F_\theta$ for some $\kappa_{i}\in K$ such that $g_i\in \kappa_{i} A^+ K$.
 \end{itemize}         
\end{definition}

\begin{definition} \label{def.limitset} The $\theta$-limit set of a discrete subgroup $\Ga$ can be defined as follows:
$$\lat=\lat(\Ga):=\{\xi \in \F_{\theta} : \xi = \lim_{i \to \infty} {\ga}_i, \ {\ga}_i\in \Ga\}$$ where $\lim_{i \to \infty} \ga_i$ is defined as in Definition \ref{fc}.
If $\Ga$ is Zariski dense, this is the unique $\Ga$-minimal subset of $\F_{\theta}$ (\cite{Benoist1997proprietes}, \cite{Quint2002Mesures}). If we set $\La=\La_\Pi$,
then $ \pi_{\theta}(\La) = \La_{\theta}$.
\end{definition}

\begin{lemma}[{\cite[Lem. 2.6-7]{KOW_indicators}, see also \cite{LO_invariant} for $\theta = \Pi$}]  \label{lem.210inv}\label{bdd}  \label{lem.211inv} Let $g_i\in G$ be an infinite sequence.
   \begin{enumerate}
       \item If $g_i $ converges to $\xi \in \F_{\theta}$ and $p_i \in X$ is a
    bounded sequence, then $$\lim_{i \to \infty} g_i p_i = \xi.$$
  \item  If a sequence $a_i\to \infty$ in $A^+$ $\theta$-regularly, and $g_i\to g\in G$,  then for any $p \in X$, we have $$\lim_{i \to \infty} g_i a_i p = g \xi_\theta .$$
    \end{enumerate}
  
\end{lemma}

\subsection*{Jordan projections} A loxodromic element $g\in G$ is of the form $g=h a_g m h^{-1}$ for $h\in G$,
$a_g\in \inte A^+$ and $m\in M$; moreover $a_g\in \inte A^+$ is uniquely determined. We set 
\be\label{att} \lambda(g):=\log a_g \in \fa^+\quad\text{ and }\quad y_g:=hP\in \F,\ee 
 called the Jordan projection and
the attracting fixed point of $g$ respectively. 
\begin{theorem}\cite{Benoist_properties2} \label{thm.Benoistdense}
   For any Zariski dense subgroup $\Ga<G$, the subgroup generated by
   $\{\lambda(\ga):\ga \text{ is a loxodromic element of }\Ga\}$ is dense in $\fa$.
   \end{theorem}

\subsection*{Busemann maps}
The $\frak a$-valued Busemann map $\beta: \cal F\times G \times G \to\frak a $ is defined as follows: for $\xi\in \cal F$ and $g, h\in G$,
$$  \beta_\xi ( g, h):=\sigma (g^{-1}, \xi)-\sigma(h^{-1}, \xi)$$
where  $\sigma(g^{-1},\xi)\in \fa$ 
is the unique element such that $g^{-1}k \in K \exp (\sigma(g^{-1}, \xi)) N$ for any $k\in K$ with $\xi=kP$.
For $(\xi,g,h)\in \cal F_\theta\times G\times G$, we define
 \be\label{Bu} \beta_{\xi}^\theta (g, h): = 
p_\theta ( \beta_{\xi_0} (g, h)) \quad\text{for $\xi_0\in \pi_\theta^{-1}(\xi)$};\ee 
this is well-defined independent of the choice of $\xi_0$ 
\cite[Lem. 6.1]{Quint2002Mesures}. 
For $p, q \in X$ and $\xi \in \F_{\theta}$,  we set $\beta_{\xi}^{\theta}(p, q) := \beta_{\xi}^{\theta}(g, h)$ where $g, h \in G$ satisfies $go = p$ and $ho=q$. It is easy to check this is well-defined.

\subsection*{Points in general position} Let $\check{P}_{\theta}$ be the
standard parabolic subgroup of $G$ opposite to $P_\theta$ such that $P_\theta\cap \check{P}_{\theta}=L_\theta$. We have $\check{P}_{\theta} =w_0 P_{\i(\theta)}w_0^{-1}$ and hence 
$$\F_{\i(\theta)}=G/\check{P}_{\theta}.$$

For $g\in G$, we set
$$g_\theta^+ := gP_{\theta}\quad \text{ and }\quad  g_\theta^- := g w_0 P_{\i(\theta)};$$
as we fix $\theta$ in the entire paper, we write $g^{\pm}=g_\theta^{\pm}$ for simplicity when there is no room for confusion. Hence for the identity $e\in G$,
$(e^+, e^-)=(P_\theta, \check{P}_{\theta})=(\xi_\theta, w_0\xi_{\i(\theta)})$, where $\xi_\theta$ is as in \eqref{xit}.
The $G$-orbit
of $(e^+, e^-)$ is the unique open $G$-orbit
in $G/P_\theta\times G/\check{P}_{\theta}$ under the diagonal $G$-action. 
 We set
\be\label{f2} \F_{\theta}^{(2)}= \{(g_\theta^+, g_\theta^-): g\in G\}.\ee 
 Two elements
$\xi\in \F_\theta$ and $\eta\in \F_{\i(\theta)}$ are said to be in general position if $(\xi, \eta)\in \F_\theta^{(2)}$.
Since $\check{P}_{\theta}=L_\theta \check{N}_{\theta}$ where $\check{N}_\theta$ is the unipotent radical of $\check{P}_{\theta}$,  we have
\be \label{op} (g_\theta^+, e_\theta^-)\in 
\F_{\theta}^{(2)} \;\; \text{ if and only if }\;\; g\in \check{N}_\theta P_\theta .\ee

The following lemma will be useful:
\begin{lem} \cite[Coro. 2.5]{KOW_indicators} \label{cor.genweyl}
If $w\in \cal W$ is such that
$m w\in \check{N}_\theta P_\theta$ for some $m\in M_\theta$, then $w \in M_{\theta}$.    In particular, if $(w\xi_\theta, w_0\xi_{\i(\theta)})=(w_\theta^+, e_\theta^-) \in \F_\theta^{(2)}$, then $w\in M_\theta$.
\end{lem}

\subsection*{Gromov products}   
The map $g\mapsto (g^+, g^-)$ for $g\in G$ induces a homeomorphism $G/L_\theta
\simeq  \F_{\theta}^{(2)}$.
For $(\xi, \eta) \in \F_{\theta}^{(2)}$, we define the $\theta$-Gromov product as $$\cal G^{\theta}(\xi, \eta) = \beta_{\xi}^{\theta}(e, g) + \i (\beta_{\eta}^{\i(\theta)}(e, g))$$ where $g \in G$ satisfies $(g^+, g^-) = (\xi, \eta)$. This does not depend on the choice of $g$ \cite[Lem. 9.13]{KOW_indicators}.

Although the Gromov product is defined differently in \cite{BPS_anosov},
it coincides with ours (see \cite[Lem. 3.11, Rmk. 3.13]{LO_invariant}); hence we have:
\begin{prop} \cite[Prop. 8.12]{BPS_anosov} \label{Gromov} There exists $c>1$ and $c'>0$ such that for all $g\in G,$ $$c^{-1}\| \mathcal{G}^\theta(g^+,g^-) \| \le d(o, gL_{\theta} o) \le c \| \mathcal{G}^\theta (g^+,g^-) \| +c'.$$
\end{prop}

\section{Continuity of shadows}
The notion of shadows plays a crucial role in studying recurrence of diagonal flows.
In this section, we recall the definition of
$\theta$-shadows and prove some basic properties. 
In particular, we prove that shadows vary continuously on viewpoints, which are of independent interests.

\begin{figure}[h]
    \begin{tikzpicture}[scale=0.5, every node/.style={scale=0.8}]
    
        \draw (0, 0) circle(4);
        \draw[domain=279.3809950:329.7406027] plot ({-4.1666666 + 7.15752096*cos(\x)}, {7.0617983 + 7.15752096*sin(\x)});
        \draw[domain=279.3809950:329.7406027] plot ({-4.1666666 + 7.15752096*cos(\x)}, {-7.0617983 - 7.15752096*sin(\x)});

        \draw[fill=gray!10, opacity=0.5] (1, 0) circle(1.5911960891);
	\filldraw (1, 0) circle(2pt);
	\draw (1, -0.2) node[right] {$p$};
	
	\draw[thick, dotted, domain=273.8601104:296.54572789] plot ({-4.1666666 + 17.3299886*cos(\x)}, {17.2900673 + 17.3299886*sin(\x)});
        \filldraw (1.2206080, 0.8193113) circle(2pt);
        \draw (1.2206080, 0.6) node[right] {$gao$};
 	
	\filldraw (-3, 0) circle(2pt);
	\draw (-2.5, -0.5) node[below] {$q = go$};

	\draw[dashed, <->] (0.0596531, -1.28526581) -- (1, 0);
	\draw (0.5, -0.7) node[right] {$R$};

	\draw[thick, draw=blue, domain=-59.7406027:59.7406027] plot ({4*cos(\x)}, {4*sin(\x)});
        \draw[color=blue] (3, 3) node[right] {$O_R^{\theta}(q, p)$};

        \filldraw (3.5783118, 1.7876476) circle(2pt);
        \draw (3.5783118, 1.7876476) node[right] {$g^+$};
	
    \end{tikzpicture} 
    \begin{tikzpicture}[scale=0.5, every node/.style={scale=0.8}]
    
	\draw (0, 0) circle(4);
        \draw[domain=270:329.0165576] plot ({-4 + 7.0607593*cos(\x)}, {7.0617983 + 7.0617983*sin(\x)});
        \draw[domain=270:329.0165576] plot ({-4 + 7.0607593*cos(\x)}, {-7.0617983 - 7.0617983*sin(\x)});

        \draw[fill=gray!10, opacity=0.5] (1, 0) circle(1.5911960891);
	\filldraw (1, 0) circle(2pt);
	\draw (1, -0.2) node[right] {$p$};
	
	\draw[thick, dotted, domain=270:296.5197118] plot ({-4 + 16.9743042*cos(\x)}, {16.9743042 + 16.9743042*sin(\x)});
        \filldraw (1.2206080, 0.8193113) circle(2pt);
        \draw (1.2206080, 0.6) node[right] {$go$};

	\filldraw (-4, 0) circle(2pt);
	\draw (-4, 0) node[left] {$\eta = g^-$};

	\draw[dashed, <->] (0.0596531, -1.28526581) -- (1, 0);
	\draw (0.5, -0.7) node[right] {$R$};

	\draw[thick, draw=blue, domain=-59.0165576:59.0165576] plot ({4*cos(\x)}, {4*sin(\x)});
        \draw[color=blue] (3, 3) node[right] {$O_R^{\theta}(\eta, p)$};

        \filldraw (3.5783118, 1.7876476) circle(2pt);
        \draw (3.5783118, 1.7876476) node[right] {$g^+$};
	
    \end{tikzpicture}
    \caption{Shadows} \label{fig:shadow}
\end{figure}

For $p \in X$ and $R > 0$, let $B(p, R) $ denote the metric ball $ \{ x \in X : d(x, p) < R\}$. For $q \in X$, the $\theta$-shadow $O_R^{\theta}(q, p) \subset \F_{\theta}$ of $B(p, R)$ viewed from $q$ is defined as \be \label{eqn.defshadow}
\begin{aligned}
    O_R^{\theta}(q, p)   = \{ gP_{\theta} \in \F_{\theta} : g \in G, \ go = q, \ gA^+o \cap B(p, R) \neq \emptyset \} 
    \end{aligned}
    \ee
   We  also define the $\theta$-shadow $O_R^{\theta}(\eta, p) \subset \F_{\theta}$ viewed from $\eta \in \F_{\i(\theta)}$ as follows: $$O_R^{\theta}(\eta, p) = \{g P_{\theta} \in \F_{\theta} : g \in G, \ gw_0P_{\i(\theta)} = \eta, \ go \in B(p, R) \}.$$ 

For any $\tilde{\eta} \in \pi_{\i(\theta)}^{-1}(\eta)$, we have \be \label{eqn.shadowisproj}
    O_R^{\theta}(q, p) = \pi_{\theta}(O_R^{\Pi}(q, p)) \quad \text{and} \quad O_R^{\theta}(\eta, p) = \pi_{\theta}(O_R^{\Pi}(\tilde{\eta}, p)).
    \ee Note that for all $g\in G$ 
    and $\eta\in X\cup \F_{\i(\theta)}$,  
    \be \label{go} g O_R^{\theta}(\eta, p)= O_R^{\theta}(g\eta, gp).\ee
    We define the $\fa_{\theta}$-valued distance $\underline{a}_{\theta} : X \times X \to \fa_{\theta}$ by $$\underline{a}_{\theta}(q, p) := \mu_{\theta}(g^{-1}h)$$ where $q = go$ and $p = ho$ for $g, h \in G$. The following was shown for $\theta = \Pi$ in \cite[Lem. 5.7]{LO_invariant} which directly implies the statement for general $\theta$ by \eqref{eqn.shadowisproj}.

    \begin{lemma} \label{lem.buseandcartan}
        There exists $\kappa > 0$ such that for any $q, p \in X$ and $R>0$, we have $$\sup_{\xi \in O^\theta_R(q, p)}  \| \beta_{\xi}^{\theta}(q, p) - \underline{a}_{\theta}(q, p) \| \le \kappa R.$$
    \end{lemma}

\begin{lem} \label{lem.shadowtriangle}
    For any compact subset $Q \subset G$ and $R > 0$, we have that
    for any $h\in G$ and $g\in Q$,
    $$O_R^{\theta}( g o, h o) \subset O_{R+ D_Q}^{\theta}(o, h o) \quad \text{and} \quad O_R^{\theta}(h g o, o) \subset O_{R+ D_Q}^{\theta}(h o, o) $$
    where $D_Q:=\max_{g\in Q} d(go, o)$.
\end{lem}
\begin{proof} Note that  $d(a o, p a o) \le
d(o,p o) $ for all $a\in A^+$ and $p\in P$.
    Let $g \in Q$ and $\xi \in O_R^{\theta}(g o, h o)$. Then for some $k \in K$ and $a \in A^+$, we have $\xi = g k P_{\theta}$ and $d(g k a o, h o) < R.$ Write  $g k = \ell p \in KP$ for $\ell\in K$ and $p\in P$ by the Iwasawa decomposition $G=KP$.
   Since $d(\ell a o, \ell p a o)\le D_Q$, we have
    $d( \ell a o, h o) \le d(\ell a o, \ell p a o) +d(gkao , ho) < D_Q + R $. Therefore $\xi\in 
     O_{R + D_Q}^{\theta}(o, h o)$, proving the first claim. The second claim 
 follows from the first by \eqref{go}.
\end{proof}

\begin{lemma} \label{lem.continuousshadow}
    Let $p \in X$, $\eta \in \F_{\i(\theta)}$ and $r > 0$. If a sequence $\eta_i \in \F_{\i(\theta)}$ converges to $\eta \in \F_{\i(\theta)}$, then for any $0 < \varepsilon < r$, we have \be\label{oo} O_{r - \varepsilon}^{\theta}(\eta_i, p) \subset O_r^{\theta}(\eta, p) \subset O_{r + \varepsilon}^{\theta}(\eta_i, p) \quad \text{for all large } i\ge 1.\ee 
\end{lemma}
\begin{proof} 
We first prove the second inclusion.
  Let $g\in G$ be such that
  $g^+  \in O_r^{\theta}(\eta, p)$,  $g^- = \eta$ and $d(g o, p) < r$. 
 Since $\eta_i\to \eta$, we have $(g^+, \eta_i)\in \F_\theta^{(2)}$ for all large $i \ge 1$, and hence
 $(g^+, \eta_i)=(h_i^+, h_i^-)$ for some $h_i\in G$. In particular, $g =h_i q_i n_i$
 for $q_in_i \in L_\theta N_\theta =P_\theta$. By replacing  $h_i$ with $h_iq_i$,
 we may assume that $g = h_i n_i$. Since  $h_i^- \to  g^-$, we have $n_i^-\to e^-$, and hence $n_i\to e$ as $i\to \infty$. Therefore for all $i \ge 1$ large enough so that
 $d(n_io, o)\le \e$, we have $d(h_io, p) \le d(h_io, h_in_i o)+ d(g o, p)
 < \e +r$, and hence $g^+= h_i^+\in  O_{r + \varepsilon}^{\theta}(\eta_i, p) $.

To prove the first inclusion,  fix $k_i \in \stab_G(p)$ such that $k_i \eta_i = \eta$ for each $i \ge 1$. After passing to a subsequence, we may assume that the sequence $k_i$  converges to some $k\in \stab_G(p)$ as $i \to \infty$. Since $\eta_i \to \eta$, we have $k \eta = \eta$. In particular, the sequence $k_i \eta$ converges to $\eta$. Applying the second inclusion of \eqref{oo} to a sequence $k_i \eta$, we have
$$O_{r - \varepsilon}^{\theta}(k_i \eta_i, p) = O_{r - \varepsilon}^{\theta}(\eta, p) \subset O_{r}^{\theta}(k_i \eta, p) \quad \text{for all large } i \ge 1.$$
Since $k_i \in \stab_G(p)$, it follows from \eqref{go} that $O_{r - \varepsilon}^{\theta}(k_i \eta_i, p) = k_i O_{r - \varepsilon}^{\theta}(\eta_i, p) $ and $ O_{r}^{\theta}(k_i \eta, p) = k_i  O_{r}^{\theta}( \eta, p)$. This proves the first inclusion.
\end{proof}

We  show that for a fixed $p\in X$ and $\eta\in \F_{\i(\theta)}$,
shadows $O_r^{\theta}(\eta, p)$  vary continuously on a small neighborhood of $\eta$ in $G\cup \F_{\i(\theta)}$ (see \cite[Lem. 5.6]{LO_invariant} for $\theta = \Pi$):

\begin{prop}[Continuity of shadows on viewpoints] \label{lem.approxshadows}
    Let $p \in X$, $\eta \in \F_{\i(\theta)}$ and $r>0$.
    If a sequence $q_i \in X$ converges to $ \eta $ as $i \to \infty$, then for any $0<\e<r$, we have \be \label{eqn.approxshadows}
    O_{r - \varepsilon}^{\theta}(q_i, p) \subset O_r^{\theta}(\eta, p) \subset O_{r + \varepsilon}^{\theta}(q_i, p) \quad \text{for all large } i\ge 1.
    \ee
\end{prop}

\begin{proof}
We first prove the second inclusion which requires more delicate arguments.
By \eqref{go} and the fact that $K$ acts transtively on $\F_{\i(\theta)}$,
we may assume without loss of generality that $\eta = P_{\i(\theta)}=w_0^-$ and $p = o$.
  Write  $q_i = k_i' a_i o$ with $k_i' \in K $ and $a_i\in A^+$ using Cartan decomposition. Since $q_i \to w_0^-$, we have
  $k_i' w_0^- \to w_0^-$ and $a_i \to \infty$ $\i(\theta)$-regularly. 
    
    By Lemma \ref{lem.continuousshadow}, we may assume $k_i'=e$ without loss of generality.
By \eqref{eqn.shadowisproj}, the claim follows if we replace $\theta$ by any subset containing $\theta$. Therefore we may assume without loss of generality that 
$\alpha(\log a_i)$ is uniformly bounded for all $\alpha \in \Pi - \i(\theta)$. 

    Let $\xi \in O_{r}^{\theta}(P_{\i(\theta)}, o)$, i.e., $\xi = hP_{\theta}$ for some $h \in G$ such that $d(h o, o) < r$ and $hw_0 P_{\i(\theta)} = P_{\i(\theta)}$. Since $P_{\i(\theta)} =P  M_{\i(\theta)}$ and $w_0^{-1} M_{i(\theta)} w_0=M_\theta$,
   we may assume $hw_0\in P$ by replacing $h$ with $hm$ for some $m\in M_\theta$.
   We need to show that for some $p_i \in P_\theta$ such that $hp_io=a_io$,
   $d(p_i A^+o,  o) < \e $; this then implies $d(h p_iA^+o, o)< r+\e$, and hence $\xi\in
   O_{r+\e}^{\theta}(a_i o, o) $.

   We start by writing $$a_i^{-1} h = {k}_i \tilde{a}_i {n}_i \in KAN ,\;\; \tilde a_i = c_i d_i \in A_\theta B_\theta \text{ and } n_i=u_iv_i\in  (S_\theta \cap N) N_\theta .$$
   As $hw_0\in P$ and $a_i \in A^+$, the sequence $a_i^{-1} h w_0 a_i$ is bounded.    
Since $$a_i^{-1} h w_0 a_i = (k_i w_0)(w_0^{-1} \tilde{a}_i w_0 a_i) (a_i^{-1} w_0^{-1} n_i w_0 a_i) \in KAN^+,$$ it follows that
both sequences $w_0^{-1} \tilde{a}_i w_0 a_i$ and $a_i^{-1}w_0^{-1}n_i w_0 a_i$ are bounded. 

Since $w_0^{-1}n_iw_0 = (w_0^{-1}u_i w_0)(w_0^{-1}v_iw_0) \in S_{\i(\theta)} N_{\i(\theta)}^+$ and $a_i \in A^+$ with $a_i \to \infty$ $\i(\theta)$-regularly, the boundedness of $a_i^{-1}w_0^{-1}n_i w_0 a_i$ implies that $v_i \to e$ as $i \to \infty$ and $u_i$ is bounded.
 On the other hand, the boundedness of $w_0^{-1} \tilde{a}_i w_0 a_i$ implies that $\tilde{a}_i \in w_0 a_i^{-1} w_0^{-1}A_C$ for some $C > 0$.
As $a_i \to \infty$ $\i(\theta)$-regularly, it follows
that $c_i \in A_{\theta}^+$ and $c_i \to \infty$ $\theta$-regularly. Moreover,
since $\max_{\alpha \in \Pi - \i(\theta)} \alpha(\log a_i)$ is uniformly bounded, the sequence $d_i$ is bounded.

As $d_iu_i \in S_{\theta}$, we may write its Cartan decomposition $d_iu_i = m_i b_i m_i' \in M_{\theta} B_{\theta}^+ M_{\theta}$. Since $c_i\to \infty$ $\theta$-regularly and
$d_iu_i$, and hence $b_i\in B_\theta^+$, is uniformly bounded, 
 we have $c_i b_i\in A^+$ for all large $i\ge 1$.
Set  $p_i= (m_i^{-1} \tilde a_i n_i)^{-1} \in P_\theta$. Recalling $a_i^{-1} h=k_i\tilde a_i n_i$,
we have $hp_i o = h n_i^{-1} \tilde{a}_i^{-1} o=a_i o$.
Moreover, we have 
$$p_i (c_i b_i) o = n_i^{-1} \tilde{a}_i^{-1} m_i c_i b_i m_i' o= n_i^{-1}\tilde{a}_i^{-1}c_i d_i u_i o= v_i^{-1} o$$
using the commutativity of $M_\theta$ and $A_\theta$ as well as the identity $m_ib_im_i'=d_i u_i$.
Since $v_i \to e$,
we have $d(p_i (c_ib_i)o, o)\to 0$.  This proves the second inclusion.

\medskip 
 We now prove the first inclusion.
    Similarly, as in the previous case, we may assume that $q_i = a_i o$ for $a_i \in A^+$ and $\eta = P_{\i(\theta)}$.  Let $\eta_i \in O_{r - \varepsilon}^{\theta}(a_i o, o)$, i.e.,  $\eta_i = a_i k_i P_{\theta}$ and $d(a_ik_ib_i o, o) < r - \varepsilon$ for some $k_i \in K$ and $b_i \in A^+$.
   Set $g_i = a_i k_i b_i$, which is a bounded sequence. 
 We will find  $n_i\in N_\theta$ such that
    $(g_in_i)^-= P_{\i(\theta)}$ and $d(g_in_io, o) < r$ from which $\eta_i \in O_r^{\theta}(\eta, o)$ follows.

We may assume that $g_i$ converges to some $g \in G$. Since $a_i\to \infty$ $\i(\theta)$-regularly,  the boundedness of $g_i = a_i k_i b_i$ together with Lemma \ref{lem.cptcartan} implies that $b_i \to \infty$ $\theta$-regularly. Since $a_i k_i \to P_{\i(\theta)}$ and $a_i k_i = g_i w_0 (w_0^{-1}b_i^{-1}w_0)w_0^{-1} \to gw_0 P_{\i(\theta)}$ as $i \to \infty$ by Lemma \ref{lem.210inv}, we have $$g w_0 P_{\i(\theta)} = P_{\i(\theta)}.$$
    
    On the other hand, as $i \to \infty$, we have $$g_i(P_{\theta}, w_0 P_{\i(\theta)}) \to g(P_{\theta}, w_0 P_{\i(\theta)}) = (gP_{\theta}, P_{\i(\theta)}).$$
    Hence for all large $i\ge 1$, $g_iP_{\theta}$ is in general position with $P_{\i(\theta)}$ and thus we have a sequence $h_i \in G$ such that $$(g_i P_{\theta}, P_{\i(\theta)}) = h_i(P_{\theta}, w_0 P_{\i(\theta)}).$$ 
    As $g_i P_{\theta} = h_i P_{\theta}$, we write $h_i = g_i n_i \ell_i$ for some $n_i \in N_{\theta}$ and $\ell_i \in L_{\theta}$.
Note that  $(g_in_i)^-=h_i^-=P_{\i(\theta)}$. We now claim that $n_i\to e$, from which $d(g_in_i o, o) \le d(g_i n_io, g_i o)+ d(g_io, o) < r$ follows for all large $i$.
    
    Since $h_i(P_{\theta}, w_0 P_{\i(\theta)}) = (g_iP_{\theta}, P_{\i(\theta)}) \to (gP_{\theta}, P_{\i(\theta)}) = g(P_{\theta}, w_0 P_{\i(\theta)})$, we have $h_i L_{\theta}=g_i n_iL_\theta \to g L_{\theta}$. Since $g_i \to g$ and $n_i\in N_\theta$, we have $n_i \to e$ as $i\to \infty$. This finishes the proof. 
    \end{proof}

\begin{lem}\label{shad_con}
Let $S>0$. For any sequence $g_i\to \infty$ in $G$ $\theta$-regularly,
the product $O_S^\theta(o, g_io)\times O_S^{\i(\theta)}(g_io, o)$ is precompact
in $\F_\theta^{(2)}$ for all sufficiently large
$i\ge 1$. 
\end{lem}

\begin{proof} Consider an infinite sequence $(\xi_i, \eta_i) \in O_S^{\theta}(o, g_i o) \times O_S^{\i(\theta)}(g_i o, o)$. By the $\theta$-regularity of $g_i\to \infty$, we have
$g_i o \to \xi$ as $i \to \infty$ for some $\xi \in \F_{\theta}$, after passing to a subsequence. For each $i$, we write $\xi_i = k_i P_{\theta}$ for $k_i \in K$ such that $d(k_i a_i o, g_i o) < S$ for some $a_i \in A^+$. In particular, $a_i \to \infty$ $\theta$-regularly. After passing to a subsequence, we may assume that $k_i \to k \in K$ so that $k_i a_i o \to kP_{\theta}$ as $i \to \infty$. On the other hand,  the boundedness of $d(k_i a_i o, g_i o) < S$ implies that $k_ia_i o \to \xi$ by Lemma \ref{lem.210inv}. Therefore, $\xi = kP_{\theta} = \lim_i \xi_i$.
By passing to a subsequence, we may assume that $\eta_i \to \eta$ for some $\eta \in \F_{\i(\theta)}$. Since $g_io\to \xi$,
and $\eta_i\in  O_S^{\i(\theta)}(g_i o, o)$, it follows from  Proposition \ref{lem.approxshadows} that $\eta \in
 O_{2S}^{\i(\theta)}(\xi, o)$. In particular, $(\xi, \eta) \in \F_{\theta}^{(2)}$. 
\end{proof}

\section{Growth indicators and conformal measures on $\F_\theta$}

In this section, we review the notion of $\theta$-growth indicators and discuss their influence 
on conformal measures on the $\theta$-boundary.

\medskip

Let $\Ga<G$ be a Zariski dense discrete subgroup.
We say that $\Ga $ is $\theta$-discrete if the restriction $\mu_{\theta}|_{\Ga} : \Ga \to \fa_{\theta}^+$ is a proper map. Observe that $\Ga$ is $\theta$-discrete if and only if the counting measure on $\mu_\theta(\Gamma)$ weighted with multiplicity is locally finite
i.e., finite on compact subsets.
 Following Quint's notion of growth indicators \cite{Quint2002divergence}, we have
 introduced the following in \cite{KOW_indicators}:
 \begin{definition}[$\theta$-growth indicator]  For a $\theta$-discrete subgroup $\Ga<G$,
we define the $\theta$-growth indicator $\psi_\Ga^{\theta}:\fa_\theta\to [-\infty, \infty] $ as follows: if $u \in \fa_\theta$ is non-zero,
\be\label{gi2} \psi_\Ga^{\theta}(u)=\|u\| \inf_{u\in \cal C}
\tau^\theta_{\mathcal C}\ee 
where $\cal C\subset \fa_\theta$ ranges over all open cones containing $u$, and $\psi_{\Ga}^{\theta}(0) = 0$.
Here $-\infty\le \tau^{\theta}_{\cal C}\le \infty$ is the abscissa of convergence of $s \mapsto \sum_{\ga\in \Ga, \mu_\theta(\ga)\in \mathcal C} e^{-s\|\mu_\theta(\ga)\|}$.
\end{definition}
We showed (\cite[Thm. 3.3]{KOW_indicators}):
\begin{itemize}
    \item $\psi_\Ga^{\theta} <\infty$;
\item $\psi_\Ga^\theta$ is  upper semi-continuous and concave; 
\item $\L_\theta=\{\psi_\Ga^{\theta}\ge 0\}=\{\psi_\Ga^{\theta}>-\infty\}$, and
    $\psi_\Ga^{\theta}>0$ on $\inte \L_\theta$ .
\end{itemize}

Let $\psi\in \fa_\theta^*$. Recall that a $(\Ga, \psi)$-conformal measure  $\nu$ 
is a Borel probability measure on $\F_\theta$ such that $$\frac{d \gamma_*\nu}{d\nu}(\xi)=e^{\psi(\beta_\xi^\theta(e,\gamma))} \quad \text{for all $\gamma \in \Gamma$ and $ \xi \in \mathcal{F}_\theta$}. $$
 A linear form $\psi \in \fa_{\theta}^*$ is said to be tangent to $\psi_{\Ga}^{\theta}$ at $v \in \fa_{\theta} - \{0\}$ if $\psi \ge \psi_{\Ga}^{\theta}$ and $\psi(v) = \psi_{\Ga}^{\theta}(v)$. 
\begin{prop}[{\cite[Thm. 8.4]{Quint2002Mesures}, \cite[Prop. 5.8]{KOW_indicators}}] \label{qc}
 For any $\psi \in \fa_{\theta}^*$ which is tangent to $\psi_{\Ga}^{\theta}$ at an interior direction of  $\fa_{\theta}^+$, there exists a $(\Ga, \psi)$-conformal measure supported on $\La_\theta$.
\end{prop}

Recall that
$\Ga$ is called {\it $\theta$-transverse}, if 
\begin{itemize}
    \item   $\Ga$ is {\it $\theta$-regular}, i.e.,
$ \liminf_{\ga\in \Ga} \alpha(\mu({\ga}))=\infty $ for all $\alpha\in \theta$; and
\item $\Ga$ is  {\it $\theta$-antipodal}, i.e.,
any distinct $\xi, \eta\in \La_{\theta\cup \i(\theta)}$ 
are in general position.   \end{itemize}

Recall also that $\psi\in \fa_\theta^*$ is
 $(\Ga, \theta)$-proper if  $\psi \circ \mu_{\theta}|_{\Ga}$ is a proper map into $[-\e, \infty)$
 for some $\e>0$.
\begin{theorem}[{\cite[Thm. 8.1]{Quint2002Mesures} for $\theta=\Pi$,
\cite[Thm. 7.1]{KOW_indicators} in general}] \label{tgg}
    Let $\Ga$ be a Zariski dense $\theta$-transverse subgroup of $G$. If there exists a $(\Ga, \psi)$-conformal measure $\nu$ on $\F_{\theta}$ for a $(\Ga, \theta)$-proper $\psi \in \fa_{\theta}^*$, then 
    $$
    \psi\ge \psi_{\Ga}^{\theta}.
    $$
     Moreover, if $\sum_{\ga \in \Ga} e^{-\psi(\mu_{\theta}(\ga))} = \infty$ in addition, then the abscissa of convergence of $s \mapsto \sum_{\ga \in \Ga} e^{-s \psi(\mu_{\theta}(\ga))}$ is equal to one.
\end{theorem}

\subsection*{Shadow lemma}
The following is an analog of Sullivan's shadow lemma for $\Ga$-conformal measures on $\F_{\theta}$ which was proved in \cite[Lem. 7.2]{KOW_indicators}.

\begin{lemma}[Shadow lemma] \label{lem.shadowlemma}
Let $\nu$ be a $(\Ga, \psi)$-conformal measure on $\F_{\theta}$. We have the following:
\begin{enumerate}
    \item for some $R = R(\nu) > 0$, we have $c := \inf_{\ga \in \Ga} \nu(O_R^{\theta}(\ga o, o)) > 0$; and
    \item for all $r \ge R$ and for all $\ga \in \Ga$, \be\label{ineq} c e^{-\|\psi\|\kappa r} e^{-\psi(\mu_{\theta}(\ga))} \le \nu(O_r^{\theta}(o, \ga o)) \le e^{\|\psi\|\kappa r} e^{-\psi(\mu_{\theta}(\ga))}\ee where $\kappa > 0$ is the constant given in Lemma \ref{lem.buseandcartan}.
\end{enumerate}
If $\Ga$ is a $\theta$-transverse subgroup with $\#\La_\theta \ge 3$
(which is not necessarily Zariski dense), then \eqref{ineq} holds for any $(\Ga, \psi)$-conformal measure supported on $\La_\theta$.
\end{lemma}

\section{Directional recurrence for transverse subgroups}\label{sec:rec}

In this section, we recall the flow space $\Omega_\theta$ for each $\theta$-transverse subgroup
$\Ga$. We then define the directional conical set of $\Ga$ and
give a characterization in terms of the recurrence set for directional flows on $\Omega_\theta$.
\medskip

We suppose that $\Ga$ is a Zariski dense $\theta$-transverse subgroup unless mentioned otherwise.
The $\Ga$-action on $G/S_\theta$ by left translations is not properly discontinuous in general, but there is a closed subspace  $\tilde \Omega_\theta \subset G/S_\theta$ on which $\Ga$ acts properly discontinuously. 

We first describe a parametrization of $G/S_\theta$ as 
$\F_{\theta}^{(2)} \times \fa_{\theta}$, which can be thought as a generalized Hopf-parametrization.
For $g\in G$, let
$$[g] := (g^+, g^-, \beta_{g^+}^{\theta}(e, g))\in \F_{\theta}^{(2)} \times \fa_{\theta}.$$
Consider the action of $G$ on the space $\F_{\theta}^{(2)} \times \fa_{\theta}$ by
\be\label{hopf} g . (\xi, \eta, b) = (g
\xi, g \eta, b + \beta_{\xi}^{\theta}(g^{-1}, e)) \ee 
where $g\in G$ and  $(\xi, \eta, b)\in  \F_{\theta}^{(2)} \times \fa_{\theta}$. Then the map $G\to\F_{\theta}^{(2)} \times \fa_{\theta}$ given by $g\mapsto [g]$ 
factors through $G/S_\theta$ and  defines a $G$-equivariant 
 homeomorphism $$G/S_\theta\simeq \F_{\theta}^{(2)} \times \fa_{\theta} .$$
 The subgroup $A_\theta$ acts on $G/S_\theta$ on the right
 by $[g]a:=[ga]$ for $g\in G$ and $a\in A_\theta$; this is well-defined as $A_\theta$ commutes with $S_\theta$.
The  corresponding $A_\theta$-action on $\F_{\theta}^{(2)} \times \fa_{\theta}$
is given by $$(\xi, \eta, b).a = (\xi, \eta, b + \log a)$$
for $a \in A_{\theta}$ and $(\xi, \eta, b)\in  \F_{\theta}^{(2)} \times \fa_{\theta}$.
 For $\theta = \Pi$, this homeomorphism is called the Hopf parametrization of $G/M$.

Set $\La_\theta^{(2)}:=(\La_\theta\times \La_{\i(\theta)})\cap \F_\theta^{(2)}$, and define
\be\label{odd} \tilde \Omega_\theta=\La_\theta^{(2)}\times \fa_\theta\ee 
which is a closed left $\Ga$-invariant and right $A_\theta$-invariant subspace of $\F_\theta^{(2)}\times \fa_\theta$.

\begin{theorem}\cite[Thm. 9.1]{KOW_indicators} \label{thm.propdisckow}
    If $\Ga$ is $\theta$-transverse, then $\Ga$ acts properly discontinuously on $\tilde \Omega_\theta$ and hence
    $$\Omega_\theta:=\Ga\ba \tilde \Omega_\theta$$ is a second countable locally compact Hausdorff space.
\end{theorem}

By \cite{Benoist1997proprietes}, the set
$\{(y_\ga, y_{\ga^{-1}})\in \La^{(2)}:\ga\in \Ga \text{ loxodromic}\}$ is dense in $\La^{(2)}$ (see \eqref{att} for the notation $y_\ga$).  Hence the projection
$\{( \pi_\theta(y_\ga) , \pi_{\i(\theta)} (y_{\ga^{-1}}) \in \La_\theta^{(2)}:\ga\in \Ga \text{ loxodromic}\}$ is dense in $\La_\theta^{(2)}$. This implies that $\Omega_\theta$ is a non-wandering set for $A_\theta$, that is, for any open subset $O\subset \Omega_\theta$,
the intersection $O\cap O a_i$ is non-empty for some  sequence $a_i\in A_\theta$ going to $\infty$.

Fix $u\in  \fa_\theta^+ - \{0\}$ and set
$$a_{tu}=\exp tu \quad\text{ for $t\in \br$}. $$
   We describe the recurrent dynamics of the one-parameter subgroup $A_u=\{a_{tu}: t\in \br\}$ on $\Omega_\theta $. That is,
for a given compact subset $Q_0\subset \Omega_\theta$,  we describe 
 when the translate $Q_0 a_{tu}$  comes back to $Q_0$ and what the intersection $Q_0a_{tu}\cap Q_0$
 looks like for $t$ large enough. This is equivalent to studying
 $Q a_{tu}\cap \Gamma Q$ for a compact subset $Q\subset \tilde \Omega_\theta\subset G/S_\theta$. Difficulties arise because $S_\theta$ is not compact, and the $\theta$-transverse hypothesis on $\Ga$ is crucial in the following discussions.

We will repeatedly use the following lemma: note that the product $A_\theta^+B_\theta^+$ is not contained in $A^+$ in general. 

\begin{lemma} \label{lem.repeat}
Suppose that  $\ga_i\in \Ga$ and $d_i\in  A_\theta^+ B_\theta^+$ are sequences such that
  the sequence $\ga_i h_i m_i d_i $ is uniformly bounded
     for some bounded sequences $h_i\in G$ with $h_i P \in \La$ and $m_i \in M_{\theta}$.
   Then there exists $w \in \cal W \cap M_{\theta}$ such that
   after passing to a subsequence,  
   $$d_i\in  wA^+ w^{-1} \quad \text{ for all $i\ge 1$.}$$
\end{lemma}

    \begin{proof} Since $d_i\in A$, by passing to a subsequence, there exists $w\in \cal W$ such that
    $d_i= wc_i w^{-1}$ for some $c_i\in A^+$. We will show that $w\in M_\theta$.
    We may assume without loss of generality that as $i\to \infty$,
    $h_i $ and $m_i$ converge to some $h\in G$ and  $ m\in M_\theta$ respectively.  The $\theta$-regularity of $\Ga$ implies that $\ga_i^{-1} \to \infty$ $\theta \cup \i(\theta)$-regularly.
    Since $h_i':=\ga_i h_i m_i w c_iw^{-1}$ is bounded, it follows that
    $c_i \to \infty$ in $A^+$   $\theta \cup \i(\theta)$-regularly as well by Lemma \ref{lem.cptcartan}.

   By Lemma \ref{lem.210inv}(1)-(2), we have that $\ga_i^{-1} h_i'$ converges to a point in $\La_{\theta \cup \i(\theta)}$ and  $h_i m_i w c_i w^{-1} \to hmwP_{\theta \cup \i(\theta)}$ as $i \to \infty$. Therefore, we have $hmwP_{\theta \cup \i(\theta)} \in \La_{\theta \cup \i(\theta)}$. Since $h P_{\theta \cup \i(\theta)} \in \La_{\theta \cup \i(\theta)}$ by the hypothesis, it follows from the $\theta \cup \i(\theta)$-antipodality of $\Ga$ that  either $w P_{\theta \cup \i(\theta)} = m^{-1} P_{\theta \cup \i(\theta)}$ or 
   $w P_{\theta \cup \i(\theta)}$ is in general position with $ m^{-1} P_{\theta \cup \i(\theta)}$. In the former case, by considering the projection to $\F_\theta$,
   we get $wP_\theta=m^{-1}P_\theta$ and hence $w\in M_\theta$ as desired.
It remains to show that the latter case does not happen.
The latter case would mean that $wP_{\i(\theta)}$ is in general position with $m^{-1}P_\theta=P_\theta$.  
By Lemma \ref{cor.genweyl}, this implies
$w\in w_0 M_{\i(\theta)} =  M_{\theta} w_0$. Writing $d_i=a_ib_i\in A_\theta^+B_\theta^+$ and $w= m_0w_0$ with $m_0\in M_{\theta}\cap N_K(A)$,
we get $c_i= w^{-1} d_i w  = w_0^{-1} a_i w_0 (w_0^{-1}m_0^{-1} b_i m_0 w_0)\in A_{\i(\theta)} (S_{\i(\theta)}\cap A)= A_{\i(\theta)} B_{\i(\theta)}. $
As $c_i\in A^+\subset A_{\i(\theta)}^+ B_{\i(\theta)}^+$, we must have $w_0^{-1} a_i w_0\in A_{\i(\theta)}^+$, which is a contradiction since $a_i\in A_\theta^+$. 
This finishes the proof.
\end{proof}

\begin{prop} \label{lem.dircartan} \label{lem.prodshadows} \label{lem.gromovbound}
Let $Q\subset \tilde \Omega_\theta$ be a compact subset
and $u\in \fa_\theta^+ - \{0\}$. 
There are positive constants $C_1=C_1(Q), C_2=C_2(Q)$ and $R=R(Q)$ such that if $[h] \in Q \cap \ga Qa_{-tu} $ for some $h \in G$, $\ga \in \Ga$ and $t > 0$, then the following hold:
    \begin{enumerate}
        \item   $\| \mu_{\theta}(\ga) - t u \| < C_1;$
\item $(h^+, h^-) \in O_R^{\theta}(o, \ga o) \times O_R^{\i(\theta)}(\ga o, o) ;$
\item $\|\cal G^{\theta}(h^+, h^-) \| < C_2.$
\end{enumerate}
\end{prop}
\begin{proof}  Let $Q'\subset G$ be a compact subset such that $Q'M_\theta =Q'$ and $Q\subset Q'S_\theta/S_\theta$.

To prove (1),
    suppose not. Then there exist sequences $\ga_i \in \Ga$, $h_i \in G$ and a sequence $t_i \to +\infty$ such that  $\|\mu_{\theta}(\ga_i) - t_i u \| \ge i $ and $[h_i] \in Q \cap \ga_i Q a_{-t_i u}$ for all $i\ge 1$. 
    By replacing $h_i$ by an element in $h_iS_\theta$, we may assume that $h_i\in Q'$ 
    and there exist 
    $ h_i' \in Q'$ and $s_i \in S_{\theta}$ such that $h_is_i a_{t_i u}= \ga_i h_i'  $. 
    Since $Q\subset \tilde \Omega_\theta$, 
    we have $h_i P_{\theta} \in \La_{\theta}$. By replacing $h_i$ with an element of $h_i M_{\theta}$, we may assume that $h_i P \in \La$ as well. Since $t_i\to +\infty$, $\ga_i \to \infty$ in $\Ga$. Writing
    $s_i = m_i b_i m_i' \in M_{\theta} B_{\theta}^+ M_{\theta}$ in the Cartan decomposition of $S_\theta$,
     we have $h_im_i a_{t_iu}b_i m_i' = \ga_i h_i'$. By Lemma \ref{lem.repeat}, by passing to a subsequence, there exists $w \in \cal W \cap M_{\theta}$ such that $a_{t_iu}b_{i} = w c_i w^{-1}$ for some $c_i \in A^+$. Since
     $c_i = a_{t_iu} (w^{-1} b_i w) \in A^+ \cap A_\theta B_\theta$, It follows that
     $$\mu_\theta (c_i) = p_\theta (\log c_i)  =t_i u.$$
     
     Since  $h_im_i w c_i w^{-1} m_i' = \ga_i h_i',$
we get that the sequence $\|\mu_\theta (\ga_i) - \mu_\theta (c_i)\|$ is uniformly bounded  by Lemma  \ref{lem.cptcartan}. Hence $\|\mu_\theta(\ga_i)-t_iu\|$ is uniformly bounded, yielding a contradiction.

To prove (2),  suppose not. Then there exist sequences $h_i \in Q$, $\ga_i \in \Ga$ and $t_i > 0$ such that $[h_i] \in Q \cap \ga_i Q a_{-t_i u}$ and $h_i^+ \notin O_i^{\theta} (o, \ga_i o)$ or $h_i^- \notin O_i^{\i(\theta)}(\ga_i o, o)$ for all $i\ge 1$.  As before,
    we may assume $h_i\in Q'$, $h_iP\in \Lambda$ and for some $ h_i' \in Q'$ and $s_i \in S_{\theta}$, we have $h_is_i a_{t_iu} = \ga_i h_i'$. If $\ga_i$ were
    a bounded sequence,  $O_i^{\theta}(o, \ga_i o)\to \F_\theta$ and
     $O_i^{\i(\theta)}(o, \ga_i o)\to \F_{\i(\theta)}$
    as $i\to \infty$, which cannot be the case by the hypothesis on $h_i^{\pm}$.
     Hence $\ga_i \to \infty$ in $\Ga$. 
     As in the proof of Item (1), there exist $w\in \cal W\cap M_\theta$, $b_i\in B_\theta^+$, $m_i, m_i'\in M_\theta $ and $c_i\in A^+$ such that $$h_im_i w c_i w^{-1} m_i' = \ga_i h_i'$$
    and $a_{t_i u}b_i=wc_iw^{-1}$.
      Then we have $h_im_iw P_{\theta} = h_i P_{\theta}$ and $h_im_iw c_i = \ga_i h_i'm_i'^{-1}w$. Since $h_i'm_i'^{-1}w\in Q' $, it follows that 
      $$h_i^+ \in O_{R_0}^{\theta}(h_i o, \ga_i o)  \quad\text{for all $i\ge 1$}$$
      where $R_0 = 1 + \max_{q\in Q'\cup Q'w_0} d(qo, o) > 0$.
    On the other hand, we have $$h_i m_i w w_0^{-1} = \ga_i h_i' m_i'^{-1}w w_0^{-1}(w_0 c_i^{-1}w_0^{-1}),$$ which is a bounded sequence. Since $ \ga_i h_i' m_i'^{-1} w w_0^{-1} P_{\i(\theta)} = h_i m_i w w_0^{-1} P_{\i(\theta)} = h_i w_0 P_{\i(\theta)}$, we have $$h_i^{-} \in O_{R_0}^{\i(\theta)}(\ga_i h_i' o, o) \quad \text{for all $i\ge 1$} .$$ Therefore, 
    by Lemma \ref{lem.shadowtriangle},
    we have $$(h_i^+, h_i^-) \in O_{ 2R_0}^{\theta}(o, \ga_i o) \times O_{2R_0}^{\i(\theta)}(\ga_i  o, o) \quad \text{for all } i\ge 1,$$
    yielding a contradiction.

To prove (3),
as before, we may assume $h\in Q'$ and $h = \ga h_1 a_{-tu} s$ for some $h_1\in Q'$ and $s \in S_{\theta}$.  Then we have $$\begin{aligned}
    \beta_{h^+}^{\theta}(e, h) & = \beta_{h^+}^{\theta}(e, \ga) + \beta_{e^+}^{\theta}(h_1^{-1}, e) + \beta_{e^+}^{\theta}(e, a_{-tu}s) \\
    \beta_{h^-}^{\i(\theta)}(e, h) & = \beta_{h^-}^{\i(\theta)}(e, \ga) + \beta_{e^-}^{\i(\theta)}(h_1^{-1}, e) + \beta_{e^-}^{\i(\theta)}(e, a_{-tu}s).
\end{aligned}$$

Since $\beta_{e^+}^{\theta}(e, a_{-tu}s) + \i (\beta_{e^-}^{\i(\theta)}(e, a_{-tu}s)) = \cal G^{\theta}(e^+, e^-)=0$,  
 we deduce that $$\begin{aligned}
    \cal G^{\theta}(h^+, h^-) & = \beta_{h^+}^{\theta}(e, \ga) + \i( \beta_{h^-}^{\i(\theta)}(e, \ga)) 
    + \beta_{e^+}^{\theta}(h_1^{-1}, e) + \i( \beta_{e^-}^{\i(\theta)}(h_1^{-1}, e)) .
\end{aligned}$$
Observe that
$\| \beta_{e^+}^{\theta}(h_1^{-1}, e) + \i (\beta_{e^-}^{\i(\theta)}(h_1^{-1}, e))\|
\le 2 \max_{q\in Q'} d(qo,o) $.  Since $(h^+, h^-) \in O_R^{\theta}(o, \ga o) \times O_R^{\i(\theta)}(\ga o, o)$ by Item (2), it follows from Lemma \ref{lem.buseandcartan} that 
$$\| \beta_{h^+}^{\theta}(e, \ga) - \mu_{\theta}(\ga) \| \le \kappa R\text{ and } 
\|  \i( \beta_{h^-}^{\i(\theta)}(\ga, e))) -\i( \mu_{\i(\theta)}(\ga^{-1}))\| <\kappa R . $$
Since $\mu_{\theta}(\ga) =\i( \mu_{\i(\theta)}(\ga^{-1}))$, 
we get $\| \beta_{h^+}^{\theta}(e, \ga) +\i( \beta_{h^-}^{\i(\theta)}(e, \ga)) \|\le 2\kappa R ,$
and hence  $$\| \cal G^{\theta}(h^+, h^-)\| \le 2\kappa R + 2 \max_{q\in Q'} d(qo,o) .$$
This finishes the proof. \end{proof}

\subsection*{Directional conical sets} 
 A point $\xi \in \F_\theta$ is called a $\theta$-conical point of $\Ga$ if and only if there exist $R > 0$ and a sequence $\ga_i \to \infty$ in $\Ga$ such that $\xi \in O_R^{\theta}(o, \ga_i o)$, that is, $\xi=k_i P_\theta$ for some $k_i \in K$ such that $d(k_i A^+o, \ga_i o)<R$, for all $i\ge 1$.
Using the identification $\F_\theta=K/M_\theta$, the $\theta$-conical set of $\Ga$ is equal to
\be\label{cd} \La_\theta^{\mathsf{con}}=\left\{k M_\theta \in \F_\theta : k \in K \text{ and } \limsup \Gamma k M_\theta A^+ \ne \emptyset \right\}.\ee

For $r>0$, we set
$$ \Ga_{u, r}  := \{ \ga \in \Ga : \| \mu_{\theta}(\ga) - \br u \| < r  \}.$$

\begin{definition}[Directional conical sets] \label{def.directionalconicalshadow}
    For $u \in \fa_{\theta}^+ - \{0\}$, we say $\xi\in \F_\theta$
    is a
    $u$-directional conical point of $\Ga$ if there exist $R,r > 0$ and a sequence $\ga_i \to \infty$ in $\Ga_{u, r}$ such that  for all $i\ge 1$,
    $$\xi \in O_R^{\theta}(o, \ga_i o),$$
    that is,
 $\xi=k_i P_\theta$ for some $k_i \in K$ such that $d(k_iA^+o, \ga_i o)<R$.
In other words, the $u$-directional conical set  is  given by
    \be \La_{\theta}^u = \{k M_{\theta} \in \F_{\theta} : k \in K \text{ and } \limsup \Ga_{u, r}^{-1} kM_{\theta}A^+ \neq \emptyset \text{ for some } r > 0\}.\ee 
We note that $\Ga_{u, r}^{-1}=\{\ga\in \Ga:
\|\mu_{\i(\theta)}(\ga)-\br \i(u)\|<r\}$.
\end{definition}
 
Clearly, $\La_{\theta}^{u} \subset \La_{\theta}^{\sf con}$ for all $u \in \fa_{\theta}^+ - \{0\}$ and $\La_{\theta}^u = \emptyset$ if $u \notin \L_{\theta}$. 
These notions of conical and directional conical sets can be defined for any discrete subgroup. On the other hand, for $\theta$-transverse subgroups, these notions can also be defined in terms of recurrence of $A_\theta$ and $A_u$-actions on $\Omega_\theta$ respectively: we emphasize that for a sequence $g_i\in G$,
the sequence $[g_i]\in \tilde \Omega_\theta$ is precompact if and only if  there exists
$s_i\in S_\theta$ (which is not necessarily bounded) such that the sequence $g_is_i$ is bounded in $G$.

\begin{lem}[Conical points and recurrence] \label{defdir}
Let $\Ga$ be $\theta$-transverse.
Then
\begin{enumerate}
    \item   $\xi \in \La_{\theta}^{\mathsf{con}}$ if and only if $\xi=gP_\theta$ for some $g\in G$
 such that  $[g]\in \tilde \Omega_\theta$ and $\gamma_i[g]a_i $ is precompact in $\tilde \Omega_\theta$
 for infinite sequences $\gamma_i\in \Ga$ and $a_i\in A_\theta^+$.

 \item $\xi \in \La_{\theta}^{u}$ if and only if $\xi=gP_\theta$ for some $g\in G$
 such that  $[g]\in \tilde \Omega_\theta$ and $\gamma_i[g] a_{t_iu}$ is precompact in $\tilde{\Omega}_{\theta}$
 for infinite sequences $\gamma_i\in \Ga$ and $t_i>0$.
 \end{enumerate}

\end{lem}
\begin{proof} {\bf Item (1):} Let $\xi \in \La_{\theta}^{\sf con}$; so there exist $k \in K$, $\ga_i \in \Ga$, $m_i \in M_{\theta}$ and $c_i \in A^+$ so that $\xi = kP_{\theta}$ and $\ga_i k m_i c_i$ is a bounded sequence in $G$. By the $\theta$-regularity of $\Ga$, we have $\La_{\theta}^{\sf con} \subset \La_{\theta}$ \cite[Prop. 5.6(1)]{KOW_indicators}, and hence $k^+=kP_\theta \in \La_{\theta}$. Since $\La_{\i(\theta)}$ is Zariski dense and $kN_{\theta}w_0P_{\i(\theta)}$ is a Zariski open subset of $\F_{\i(\theta)}$, we have $(kn)^- \in \La_{\i(\theta)}$ for some $n \in N_{\theta}$. Since $(kn)^+ = k^+ = \xi$, we have $[kn] \in \tilde{\Omega}_{\theta}$.
Note that $\ga_i k n m_i c_i = (\ga_i k m_i c_i)(c_i^{-1}n_i'c_i)$ where $n_i' := m_i^{-1}nm_i \in N_{\theta}$ is a bounded sequence. Since $c_i \in A^+$, the sequence $c_i^{-1}n_i'c_i$ is bounded as well and hence $\ga_i k n m_i c_i$ is bounded.
Write $c_i = b_ia_i \in B_{\theta}^+A_{\theta}^+$; so the sequence $\ga_i (knm_ib_i) a_i$ is contained in some compact subset of $G$ and $m_ib_i \in S_{\theta}$. Since the map
$g\mapsto [g]\in \tilde \Omega_\theta$ is continuous, and hence the image of a compact subset is compact, the sequence $\ga_i[k n]a_i = [\ga_i k n m_i b_i a_i]$ is precompact in $\tilde{\Omega}_{\theta}$, as desired.

Conversely, suppose that $\xi = gP_{\theta}$ for some $g \in G$ such that $[g] \in \tilde{\Omega}_{\theta}$ and $\ga_i[g]a_i$ is precompact for infinite sequences $\ga_i \in \Ga$ and $a_i \in A_{\theta}^+$. We can replace $g$ with an element in $gM_{\theta}$ so that $gP \in \La$. Since the sequence $\ga_i [g] a_i = [\ga_i g a_i]$ is precompact, there exists a bounded sequence $h_i \in G$ such that for all $i\ge 1$, $[h_i] =  \ga_i [g]a_i \in \tilde{\Omega}_{\theta}$, that is, $ga_is_i = \ga_i^{-1}h_i$  for some $s_i \in S_{\theta}$.
 Writing the Cartan decomposition $s_i = m_i b_i m_i' \in M_{\theta} B_{\theta}^+ M_{\theta}$, we have $gm_ia_ib_i m_i' = \ga_i^{-1}h_i$. Since the sequence $\ga_i g m_i a_ib_i = h_i m_i'^{-1}$ is bounded, it follows from Lemma \ref{lem.repeat} that $a_ib_i = wc_iw^{-1}$ for some $w \in \cal W \cap M_{\theta}$ and $c_i \in A^+$, after passing to a subsequence. Hence we have $g m_i w c_i = \ga_i^{-1}h_i m_i'^{-1} w$, which implies that $\xi = gP_{\theta} \in O_R^{\theta}(go, \ga_i^{-1}o)$ for all $i$ where $R =1 +  \max_i d(h_io, o)$. By Lemma \ref{lem.shadowtriangle}, we have $\xi \in O_{R + d(go, o)}^{\theta}(o, \ga_i^{-1}o)$ for all $i\ge 1$, completing the proof.

\smallskip
{\noindent \bf Item (2):} Let $\xi \in \La_{\theta}^u$. Then $\xi = kP_{\theta}$ for some $k \in K$ and $\ga_i k m_i a_i$ is a bounded sequence in $G$ for some infinite sequences $\ga_i \in \Ga_{u, r}^{-1}$, $m_i \in M_{\theta}$ and $a_i \in A^+$. Since $\xi = kP_{\theta} \in \La_{\theta}^u$ and $\La_{\theta}^u \subset \La_{\theta}^{\sf con} \subset \La_{\theta}$ by the $\theta$-regularity of $\Ga$ \cite[Prop. 5.6(1)]{KOW_indicators}, we have $k^+ \in \La_{\theta}$. As in the proof of Item (1) above, there exists $n \in N_{\theta}$ so that $(kn)^- \in \La_{\i(\theta)}$ and $\ga_i k n m_i a_i$ is bounded. In particular, $[kn] \in \tilde{\Omega}_{\theta}$.

Since $\ga_iknm_ia_i$ is a bounded sequence in $G$ and $\ga_i^{-1} \in \Ga_{u, r}$, we have $a_i = a_{t_i u} b_i$ for some $t_i > 0$ and a bounded sequence $b_i \in A$ by Lemma \ref{lem.cptcartan}. Hence the sequence $\ga_i k nm_i a_{t_iu}$ is bounded as well. Therefore, $\ga_i [kn] a_{t_i u} = [\ga_i k n m_i a_{t_iu}]$ is precompact in $\tilde{\Omega}_{\theta}$. Since $(kn)^+ = k^+ = \xi$, this shows the only if direction in (2).

To show the converse implication, suppose that the sequence $\gamma_i[g]a_{t_iu}$ is contained in some compact subset $Q$ of $\tilde{\Omega}_{\theta}$ which we also assume contains $[g]$. Since $[g] \in Q \cap \ga_i^{-1} Q a_{-t_i u}$, it follows from Proposition \ref{lem.dircartan} that $\ga_i^{-1} \in \Ga_{u, C_1}$ and $g^+ = g P_{\theta} \in O_R^{\theta}(o, \ga_i^{-1} o)$ for all $i\ge 1$ where $C_1 = C_1(Q)$ and $R = R(Q)$ are given in Proposition \ref{lem.dircartan}. Therefore, $g^+ \in \La_{\theta}^u$.
\end{proof}

\begin{theorem} \label{thm.convergencedir}
Let $\Ga<G$ be a Zariski dense discrete subgroup. Let $u \in \fa_{\theta}^+ - \{0\}$ and $\psi \in \fa_{\theta}^*$ be $(\Ga, \theta)$-proper. Suppose that $\sum_{\ga \in \Ga_{u, r}} e^{-\psi(\mu_{\theta}(\ga))} < \infty$ for all $r > 0$. For any $(\Ga, \psi)$-conformal measure $\nu$ on $\F_{\theta}$, we have
   $$\nu(\La_{\theta}^u) = 0.$$ 
\end{theorem}
\begin{proof}
    For each $r> 0$, we set $\La_{\theta, r}^u = \limsup_{\ga \in \Ga_{u, r}} O_r^{\theta}(o, \ga o)$. In other words, $\xi \in \La_{\theta, r}^u$ if and only if there exists a sequence $\ga_i \to \infty$ in $\Ga_{u, r}$ such that $\xi \in O_{r}^{\theta}(o, \ga_i o)$ for all $i \ge 1$. Then $\La_{\theta}^u = \bigcup_{r > 0} \La_{\theta, r}^u$. Let $\nu$ be a $(\Ga, \psi)$-conformal measure on $\F_{\theta}$. Since $$\La_{\theta, r}^u \subset \bigcup_{\ga \in \Ga_{u, r}, \|\mu_\theta(\ga) \| > t} O_r^{\theta}(o, \ga o)
    \quad\text{ for all $t > 0$},$$ it follows from Lemma \ref{lem.shadowlemma} that 
   \be \label{eqn.partialsum}
    \nu(\La_{\theta, r}^u) \ll \sum_{\ga \in \Ga_{u, r}, \|\mu_\theta(\ga) \| > t} e^{-\psi(\mu_{\theta}(\ga))}
    \quad\text{for all $t > 0$.}  \ee
   Since $\sum_{\ga \in \Ga_{u, r}} e^{-\psi(\mu_{\theta}(\ga))} < \infty$, taking $t \to \infty$ in \eqref{eqn.partialsum} implies $\nu(\La_{\theta, r}^u) = 0$. Therefore, $\nu(\La_{\theta}^u) = \limsup_{r \to \infty} \nu(\La_{\theta, r}^u) = 0$.
\end{proof}

\begin{lemma} \label{lem.tangent}
    If $\sum_{\ga \in \Ga_{u, r}} e^{-\psi(\mu_{\theta}(\ga))} = \infty$ for some $r > 0$,
    then $\psi(u)>0$.
If     there exists a $(\Ga, \psi)$-conformal measure on $\F_\theta$ in addition, then $$\psi(u) =\psi_{\Ga}^{\theta}(u) .$$
     Moreover the abscissa of convergence of the series $s\mapsto \sum_{\ga \in \Ga_{u, r}} e^{-s\psi(\mu_{\theta}(\ga))}$ is equal to one.
\end{lemma}
\begin{proof}
Suppose that $\sum_{\ga \in \Ga_{u, r}} e^{-\psi(\mu_{\theta}(\ga))} = \infty$. Then $\# \Ga_{u, r} = \infty$. If $\psi(u) $ were not positive, then $(\psi \circ \mu_{\theta})(\Ga_{u, r}) $ is contained in the interval
 $ (-\infty, \|\psi\|r]$. Therefore it contradicts the $(\Ga, \theta)$-proper hypothesis on $\psi$. Hence $\psi(u) > 0$.

Now suppose that there exists a $(\Ga, \psi)$-conformal measure on $\F_{\theta}$. We then have $\psi \ge \psi_{\Ga}^{\theta}$ by Theorem \ref{tgg}. Now suppose that $\psi(u) > \psi_{\Ga}^{\theta}(u)$. We may assume that $u$ is a unit vector
 as both $\psi$ and $\psi_\Ga^\theta$ are homogeneous of degree one. By the definition of $\psi_{\Ga}^{\theta}$, there exists an open cone $\cal C$ containing $u$ so that $\sum_{\ga \in \Ga, \mu_{\theta}(\ga) \in \cal C} e^{-\psi(u) \|\mu_{\theta}(\ga)\|} < \infty.$ Since $\mu_{\theta}(\Ga_{u, r})$ is contained in $\cal C$ possibly except for finitely many elements, we have $$\sum_{\ga \in \Ga_{u, r}} e^{-\psi(\mu_{\theta}(\ga))} \ll \sum_{\ga \in \Ga_{u, r}} e^{-\psi(u) \|\mu_{\theta}(\ga)\|} < \infty,$$ which is a contradiction. Therefore, $\psi(u) =\psi_{\Ga}^{\theta}(u)$. 
 
 We now show the last claim. Since $\sum_{\ga \in \Ga} e^{-\psi(\mu_{\theta}(\ga))} \ge \sum_{\ga \in \Ga_{u, r}} e^{-\psi(\mu_{\theta}(\ga))} = \infty$, the abscissa of convergence of $s \mapsto \sum_{\ga \in \Ga} e^{-s \psi(\mu_{\theta}(\ga))}$ is equal to one by Theorem \ref{tgg}.
 Hence the abscissa of convergence  of $s\mapsto \sum_{\ga \in \Ga_{u, r}} e^{-s\psi(\mu_{\theta}(\ga))}$ is at most one. Since $\sum_{\ga \in \Ga_{u, r}} e^{-\psi(\mu_{\theta}(\ga))} = \infty$, it must be exactly one.
\end{proof}

\section{Bowen-Margulis-Sullivan measures}
In this short section, we recall the definition of  Bowen-Margulis-Sullivan measures on $\Omega_\theta$. We also recall one-dimensional flow space $\Omega_\psi$ and the corresponding Bowen-Margulis-Sullivan measures on it.

Let $\Ga < G$ be a Zariski dense $\theta$-transverse subgroup. Recall our flow space from the previous section:
$$\Omega_{\theta} = \Ga \ba \La_{\theta}^{(2)} \times \fa_{\theta}$$
where the action is given by \eqref{hopf}.

\subsection*{Bowen-Margulis-Sullivan measures on $\Omega_\theta$}  We may identify $\fa_\theta^*$ with $\{\psi\in \fa^*: \psi \circ p_\theta=\psi\}$.
Hence for $\psi\in \fa_\theta^*$, we have $ \psi \circ \i\in \fa_{\i(\theta)}^*$.
For a pair of a $(\Ga, \psi)$-conformal measure $\nu$ on $\La_{\theta}$ and  a $(\Ga, \psi \circ \i)$-conformal measure 
$\nu_{\i}$ on $\La_{\i(\theta)}$, we  define a Radon measure $d \tilde{\mathsf m}_{\nu, \nu_{\i}}$ on $\La_{\theta}^{(2)} \times \fa_{\theta}$ as follows: \be\label{bms} d\tilde{\mathsf m}_{\nu, \nu_{\i}} (\xi, \eta, b) = e^{\psi (\cal G^{\theta}(\xi, \eta))} d\nu(\xi) d \nu_{ \i}(\eta) db\ee  where $db$ is the Lebesgue measure on $\fa_{\theta}$.
It is easy to check that $\tilde{\mathsf m}_{\nu, \nu_{\i}} $
is left $\Ga$-invariant, and hence induces a $A_\theta$-invariant Radon measure on $\Omega_\theta$ which we 
 denote by
\be\label{bdef} {\mathsf m}_{\nu, \nu_{\i}} .\ee  We call it the Bowen-Margulis-Sullivan measure
(or simply BMS measure) associated with the pair $(\nu, \nu_{\i})$.

\subsection*{Bowen-Margulis-Sullivan measures on $\Omega_\psi$}
Let $\psi\in \fa_\theta^*$ be a $(\Ga, \theta)$-proper form. We remark that
this implies that $\psi\ge 0$ on $\L_\theta$ and $\psi>0$ on $\inte\L_\theta$ \cite[Lem. 4.3]{KOW_indicators}.
Consider the $\Ga$-action on $\tilde \Omega_\psi:=\La_\theta^{(2)}\times \br$ given by
\be \label{linearformhopf}
\ga.(\xi. \eta, s)=(\ga \xi, \ga \eta, s+\psi (\beta_\xi^\theta (\ga^{-1}, e)))
\ee
for $\ga\in \Ga$ and $(\xi, \eta, s)\in \La_\theta^{(2)}\times \br$.
\begin{theorem}\cite[Thm. 9.2]{KOW_indicators} \label{thm:opsi}
    If $\Ga$ is Zariski dense $\theta$-transverse and $\psi\in \fa_\theta^*$ is $(\Ga, \theta)$-proper, then $\Ga$ acts properly discontinuously on $\tilde \Omega_\psi$ and hence
    \be \label{opsi}
    \Omega_\psi:=\Ga\ba \tilde \Omega_\psi\ee  is a second countable locally compact Hausdorff space.
\end{theorem}

The map $\La_\theta^{(2)}\times\fa_\theta\to
\La_\theta^{(2)}\times \br$ given by $(\xi, \eta, v)\mapsto (\xi, \eta, \psi(v))$ is a principal $\ker \psi$-bundle which is trivial since $\ker \psi$ is a vector space. Therefore it induces a $\ker \psi$-equivariant homeomorphism between
\be\label{trivial} \Omega_\theta \simeq \Omega_\psi \times \ker\psi .\ee 

Let \be \label{omm} \m_{\nu, \nu_{\i}}^{\psi} \ee
be the Radon measure  on $\Omega_\psi$ induced from the $\Ga$-invariant measure on $\tilde \Omega_\psi$:
$$ d\tilde \m_{\nu, \nu_{\i}}^{\psi} (\xi, \eta, s) := e^{\psi(\cal G^{\theta}(\xi, \eta))} d\nu(\xi)d\nu(\eta) ds. $$

We then have  $$ \m_{\nu, \nu_{\i}}=\m_{\nu, \nu_{\i}}^{\psi} \otimes \op{Leb}_{\ker \psi} .$$

\section{Directional conical sets and  Poincar\'e series} \label{sec.dirpoincare}
In this section, we relate the divergence of the directional $\psi$-Poincar\'e series with the size of the directional conical set with respect to a $(\Ga, \psi)$-conformal measure on $\F_\theta$. The main theorem of this section (Theorem \ref{prop.mainprop})
is the most significant part of Theorem \ref{main}. 

\smallskip

Let $\Ga < G$ be a Zariski dense $\theta$-transverse subgroup. We fix 
$$u \in \fa_{\theta}^+ - \{0\}\text{ and a } (\Ga, \theta)\text{-proper }
    \psi\in \fa_\theta^*.$$ We also fix a pair  $\nu, \nu_{\i}$ of $(\Ga, \psi)$ and $(\Ga, \psi\circ\i)$-conformal measures on $\La_\theta$ and $\La_{\i(\theta)}$ respectively. Denote by $\tilde \m=\tilde \m_{\nu, \nu_{\i}}$ and $\m=\m_{\nu, \nu_{\i}}$   the  associated BMS measures on $\tilde \Omega_\theta$ and $\Omega_\theta$ respectively.
 
Our dichotomy theorem is stated under a hypothesis on $\m$ which we call a $u$-balanced condition:
 
 \begin{definition} [$u$-balanced condition] \label{def.ubalanced}
     A Borel measure space $(X, m)$ with $\{a_{tu}\}$-action
     is called $u$-balanced or simply $m$ is $u$-balanced, if for any bounded Borel subset $O_i \subset X$ with $m(O_i) > 0$ for $i = 1, 2$,   for all $T>1$,
$$\int_0^T m(O_1 \cap O_1 a_{tu}) dt\asymp\footnote{The notation
    $f(T)\asymp g(T)$ means that $f(T), g(T)\to \infty$ as $T\to \infty$ and $f(T)\ll g(T)$ and $g(T)\ll f(T)$.} \int_0^T m(O_2 \cap O_2 a_{tu} )dt .$$ 
 \end{definition}
 
The main goal of this section is to prove the following: 
\begin{theorem} \label{prop.mainprop}
    Suppose that $\m$ is $u$-balanced. If $ \sum_{\ga \in \Ga_{u, r}} e^{-\psi(\mu_{\theta}(\ga))} = \infty$ for some $r > 0$, then $$\nu(\La_{\theta}^u) > 0 \quad \text{and} \quad \nu_{\i}(\La_{\i(\theta)}^{\i(u)}) > 0.$$
\end{theorem}

\begin{rmk}  When $\sum_{\ga \in \Ga} e^{-\psi(\mu_{\theta}(\ga))} = \infty$, there exists at most one $(\Ga, \psi)$-conformal measure on $\F_\theta$ (\cite[Thm. 1.5]{KOW_indicators}). Furthermore,
the existence of a $(\Ga, \psi)$-conformal measure on $\La_\theta$ implies the existence of $(\Ga, \psi\circ \i)$-conformal measure on $\La_{\i(\theta)}$ as well.
Indeed,
it follows from \cite[Thm. 7.1]{KOW_indicators} that $\delta_{\psi} = 1$ where $\delta_{\psi}$ is the abscissa of the convergence  of the Poincar\'e series $s \mapsto \sum_{\ga \in \Ga} e^{-s \psi(\mu_{\theta}(\ga))}$. In particular, $\delta_{\psi \circ \i} = \delta_{\psi} = 1$. By \cite{CZZ_transverse} and \cite[Lem. 9.5]{KOW_indicators}, there exists a $(\Ga, \psi \circ \i)$-conformal measure $\nu_{\i}$ on $\La_{\i(\theta)}$ which is the unique $(\Ga, \psi \circ \i)$-conformal measure on $\F_{\i(\theta)}$, since $\sum_{\ga \in \Ga} e^{-(\psi \circ \i)(\mu_{\i(\theta)}(\ga))} = \infty$ as well.
\end{rmk}

For simplicity, we set for all $t\in \br$
$$a_t=a_{tu}=\exp tu.$$
The following proposition is the key ingredient of the proof of Theorem \ref{prop.mainprop}:

\begin{proposition} \label{prop.upperbound} \label{prop.lowerbound}
Suppose that $\sum_{\ga \in \Ga_{u, r}} e^{-\psi(\mu_{\theta}(\ga))} = \infty$ for some $r > 0$. Set $\delta=\psi(u)$, which is positive by Lemma \ref{lem.tangent}.
 \begin{enumerate}
     \item 
  For any compact subset $Q\subset \tilde \Omega_\theta$, there exists $r=r(Q)>0$ such that for any $T > 1$, we have 
    $$\int_0^T \int_0^T \sum_{\ga, \ga' \in \Ga} \tilde{\m}(Q \cap \ga Q a_{-t} \cap \ga' Q a_{-t-s}) dt ds \ll \left( \sum_{\substack{\ga \in \Ga_{u, r} \\ \psi(\mu_{\theta}(\ga)) \le \delta T}} e^{-\psi(\mu_{\theta}(\ga))} \right)^2.
    $$
\item 
    For any $r > 0$, there exists a compact subset $Q'=Q'(r)\subset \tilde \Omega_\theta$ such that for any $T > 1$,
     $$\int_0^T \sum_{\ga \in \Ga} \tilde{\m}(Q' \cap \ga Q' a_{-t}) dt  \gg \sum_{\substack{\ga \in \Ga_{u, r} \\ \psi(\mu_{\theta}(\ga)) \le \delta T}} e^{-\psi(\mu_{\theta}(\ga))}.
     $$ 
 \end{enumerate} 
 \end{proposition}

To prove this proposition, we relate the integrals on the left hand sides to shadows and apply the shadow lemma. Together with results obtained in Section \ref{sec:rec},
the following proposition on the   multiplicity bound on shadows for transverse subgroups is crucial.

\begin{prop} \cite[Prop. 6.2]{KOW_indicators} \label{prop.mult}  For any $R, D > 0$, there exists $q = q(\psi, R, D) > 0$ such that for any $T > 0$, the collection of shadows $$\left\{O^{\theta}_R(o, \ga o) \subset \F_\theta : T \le \psi (\mu_{\theta}(\ga)) \le T + D \right\}$$ has multiplicity at most $q$.
\end{prop}

\begin{lem} \label{lem.shadowlikeintersection}
    Let $Q \subset \tilde{\Omega}_{\theta}$ be a compact subset. For any $t > 1$, we have $$\tilde{\m}(Q \cap \ga Q a_{-t}) \ll e^{-\psi(\mu_{\theta}(\ga))}$$ where the implied constant is independent of $t$.
\end{lem}

\begin{proof}
    There exists $c_0 = c_0(Q) > 0$ such that if $Q \cap Q a \neq \emptyset$ for some $a \in A_{\theta}$, then $\| \log a \| < c_0$.
By Proposition \ref{lem.prodshadows}(2) and the compactness of $Q$, we have for large enough $R >0$ that $$ \begin{aligned}
    \tilde{\m}&(Q \cap \ga Q a_{-t}) \\
    &\ll  \int_{O_R^{\theta}(o, \ga o) \times O_R^{\i(\theta)}(\ga o, o)}  \int_{\fa_{\theta}} \mathbbm{1}_{Q \cap \ga Q a_{-t}}(\xi, \eta, b) e^{\psi(\cal G^{\theta}(\xi, \eta))} d\nu(\xi) d \nu_{\i}(\eta) db \\
&\ll  \int_{O_R^{\theta}(o, \ga o) \times O_R^{\i(\theta)}(\ga o, o)}   \mathbbm{1}_{Q \cap \ga Q a_{-t}}(\xi, \eta) e^{\psi(\cal G^{\theta}(\xi, \eta))} d\nu(\xi) d \nu_{\i}(\eta).
\end{aligned}$$
Since $\cal G^{\theta}(\xi, \eta)$ in the above integrand is uniformly bounded by Proposition \ref{lem.prodshadows}(3), we obtain
$$
\tilde{\m} (Q \cap \ga Q a_{-t}) \ll  \nu(O_R^{\theta}(o, \ga o)) \nu_{\i}(O_R^{\i(\theta)}(\ga o, o)).
$$
By Lemma \ref{lem.shadowlemma}, we have $$\tilde{\m}(Q \cap \ga Q a_{-t}) \ll \nu(O_R^{\theta}(o, \ga o)) \ll e^{-\psi(\mu_{\theta}(\ga))}.$$ \end{proof}

The following is immediate from Proposition \ref{lem.dircartan}(1).
\begin{lem} \label{cor.dircartan}
    Let $Q \subset \tilde{\Omega}_{\theta}$ be a compact subset. If $Q \cap \ga Q a_{-t} \cap \ga' Q a_{-t-s} \neq \emptyset$ for some $\ga, \ga' \in \Ga$ and $t, s > 0$, then we have
    \begin{enumerate}
        \item $\|\mu_{\theta}(\ga) - t u \|$ , $\| \mu_{\theta}(\ga^{-1} \ga') - s u \|$, $\| \mu_{\theta}(\ga') - (t + s)u\| < C_1$;
        \item $\psi(\mu_{\theta}(\ga)) + \psi(\mu_{\theta}(\ga^{-1}\ga')) < \psi(\mu_{\theta}(\ga')) + 3C_1 \|\psi\|$
    \end{enumerate} where $C_1=C_1(Q)$ is given as in Proposition \ref{lem.dircartan}.
\end{lem}

\subsection*{Proof of Proposition \ref{prop.upperbound}(1)}
    Let $Q \subset \tilde{\Omega}_{\theta}$ be a compact subset.  
    Fix $s,t>0$. For $\ga, \ga' \in \Ga$ such that $Q \cap \ga Q a_{-t} \cap \ga' Q a_{-t-s} \neq \emptyset$, it follows from Lemma \ref{lem.shadowlikeintersection} that $$\tilde{\m}(Q \cap \ga Q a_{-t} \cap \ga' Q a_{-t-s}) \ll e^{-\psi(\mu_{\theta}(\ga'))}.$$ By Lemma \ref{cor.dircartan}(2), we have $\psi(\mu_{\theta}(\ga)) + \psi(\mu_{\theta}(\ga^{-1}\ga')) < \psi(\mu_{\theta}(\ga')) + 3 C_1 \|\psi\|$ and hence $$\tilde{\m}(Q \cap \ga Q a_{-t} \cap \ga' Q a_{-t-s}) \ll  e^{-\psi(\mu_{\theta}(\ga))} e^{-\psi(\mu_{\theta}(\ga^{-1}\ga'))}.$$
    Since we also have $\| \mu_{\theta}(\ga) - tu \|$, $\| \mu_{\theta}(\ga^{-1}\ga') - su \| < C_1$ by Lemma \ref{cor.dircartan} where $C_1$ is given in Proposition \ref{lem.dircartan}(1), we deduce by replacing $\ga^{-1}\ga'$ with $\hat{\ga}$ that $$\begin{aligned}
        \sum_{\ga, \ga' \in \Ga} & \tilde{\m}(Q \cap \ga Q a_{-t} \cap \ga' Q a_{-t-s}) \\
        & \ll \left( \sum_{\substack{\ga \in \Ga_{u, C_1} \\ \psi(\mu_{\theta}(\ga)) \in (\delta t - c, \delta t + c)}} e^{-\psi(\mu_{\theta}(\ga))} \right) \left( \sum_{\substack{\hat{\ga} \in \Ga_{u, C_1} \\ \psi(\mu_{\theta}(\hat{\ga})) \in (\delta s - c, \delta s + c)}} e^{-\psi(\mu_{\theta}(\hat{\ga}))} \right)
    \end{aligned}$$
    where $c := C_1 \|\psi\|$.

    We observe that if $\psi(\mu_{\theta}(\ga)) \in (\delta t - c, \delta t + c)$ for some $t \in [0, T]$, then $\psi(\mu_{\theta}(\ga)) \le \delta T + c$. Hence we have
    $$\int_0^T \left( \sum_{\substack{\ga \in \Ga_{u, C_1} \\ \psi(\mu_{\theta}(\ga)) \in (\delta t - c, \delta t + c)}} e^{-\psi(\mu_{\theta}(\ga))} \right) dt \ll \sum_{\substack{\ga \in \Ga_{u, C_1} \\ \psi(\mu_{\theta}(\ga)) \le \delta T + c}} e^{-\psi(\mu_{\theta}(\ga))}.$$ Similarly we also have $$\int_0^T  \left( \sum_{\substack{\hat{\ga} \in \Ga_{u, C_1} \\ \psi(\mu_{\theta}(\hat{\ga})) \in (\delta s - c, \delta s + c)}} e^{-\psi(\mu_{\theta}(\hat{\ga}))} \right) ds \ll  \sum_{\substack{\hat{\ga} \in \Ga_{u, C_1} \\ \psi(\mu_{\theta}(\hat{\ga})) \le \delta T + c}} e^{-\psi(\mu_{\theta}(\hat{\ga}))}.$$
    Therefore, we have $$\int_0^T \int_0^T \sum_{\ga, \ga' \in \Ga} \tilde{\m}(Q \cap \ga Q a_{-t} \cap \ga' Q a_{-t-s}) dt ds \ll \left(\sum_{\substack{\ga \in \Ga_{u, C_1} \\ \psi(\mu_{\theta}(\ga)) \le \delta T + c}} e^{-\psi(\mu_{\theta}(\ga))}\right)^2.$$
Since $$\sum_{\substack{\ga \in \Ga_{u, C_1} \\ \delta T < \psi(\mu_{\theta}(\ga)) \le \delta T + c}} e^{-\psi(\mu_{\theta}(\ga))} \ll \sum_{\substack{\ga \in \Ga_{u, C_1} \\ \delta T < \psi(\mu_{\theta}(\ga)) \le \delta T + c}} \nu (O_R^{\theta}(o, \ga o)) \ll 1$$ for large $R = R(\nu)$ by Lemma \ref{lem.shadowlemma} and Proposition \ref{prop.mult}, setting $r(Q) = C_1(Q)$ completes the proof. \qed

\medskip

\begin{lem}\label{lem.prodshadowgromov}
     For any $R > 0$, there exists $0 < \ell_R<\infty$ such that  any  $(\xi, \eta) \in \bigcup_{\ga\in \Ga, \|\mu_{\theta}(\ga)\| > \ell_R} O_{R}^{\theta}(o, \ga o) \times O_{R}^{\i(\theta)}(\ga o, o)$ satisfies $\| \cal G^{\theta}(\xi, \eta) \| < \ell_R.$
\end{lem}

\begin{proof}
    Suppose not. Then there exist sequences $\ga_i \to \infty$ in $\Ga$ and $(\xi_i, \eta_i) \in O_{R}^{\theta}(o, \ga_i o) \times O_{R}^{\i(\theta)}(\ga_i o, o)$ such that $\| \cal G^{\theta}(\xi_i, \eta_i) \| \to \infty$ as $i \to \infty$.
    We may assume that $\xi_i\to \xi$ and $\eta_i\to \eta$ by passing to subsequences.
    As $\ga_i\to \infty$ $\theta$-regularly,
    Lemma \ref{shad_con} implies that
    $(\xi, \eta)\in \F_\theta^{(2)} $.
   Since $\| \cal G^{\theta}(\xi_i, \eta_i) \| \to \| \cal G^{\theta}(\xi, \eta)\| < \infty$, this is a contradiction.
\end{proof}

   \begin{lem} \label{lem.findsegment} Let $u\in \fa_\theta^+ - \{0\}$.
    For any $r,  R > 0$,  there exists a compact subset $Q = Q(r,R) \subset \tilde{\Omega}_{\theta}$ such that for any $\ga \in \Ga_{u, r}$ with $\|\mu_{\theta}(\ga) \| > \ell_{R}$ and
    $$(\xi, \eta) \in \left(O_{R}^{\theta}(o, \ga o) \times O_{R}^{\i(\theta)}(\ga o, o)\right) \cap \La_{\theta}^{(2)},$$ there exists $v \in \fa_{\theta}$ and $t \ge 0$ such that $$(\xi, \eta, v) \in Q \quad \text{and} \quad (\xi, \eta, v) a_{[t-1, t+1]} \subset \ga Q.$$
\end{lem}

\begin{proof}
    Let $(\xi, \eta) \in  (O_{R}^{\theta}(o, \ga o) \times O_{R}^{\i(\theta)}(\ga o, o)) \cap \La_{\theta}^{(2)}$ for some $\ga \in \Ga_{u, r}$ with $ \| \mu_{\theta}(\ga) \| > \ell_R$. Then there exists $k \in K$ such that $\xi = kP_{\theta}$ and $d(ka_0 o, \ga o) < R$ for some $a_0 \in A^+$. Write $a_0 = ab \in A_{\theta}^+B_{\theta}^+$.

    By Lemma \ref{lem.cptcartan}, we have $\|\mu(\ga) - \log a_0 \| < D$ for some $D = D(R)$, and hence $\|\mu_{\theta}(\ga) - \log a \| < D.$ We also obtain from $\ga \in \Ga_{u, r}$ that $\|\mu_{\theta}(\ga) - tu \| < r$ for some $t \ge 0$ and hence we have $\|tu - \log a \| < D + r.$
    Therefore, we have \be \label{eqn.lower1}
    \begin{aligned}
    d(ka_{tu}b o, \ga o) & \le d(ka_{tu} b o, ka_0 o) + d(ka_0 o, \ga o) \\
    & = d(a_{tu} o, a o) + d(ka_0 o, \ga o) \\
    & < D + r + R.
    \end{aligned}
    \ee
We also note that $$\|tu + \log b - \log a_0 \| = \|tu - \log a \| < D + r.$$ Hence there exists $\tilde{a} \in A$ such that $$\|\log \tilde{a}\| < D + r \quad \text{and} \quad a_{tu} b \tilde{a} \in A^+.$$
Let $g_0 \in G$ such that $(g_0P_{\theta}, g_0w_0 P_{\i(\theta)}) = (\xi, \eta)$. Since $(\xi, \eta) \in  O_{R}^{\theta}(o, \ga o) \times O_{R}^{\i(\theta)}(\ga o, o)$ and $\|\mu_{\theta}(\ga) \| > \ell_R$, we have $\| \cal G^{\theta}(\xi, \eta) \| < \ell_R$. By Proposition \ref{Gromov}, we can replace $g_0$ by an element of $g_0 L_{\theta}$ so that we may assume that $$d(o, g_0 o) \le c \| \cal G^{\theta}(\xi, \eta) \| + c' < c \ell_R + c'.$$
Since $\xi = kP_{\theta} = g_0 P_{\theta}$, we have $g_0^{-1}k \in P_{\theta}$. We write the Iwasawa decomposition $$g_0^{-1}k = m \hat{a} \hat{n} \in KAN.$$ Then we have $m = g_0^{-1}k \hat{n}^{-1} \hat{a}^{-1} \in P_{\theta} \hat{n}^{-1} \hat{a}^{-1} = P_{\theta}$. In particular, we have $m \in P_{\theta} \cap K = M_{\theta}$. We let $g = g_0 m$. Since $m \in M_{\theta} \subset L_{\theta}$, we still have $(gP_{\theta}, gw_0P_{\i(\theta)}) = (\xi, \eta)$ and $d(o, go) = d(o, g_0 o) < c\ell_R + c'$. Moreover, we have $g^{-1}k = \hat{a}\hat{n} \in P$. Now for $s \in [t - 1, t + 1]$, we have $$\begin{aligned}
        &d(gba_{su} o, kb a_{tu} o)  \le d(gba_{su} o, gba_{tu} o) + d(gba_{tu} o, kba_{tu} o) \\
        & \le 1 + d(gba_{tu} o, gba_{tu} \tilde{a} o) + d(gba_{tu} \tilde{a} o, kba_{tu} \tilde{a} o) + d(kba_{tu} \tilde{a} o, kba_{tu} o) \\
        & = 1 + 2d(o, \tilde{a} o) + d(gba_{tu} \tilde{a} o, kba_{tu} \tilde{a} o).
    \end{aligned}$$
    Since $g^{-1}k \in P$ and $ba_{tu} \tilde{a} \in A^+$, we get $d(gba_{tu} \tilde{a} o, kba_{tu} \tilde{a} o) \le d(g o, ko) = d(g o, o) < c \ell_R + c'$. Together with $\|\log \tilde{a} \| < D + r$, we have $$d(gba_{su} o, kba_{tu} o) < 1 + 2(D + r) + c \ell_R + c'.$$ Since $d(kba_{tu} o, \ga o) < D + r + R$, we finally have $$d(gba_{su} o, \ga o) < 1 + 3(D+r) + R + c \ell_R + c'.$$
    We set $R' = 1 + 3(D+r) + R + c \ell_R + c'$ and $Q := \{[h] \in \tilde{\Omega}_{\theta} : d(ho, o) \le R'\}$ which is a compact subset of $\tilde{\Omega}_{\theta}$.

    Now the image of $g$ under the projection $G \to \F_{\theta}^{(2)} \times \fa_{\theta}$ is of the form $(\xi, \eta, v)$ for some $v \in \fa_{\theta}$. Since $b \in S_{\theta}$, the product $gb$ also projects to the same element $(\xi, \eta, v)$. It follows from $d(o, go) < c \ell_R + c' \le R'$ that $(\xi, \eta, v) \in Q$. Moreover, since $d(\ga^{-1}gba_{su} o, o) < R'$ for all $s \in [t-1, t+1]$, we have $\ga^{-1}(\xi, \eta, v)a_{su} \in Q$ and hence $(\xi, \eta, v)a_{[t-1, t+1]} \subset \ga Q$. This finishes the proof.
\end{proof}

Recall the notation $\delta = \psi(u) > 0$.

\begin{lem} \label{lem.lower2}
Fix $r, R > 0$, and let $Q = Q(r,R) \subset \tilde \Omega_{\theta}$ and $C_1 = C_1(Q) > 0$ be 
as in Lemma \ref{lem.findsegment} and  Proposition \ref{lem.dircartan} respectively.  Let $T>0$ and $\ga \in \Ga_{u, r}$  be such that
$$\text{ $\|\mu_{\theta}(\ga) \|  > \ell_R\;\; $ and $\;\; C_1\|\psi\| + \delta < \psi(\mu_{\theta}(\ga)) < \delta T -  C_1 \|\psi\| - \delta$. }$$
If $\sum_{\ga_0 \in \Ga_{u, r}} e^{-\psi(\mu_{\theta}(\ga_0))} = \infty$, then, for any $(\xi, \eta) \in  (O_{R}^{\theta}(o, \ga o) \times O_{R}^{\i(\theta)}(\ga o, o) ) \cap \La_{\theta}^{(2)}$,  we have
    $$\int_0^T\int_{\fa_{\theta}} \mathbbm{1}_{Q' \cap \ga Q' a_{-t}}(\xi, \eta, b) db dt \ge 2 \op{Vol}(A_{\theta, 2}) $$ 
    where $A_{\theta, 2} = \{a \in A_{\theta} : \| \log a \| \le 2\}$ and $Q':=Q A_{\theta, 2} \subset \tilde{\Omega}_{\theta}$.
\end{lem}

\begin{proof}
    By Lemma \ref{lem.findsegment}, there exist $v \in \fa_{\theta}$ and  $t_0 \ge 0$ such that $(\xi, \eta, v) \in Q$ and $(\xi, \eta, v)a_{[t_0 - 1, t_0 + 1]} \subset \ga Q$. In other words, $(\xi, \eta, v) \in Q \cap \ga Q a_{-t}$ for all $t \in [t_0 - 1, t_0 + 1]$. Since $\|\mu_{\theta}(\ga) - t_0 u \| < C_1$ by Proposition \ref{lem.dircartan}(1), we have $|\psi(\mu_{\theta}(\ga)) - t_0 \delta| < C_1 \|\psi\|$. In particular, we have $[t_0 - 1, t_0 + 1] \subset [0, T]$ by the hypothesis. 

    We set $Q' := Q A_{\theta, 2}$ which is a compact subset of $\tilde{\Omega}_{\theta}$. We then have for each $t \in [t_0 - 1, t_0 + 1]$ that
    $$\int_{A_{\theta}} \mathbbm{1}_{Q' \cap \ga Q' a_{-t}}((\xi, \eta, v)b) db \ge \int_{A_{\theta, 2}} \mathbbm{1}_{\ga Q'}((\xi, \eta, v)ba_t)db \ge \op{Vol}(A_{\theta, 2})$$ where the last inequality follows from $(\xi, \eta, v)a_t \in \ga Q$. Therefore, we have $$\begin{aligned}
    \int_0^T \int_{\fa_{\theta}} \mathbbm{1}_{Q' \cap \ga Q' a_{-t}}(\xi, \eta, b) db dt & = \int_0^T \int_{A_{\theta}} \mathbbm{1}_{Q' \cap \ga Q' a_{-t}}((\xi, \eta, v)b) db dt \\
    & \ge \int_{t_0-1}^{t_0+1} \int_{A_{\theta}} \mathbbm{1}_{Q' \cap \ga Q' a_{-t}}((\xi, \eta, v)b) db dt \\
    & \ge 2 \op{Vol}(A_{\theta, 2})
    \end{aligned}$$ as desired.
\end{proof}

\subsection*{Proof of Proposition \ref{prop.lowerbound}(2)}
    Fix $R > \max (R(\nu), R(\nu_{\i}))$ where $R(\nu), R(\nu_{\i})$ are defined as in Lemma  \ref{lem.shadowlemma}.
    Let $Q'=Q(r, R) A_{\theta, 2}$
    where $Q(r,R)$ is given  in Lemma \ref{lem.findsegment}, so that  $Q'$ satisfies the conclusion of
    Lemma \ref{lem.lower2}. For any $\ga \in \Ga$ and $t > 0$, we have
    $$\begin{aligned}
        &\tilde{\m}(Q' \cap \ga Q' a_{-t}) \\
        & = \int_{\F_{\theta}^{(2)}}\left(\int_{\fa_\theta} \mathbbm{1}_{Q' \cap \ga Q' a_{-t}}(\xi, \eta, b) db \right) e^{\psi(\cal G^{\theta}(\xi, \eta))} d\nu(\xi)d \nu_{\i}(\eta) \\
        & \ge \int_{O_{R}^{\theta}(o, \ga o) \times O_{R}^{\i(\theta)}(\ga o, o)} \left(\int_{\fa_\theta} \mathbbm{1}_{Q' \cap \ga Q' a_{-t}}(\xi, \eta, b) db \right) e^{\psi(\cal G^{\theta}(\xi, \eta))} d\nu(\xi)d \nu_{ \i}(\eta).
    \end{aligned}$$
   
    By Lemma \ref{lem.lower2}, if $\ga \in \Ga_{u, r}$, $\| \mu_{\theta}(\ga) \| > \ell_R$ and $C_1\|\psi\| + \delta < \psi(\mu_{\theta}(\ga)) < \delta T -  C_1 \|\psi\| - \delta$ where $C_1 = C_1(Q)$, then $$\begin{aligned}
        &\int_0^T  \tilde{\m}(Q' \cap \ga Q' a_{-t}) dt \\
        & \ge 2 \op{Vol}(A_{\theta, 2}) \int_{ O_{R}^{\theta}(o, \ga o) \times O_{R}^{\i(\theta)}(\ga o, o)} e^{\psi(\cal G^{\theta}(\xi, \eta))} d\nu(\xi)d \nu_{\i}(\eta) \\
        & \ge 2 \op{Vol}(A_{\theta, 2}) e^{-\|\psi\| \ell_R} \nu( O_{R}^{\theta}(o, \ga o)) \nu_{ \i}( O_{R}^{\i(\theta)}(\ga o, o))
    \end{aligned}$$ where the last inequality follows from $\|\cal G^{\theta}(\xi, \eta) \| < \ell_R$. By Lemma \ref{lem.shadowlemma}, we conclude $$\int_0^T  \tilde{\m}(Q' \cap \ga Q' a_{-t}) dt \gg e^{-\psi(\mu_{\theta}(\ga))}.$$
    For each $T\ge 1$, we define $$\Ga_T = \{ \ga \in \Ga :  \| \mu_{\theta}(\ga) \|  > \ell_R, C_1 \|\psi\| + \delta < \psi(\mu_{\theta}(\ga)) < \delta T - (C_1 \|\psi\| + \delta) \}.$$ Since both $ \{ \ga \in \Ga : \|\mu_{\theta}(\ga) \|  \le \ell_R \}$ and $\{ \ga \in \Ga : \psi(\mu_{\theta}(\ga)) \le C_1 \|\psi\| + \delta \}$ are finite sets, we have $$\begin{aligned}
        \int_0^T \sum_{\ga \in \Ga} \tilde{\m}(Q' \cap \ga Q' a_{-t}) dt & \ge \int_0^T \sum_{\ga \in \Ga_{u, r} \cap \Ga_T} \tilde{\m}(Q' \cap \ga Q' a_{-t}) dt \\
        & \gg \sum_{\ga \in \Ga_{u, r} \cap \Ga_T} e^{-\psi(\mu_{\theta}(\ga))} \\
        & \gg \sum_{\substack{\ga \in \Ga_{u, r} \\ \psi(\mu_{\theta}(\ga)) < \delta T - (C_1 \| \psi \| + \delta)}} e^{-\psi(\mu_{\theta}(\ga))}.
    \end{aligned}$$
    By Lemma \ref{lem.shadowlemma} and Proposition \ref{prop.mult}, $$\sum_{\substack{\ga \in \Ga_{u, r} \\  \delta T - (C_1 \| \psi \| + \delta) \le  \psi(\mu_{\theta}(\ga)) \le \delta T}} e^{-\psi(\mu_{\theta}(\ga))} \ll \sum_{\substack{\ga \in \Ga_{u, r} \\  \delta T - (C_1 \| \psi \| + \delta) \le  \psi(\mu_{\theta}(\ga)) \le \delta T}} \nu (O_{R}^{\theta}(o, \ga o)) \ll 1.$$ Therefore, we obtain $$\int_0^T \sum_{\ga \in \Ga} \tilde{\m}(Q' \cap \ga Q' a_{-t}) dt  \gg \sum_{\substack{\ga \in \Ga_{u, r} \\ \psi(\mu_{\theta}(\ga)) \le \delta T}} e^{-\psi(\mu_{\theta}(\ga))}.$$ \qed

\medskip

We will apply the following version of Borel-Cantelli lemma.
\begin{lem} \cite[Lem. 2]{AS_rational} \label{lem.BC}
Let $(\Omega, \mathsf{M})$ be a finite Borel measure space and $\{P_t : t \ge 0 \}$ be a collection of subsets of $\Omega$ such that the map
$(t, \omega) \mapsto \mathbbm{1}_{P_t}(\omega)$ is measurable on $\br_+\times \Omega$. Suppose that \begin{enumerate}
    \item $\int_0^{\infty} \mathsf{M}(P_t) dt = \infty$, and
    \item for all large enough $T$, $$\int_0^T \int_0^T \mathsf{M}(P_t \cap P_s) dt ds \ll \left( \int_0^T \mathsf{M}(P_t) dt \right)^2$$ where the implied constant is independent of $T$.
\end{enumerate}
Then we have $$\mathsf{M}\left( \left\{ \omega \in \Omega : \int_0^{\infty} \mathbbm{1}_{P_t}(\omega) dt = \infty \right\} \right) > 0.$$
\end{lem}

\subsection*{Proof of Theorem \ref{prop.mainprop}}
    Let $Q\subset \tilde \Omega_\theta$ be a compact subset with $\tilde{\m}(Q) > 0$. Let $r = r(Q)>1$ be large enough so that $\sum_{\ga \in \Ga_{u, r}} e^{-\psi(\mu_{\theta}(\ga))} = \infty$ and that Proposition \ref{prop.lowerbound}(1) holds. 
    Let $Q'=Q'(r)$ be a compact subset of $\tilde{\Omega}_{\theta}$
    given by Proposition \ref{prop.lowerbound}(2). Replacing $Q'$ with a larger compact subset if necessary, we may assume that $\tilde{\m}(Q') > 0$.

    Since $\m$ is $u$-balanced, we have for $T > 1$ that
    \be \label{eqn.applybalanced}
    \int_0^T \sum_{\ga \in \Ga} \tilde{\m}(Q \cap \ga Q a_{-t}) dt \asymp \int_0^T \sum_{\ga \in \Ga} \tilde{\m}(Q' \cap \ga Q' a_{-t}) dt
    \ee with the implied constant independent of $T$. Since we already have $$\int_0^T \int_0^T \sum_{\ga, \ga' \in \Ga} \tilde{\m}(Q \cap \ga Q a_{-t} \cap \ga' Q a_{-t-s}) dt ds \ll \left( \sum_{\substack{\ga \in \Ga_{u, r} \\ \psi(\mu_{\theta}(\ga)) \le \delta T}} e^{-\psi(\mu_{\theta}(\ga))} \right)^2$$ and 
    \be \label{eqn.applyprop74}
    \sum_{\substack{\ga \in \Ga_{u, r} \\ \psi(\mu_{\theta}(\ga)) \le \delta T}} e^{-\psi(\mu_{\theta}(\ga))} \ll \int_0^T \sum_{\ga \in \Ga} \tilde{\m}(Q' \cap \ga Q' a_{-t}) dt
    \ee by Proposition \ref{prop.lowerbound}, it follows from \eqref{eqn.applybalanced} that \be \label{eqn.insteadcor}
    \int_0^T \int_0^T \sum_{\ga, \ga' \in \Ga} \tilde{\m}(Q \cap \ga Q a_{-t} \cap \ga' Q a_{-t-s}) dt ds \ll \left( \int_0^T \sum_{\ga \in \Ga} \tilde{\m}(Q \cap \ga Q a_{-t}) dt \right)^2.
    \ee
    
    By abusing notation, for a subset $U \subset \tilde{\Omega}_{\theta}$, we denote by $[U]$ the image of $U$ under the projection $\tilde{\Omega}_{\theta} \to \Omega_{\theta}$, i.e., $[U] = \Ga \ba \Ga U$. We set $\mathsf{M} = \m|_{[Q]}$ which is a finite Borel measure. We let $P_t = [ Q \cap \Ga Q a_{-t}]$ for $t \ge 0$. Since $\# \{\ga \in \Ga : Q a_{-t} \cap \ga Q a_{-t} \neq \emptyset \}$ is uniformly bounded independent of $t$, we have $\mathsf{M}(P_t) \asymp \sum_{\ga \in \Ga} \tilde{\m}(Q \cap \ga Q a_{-t})$ with the implied constant independent of $t$. Noting that $\sum_{\ga \in \Ga_{u, r}} e^{-\psi(\mu_{\theta}(\ga))} = \infty$, it follows from \eqref{eqn.applybalanced} and \eqref{eqn.applyprop74} that $$\int_0^{\infty} \mathsf{M}(P_t) dt = \infty$$ and hence the condition (1) in Lemma \ref{lem.BC} is satisfied.

    The following is a rephrase of \eqref{eqn.insteadcor}: $$\int_0^T \int_0^T \mathsf{M}(P_t \cap P_{t + s}) ds dt \ll \left( \int_0^T \mathsf{M}(P_t) dt \right)^2.$$ It implies $$\begin{aligned}
        \int_0^T \int_0^T \mathsf{M}(P_t \cap P_s) dsdt & = 2 \int_0^T \int_t^T \mathsf{M}(P_t \cap P_s)ds dt \\
        & \le 2 \int_0^T \int_0^T \mathsf{M}(P_t \cap P_{t+s}) ds dt \\
        & \ll  \left( \int_0^T \mathsf{M}(P_t) dt \right)^2,
    \end{aligned}$$ showing that the condition (2) in Lemma \ref{lem.BC} is satisfied.

    Hence, by Lemma \ref{lem.BC}, we have $$\mathsf{M}\left( \left\{ [(\xi, \eta, v)] \in [Q] : \int_0^{\infty} \mathbbm{1}_{[Q]}([(\xi, \eta, v)]a_t) dt = \infty \right\} \right) > 0.$$
    In particular, there exists a subset $Q_0 \subset Q$ such that $\tilde{\m}({Q}_0) > 0$ and for all $(\xi, \eta, v) \in {Q}_0$, there exist sequences $\ga_i \in \Ga$ and $t_i \to \infty$ such that $\ga_i^{-1} (\xi, \eta, v) a_{t_i} \in Q$ for all $i \ge 1$.  In particular, $$(\xi, \eta, v) \in Q \cap \ga_i Q a_{-t_i} \quad \text{for all } i \ge 1,$$ which implies $\xi \in \La_{\theta}^{u}$ by Lemma \ref{defdir}.
   Since this holds for all $(\xi, \eta, v) \in {Q}_0$, we have that
    $$\xi \in \La_{\theta}^u \quad \text{for all } (\xi, \eta, v) \in Q_0.$$  Since $\tilde{\m}({Q}_0) > 0$ and $\tilde{\m} $ is equivalent to the product measure $ \nu \otimes \nu_{ \i} \otimes db$, it follows that $\nu(\La_{\theta}^u) > 0$ as desired.
    Since $\m$ is $A_u$-invariant, the $u$-balanced condition remains same after changing the sign of $T$. Then the same argument with the negative $T$ gives $\nu_{\i}(\La_{\i(\theta)}^{\i(u)}) > 0$.
\qed

\begin{lem} \label{lem.zeroone}
    We have either $$
\nu (\La_{\theta}^u) = 0 \quad \text{or} \quad \nu(\La_{\theta}^u) = 1.$$
\end{lem}

\begin{proof}
Suppose that $\nu(\La_{\theta}^u) > 0$. Then by Theorem \ref{thm.convergencedir}, we must have $\sum_{\ga \in \Ga_{u, r}} e^{-\psi(\mu_{\theta}(\ga))} = \infty$ for some $r > 0$. This implies that $\nu$ is the unique $(\Ga, \psi)$-conformal measure on $\F_{\theta}$ (\cite{CZZ_transverse}, \cite[Thm. 1.5]{KOW_indicators}). On the other hand, if $0<\nu(\La_{\theta}^u) <1$, then $\tilde{\nu} := \frac{1}{\nu(\F_{\theta} - \La_{\theta}^u)} \nu|_{\F_{\theta} - \La_{\theta}^u}$ defines another $(\Ga, \psi)$-conformal measure, which would contradict the uniqueness of the $(\Ga, \psi)$-conformal measure. Therefore, $\nu(\La_{\theta}^u)$ must be either $0$ or $1$.    
\end{proof}

\begin{corollary} \label{thm.directionalpoincare} 
If
 $\m$ is $u$-balanced, the following are equivalent:
\begin{enumerate}
    \item 
$\sum_{\ga \in \Ga_{u, r}} e^{-\psi(\mu_{\theta}(\ga))} = \infty$ for some $r > 0$;
\item $\nu(\La_{\theta}^u) = 1 = \nu_{\i}(\La_{\i(\theta)}^{\i(u)}).$ 
\end{enumerate}
Similarly, if $\m$ is $u$-balanced, the following are also equivalent:
\begin{enumerate}
    \item 
 $\sum_{\ga \in \Ga_{u, r}} e^{-\psi(\mu_{\theta}(\ga))} <\infty$ for all $r > 0$;

\item  $\nu(\La_{\theta}^u) = 0= \nu_{\i}(\La_{\i(\theta)}^{\i(u)}).$
\end{enumerate}
\end{corollary}

\begin{proof}
By Lemma \ref{lem.zeroone}, we have
$\nu (\La_{\theta}^u) = 0 $ or $ \nu(\La_{\theta}^u) = 1.$
Similarly, noting that $\psi \circ \i \in \fa_{\i(\theta)}^*$ is $(\Ga, \i(\theta))$-proper as well, we also have either 
$\nu_{\i}(\La_{\i(\theta)}^{\i(u)}) = 0 $ or $\nu_{\i}(\La_{\i(\theta)}^{\i(u)}) = 1.$
Therefore  Theorem \ref{prop.mainprop} implies that 
if $\sum_{\ga \in \Ga_{u, r}} e^{-\psi(\mu_{\theta}(\ga))} = \infty$ for some $r > 0$, then $\nu(\La_{\theta}^u) = 1 =  \nu_{\i}(\La_{\i(\theta)}^{\i(u)})$.
On the other hand Theorem \ref{thm.convergencedir} implies that
if $\sum_{\ga \in \Ga_{u, r}} e^{-\psi(\mu_{\theta}(\ga))} < \infty$ for all $r > 0$, then $\nu(\La_{\theta}^u) = 0 =  \nu_{\i}(\La_{\i(\theta)}^{\i(u)})$. This proves the corollary.
\end{proof}

We finish the section with the following corollary of Proposition \ref{prop.lowerbound}, which will be used later.
The following estimate reduces the divergence
    of the series $\sum_{\ga \in \Ga_{u, r} } e^{-\psi(\mu_{\theta}(\ga))}$ to the local mixing rate for the $a_t$-flow:
\begin{cor} \label{useful}  For all sufficiently large $r>0$, there exist compact subsets $Q_1, Q_2$ of $\Omega_\theta$ with non-empty interior such that for all $T\ge 1$,
$$\left(\int_0^T  {\m}(Q_1 \cap  Q_1 a_{-t}) dt \right)^{1/2}
\ll  \sum_{\substack{\ga \in \Ga_{u, r} \\ \psi(\mu_{\theta}(\ga)) \le \delta T}} e^{-\psi(\mu_{\theta}(\ga))} \ll \int_0^T  {\m}(Q_2 \cap Q_2 a_{-t}) dt  .$$
\end{cor}
\begin{proof}
Let $Q \subset \tilde{\Omega}_{\theta}$ be a compact subset with non-empty interior. By Proposition \ref{prop.lowerbound}(1), there exists $r_0 = r_0(Q) > 0$ such that  for all $T \ge  1$ and for all $r\ge r_0$,
    \be \label{eqn.applyupperbound}
    \int_0^T \int_0^T \sum_{\ga, \ga' \in \Ga} \tilde{\m}(Q \cap \ga Q a_{-t} \cap \ga' Q a_{-t-s}) dtds \ll \left( \sum_{\substack{\ga \in \Ga_{u, r} \\ \psi(\mu_{\theta}(\ga)) \le \delta T}} e^{-\psi(\mu_{\theta}(\ga))} \right)^2.
    \ee
    Fix a small $\varepsilon > 0$ so that $Q^- := \bigcap_{0 \le s \le \varepsilon} Qa_{-s}$ has non-empty interior. Since we have
    $$\begin{aligned}
        \varepsilon \int_0^T \sum_{\ga \in \Ga} & \tilde{\m}(Q^- \cap \ga Q^- a_{-t}) dt 
         \le \int_0^T \int_0^{\varepsilon} \sum_{\ga \in \Ga} \tilde{\m}(Q \cap \ga(Q \cap Q a_{-s}) a_{-t}) ds dt,
    \end{aligned}$$ it follows from \eqref{eqn.applyupperbound} that for all $r\ge r_0$, $$\int_0^T \sum_{\ga \in \Ga}  \tilde{\m}(Q^- \cap \ga Q^- a_{-t}) dt \ll \left( \sum_{\substack{\ga \in \Ga_{u, r} \\ \psi(\mu_{\theta}(\ga)) \le \delta T}} e^{-\psi(\mu_{\theta}(\ga))} \right)^2.$$ 
 Now let $Q' = Q'(r) \subset \tilde{\Omega}_{\theta}$ be a compact subset given in Proposition \ref{prop.lowerbound}(2) such that for any $T > 1$,
 \be\label{div2} \int_0^T \sum_{\ga \in \Ga} \tilde{\m}(Q' \cap \ga Q' a_{-t}) dt  \gg \sum_{\substack{\ga \in \Ga_{u, r} \\ \psi(\mu_{\theta}(\ga)) \le \delta T}} e^{-\psi(\mu_{\theta}(\ga))}.
     \ee
     Replacing $Q'$ with a larger compact subset, we may assume that $\inte Q'\ne \emptyset $. 
Hence it suffices to set $Q_1=\Ga\ba \Ga Q^-$ and $Q_2=\Ga\ba \Ga Q'$ to finish the proof.
\end{proof}

\begin{rmk}\label{bllormk} For $\theta=\Pi$,
Corollary \ref{useful} was established
in \cite{BLLO} for any Zariski dense discrete subgroup of $G$ (see \cite[Proof of
Thm. 6.3]{BLLO}). 
If $\Ga$ is a lattice of $G$, 
then, together with the Howe-Moore mixing property of the (finite) Haar measure \cite{HM}, it implies that for any non-zero $u\in \fa^+$, we have $\sum_{\ga \in \Ga_{u, r}}  e^{-2\rho (\mu(\ga))}=\infty$ for all $r>1$ large enough where $2\rho$ denotes the sum of all positive roots
counted with multiplicity. 
\end{rmk}

\section{Transitivity subgroup and ergodicity of directional flows} \label{sec.conserg}
In this section, we complete the proof of Theorem \ref{main}, by establishing the equivalence between co-nullity of directional conical sets and conservativity/ergodicity of directional flows. We use the notion of transitivity subgroup to carry out the Hopf argument in our setting.

\medskip

Let $\Ga < G$ be a Zariski dense $\theta$-transverse subgroup. We fix a non-zero vector $u \in  \fa_{\theta}^+$ and  a $(\Ga, \theta)$-proper linear form
    $\psi\in \fa_\theta^*$. We also fix a pair  $\nu, \nu_{\i}$ of $(\Ga, \psi)$ and $(\Ga, \psi\circ~\i)$-conformal measures on $\La_\theta$ and $\La_{\i(\theta)}$ respectively. Denote by $\m=\m(\nu, \nu_{\i})$  the  associated BMS measure on  $\Omega_\theta$.
     In this section, we discuss the ergodicity and conservativity of the directional flow 
     $$A_u=\{a_t := \exp (tu) : t \in \R\}$$ on $\Omega_{\theta}$ with respect to $\m$. We emphasize that the notion of a transitivity subgroup plays a key role in showing the $A_u$-ergodicity.

\subsection*{Conservativity of directional flows}
Recall the following definitions:
\begin{enumerate}
    \item A Borel subset $B \subset \Omega_{\theta}$ is called a wandering set for $\m$ if for $\m$-a.e. $x \in B$, we have $\int_{-\infty}^{\infty} \mathbbm{1}_{B}(xa_t) dt < \infty$.
    \item We say that $(\Omega_{\theta}, A_u, \m)$ is completely conservative if there is no wandering set $B \subset \Omega_{\theta}$ with $\m(B) > 0$.
    \item We say that $(\Omega_{\theta}, A_u, \m)$ is completely dissipative if $\Omega_{\theta}$ is a countable union of wandering sets modulo $\m$.
\end{enumerate}

The following is proved for $\theta = \Pi$ in \cite[Prop. 4.2]{BLLO} and a similar proof works for general $\theta$:

\begin{prop} \label{prop.conserv}
    The flow $(\Omega_{\theta}, A_u, \m)$ is completely conservative (resp. completely dissipative) if and only if $\max \left(\nu(\La_{\theta}^u), \nu_{\i}(\La_{\i(\theta)}^{\i(u)}) \right) > 0$ (resp. $\nu(\La_{\theta}^u) = 0 = \nu_{\i}(\La_{\i(\theta)}^{\i(u)})$).
\end{prop}

\begin{proof}
Suppose that there exists a non-wandering subset $B$ with $\m(B) > 0$.   Setting $B^{\pm} := \{ x \in B : \limsup_{t \to \pm \infty} xa_t \cap B \neq \emptyset\}$, we have $\m(B^+ \cup B^-) > 0$. Since $\m$ is locally equivalent to $\nu \otimes \nu_{\i} \otimes db$, if we have $\m(B^+) > 0$, then $\nu(\La_{\theta}^u) > 0$ by Lemma \ref{defdir}. Otherwise, if $\m(B^-) > 0$, then $\nu_{\i}(\La_{\i(\theta)}^{\i(u)}) > 0$. 
    It shows the following two implications: $$\begin{aligned}
        (\Omega_{\theta}, A_u, \m) \text{ is completely conservative} & \Rightarrow \max \left( \nu(\La_{\theta}^u), \nu_{\i}(\La_{\i(\theta)}^{\i(u)}) \right) > 0;\\
        (\Omega_{\theta}, A_u, \m) \text{ is completely dissipative} & \Leftarrow \nu(\La_{\theta}^u) = 0 = \nu_{\i}(\La_{\i(\theta)}^{\i(u)})
    \end{aligned}$$ where the second implication is due to the $\sigma$-compactness of $\Omega_{\theta}$. 
    
    Now suppose that $\nu (\La_{\theta}^u) > 0$ (resp. $\nu_{\i}(\La_{\i(\theta)}^{\i(u)} > 0)$. By Theorem \ref{thm.convergencedir}, $\sum_{\ga \in \Ga_{u, r}} e^{-\psi(\mu_{\theta}(\ga))} = \infty$ (resp. $\sum_{\ga \in \Ga_{u, r}^{-1}} e^{-(\psi \circ \i)(\mu_{\i(\theta)}(\ga))} = \infty$) for some $r > 0$. Note that $\ga \in \Ga_{u, r}^{-1}$ if and only if $\|\mu_{\i(\theta)}(\ga) - t \i(u)\| < r$ for some $t \ge 0$. Hence it follows from \eqref{lem.zeroone} that $\nu(\La_{\theta}^u) = 1$ (resp. $\nu_{\i}(\La_{\i(\theta)}^{\i(u)}) = 1$). It implies that for $\m$-a.e. $\Ga[g] \in \Omega_{\theta}$, we have $g^+ \in \La_{\theta}^u$ (resp. $g^- \in \La_{\i(\theta)}^{\i(u)}$) and hence $\Ga[g]a_{t_iu}$ is a convergent sequence for some sequence $t_i \to \infty$ (resp. $t_i \to -\infty$).  Hence,  for $\m$-a.e. $x \in \Omega_{\theta}$, there exists a compact subset $B$ such that $\int_{-\infty}^{\infty} \mathbbm{1}_B(xa_t) dt = \infty$. This implies the conservativity of $(\Omega_{\theta}, A_u, \m)$ by \cite[Lem. 6.1]{LO_dichotomy}.
\end{proof}

\subsection*{Density of $\theta$-transitivity subgroups}
\begin{definition}[$\theta$-transitivity subgroup]
    For $g \in G$ with  $(g^+, g^-) \in \La_{\theta}^{(2)}$, define  $\cal H_{\Ga}^{\theta}(g) $ to be the set of all elements $a\in A_\theta$ such that there exist $\ga \in \Ga$, $s \in S_{\theta}$ and a sequence $n_1, \cdots, n_k \in N_{\theta} \cup \check{N}_{\theta}$ such that

    \begin{enumerate}
    \item $((gn_1 \cdots n_r)^+ , (gn_1 \cdots n_r)^-) \in \La_{\theta}^{(2)}$ for all $1 \le r \le k$; and
    \item $\ga g n_1 \cdots n_k = gas$.
    \end{enumerate}
    It is not hard to see that $\cal H_{\Ga}^{\theta}(g)$ is a subgroup (cf. \cite[Lem. 3.1]{Winter_mixing}).
   
\end{definition}

We deduce the density of transitive subgroups from Theorem \ref{thm.Benoistdense}:

\begin{prop} \label{prop.transitivity}
    For any $g \in G$ with $(g^+, g^-) \in \La_{\theta}^{(2)}$, the subgroup $\cal H_{\Ga}^{\theta}(g)$ is dense in $A_{\theta}$.
\end{prop}

\begin{proof}  Note that we have a Zariski dense open subset $g\check{N}_{\theta} P_{\theta} /P\subset \F$; this is well-defined since $P\subset P_\theta$.
Hence  there exists a Zariski dense Schottky subgroup $\Ga_0 < \Ga$ so that for any loxodromic element $\ga \in \Ga_0$, its attracting fixed point $y_{\ga}$ belongs to $g \check{N}_{\theta} P_{\theta}$ (cf. \cite[Lem. 7.3]{ELO_anosov}, \cite{Benoist1997proprietes}). Note that any non-trivial element of $\Ga_0$ is loxodromic.
By Theorem \ref{thm.Benoistdense}, it suffices to prove:
\be\label{ccc} \{p_{\theta}(\la(\ga)) : \ga \in \Ga_0 \} \subset \log \cal H_{\Ga}^{\theta}(g).\ee 

Fixing any non-trivial element $\ga \in \Ga_0$, write $\ga = h a_{\ga} m h^{-1} \in h A^+M h^{-1}$ for some $h \in G$. Then $\la(\ga)=\log a_\ga$ and $y_\ga= hP\in \La$; hence $ y_{\ga}^{\theta}:=h P_{\theta}  \in g\check{N}_{\theta}P_{\theta}$. Using $P_\theta=N_\theta A_\theta S_\theta$, we can write $h \in g\check{n} n A_{\theta}S_{\theta}$ for some $\check{n} \in \check{N}_{\theta}$ and $n \in N_{\theta}$. By replacing $h$ with $g\check{n}n$, we may assume that $$h = g \check{n} n \in g \check{N}_{\theta}N_{\theta} \quad \text{and} \quad \ga = h a s h^{-1}$$ for some $s \in S_{\theta}$ where $a$ is the $A_{\theta}$-component of $a_\ga$ in the decomposition $a_\ga \in A_{\theta}^+B_{\theta}^+$ so that $p_\theta(\log a_\ga)=
\log a$.
It remains to show that $a \in \cal H_{\Ga}^{\theta}(g)$. We first note  from $\ga = hash^{-1}$ and $h = g\check{n}n$ that $$\ga = (gas) \left( (as)^{-1} \check{n} (as) \right) \left( (as)^{-1} n (as) \right) n^{-1} \check{n}^{-1} g^{-1}$$ and hence \be \label{eqn.cond2}
    \ga g \check{n} n \left( (as)^{-1} n^{-1} (as) \right) \left( (as)^{-1} \check{n}^{-1} (as) \right) = gas.
    \ee
    Writing $n_1 = \check{n}$, $n_2 = n$, $n_3 = (as)^{-1} n^{-1} (as)$ and $n_4 = (as)^{-1} \check{n}^{-1} (as)$, we have $n_1, n_4 \in \check{N}_{\theta}$ and $n_2, n_3 \in N_{\theta}$. By \eqref{eqn.cond2}, the elements $n_i$, $1\le i\le 4$, satisfy the second condition for $a \in \cal H_{\Ga}^{\theta}(g)$. We now check the first condition:
    \begin{itemize}
        \item $gn_1P_{\theta} = g\check{n}P_{\theta} = hP_{\theta} = y_{\ga}^{\theta} \in \La_{\theta}$ and $gn_1w_0 P_{\i(\theta)} = g w_0 P_{\i(\theta)} \in \La_{\i(\theta)}$;
        \item $gn_1n_2P_{\theta} = hP_{\theta} \in \La_{\theta}$ and $gn_1n_2w_0 P_{\i(\theta)} = h w_0 P_{\i(\theta)} = y_{\ga^{-1}}^{\i(\theta)} \in \La_{\i(\theta)}$;
        \item $gn_1 n_2 n_3 P_{\theta} = g n_1 n_2 P_{\theta} \in \La_{\theta}$ and $g n_1 n_2 n_3 w_0 P_{\i(\theta)} = \ga^{-1} gas n_4^{-1} w_0 P_{\i(\theta)} = \ga^{-1} gas w_0 P_{\i(\theta)} = \ga^{-1} g w_0 P_{\i(\theta)} \in \La_{\i(\theta)}$ by \eqref{eqn.cond2};
        \item $gn_1n_2n_3n_4 P_{\theta} = \ga^{-1} gas P_{\theta} = \ga^{-1} g P_{\theta} \in \La_{\theta}$ and $gn_1n_2n_3n_4 w_0 P_{\i(\theta)} = gn_1n_2n_3 w_0 P_{\i(\theta)} \in \La_{\i(\theta)}$.
    \end{itemize}
This proves that $a \in \cal H_{\Ga}^{\theta}(g)$ and completes the proof. 
\end{proof}

\subsection*{Stable and unstable foliations for directional flows} 
Recall the notation that for $g \in G$, we set
$$[g] = (g^+, g^-, \beta_{g^+}^{\theta}(e, g)) \in \F_{\theta}^{(2)} \times \fa_{\theta}.$$

\begin{lemma} \label{lem.transbyhoro}
    Let $g \in G$, $n \in N_{\theta}$ and $\check{n} \in \check{N}_{\theta}$. Then $$\begin{aligned}
        [gn] & = (g^+, (gn)^-, \beta_{g^+}^{\theta}(e, g)); \\
        [g\check{n}] & = ((g\check{n})^+, g^-, \beta_{g^+}^{\theta}(e, g) + \cal G^{\theta}((g\check{n})^+, g^-) - \cal G^{\theta}(g^+, g^-)).
    \end{aligned}$$
\end{lemma}

\begin{proof}
    Since $(gn)^+ = gnP_{\theta} = g P_{\theta}$, we have $$\beta_{(gn)^+}^{\theta}(e, gn) - \beta_{g^+}^{\theta}(e, g) = \beta_{e^+}^{\theta}(e, n) = 0$$ and therefore $[gn] = (g^+, (gn)^-, \beta_{g^+}^{\theta}(e, g))$.  To see the second identity, we first note that $g\check{n} w_0 P_{\i(\theta)} = gw_0 P_{\i(\theta)}$, that is,
     $(g\check{n})^- = g^-$. 
    Since $\beta_{e^-}^{\i(\theta)}(e, \check{n}) = 0$, we have $$\begin{aligned}
        \cal G^{\theta}((g\check{n})^+, g^-)
        & = \beta_{(g\check{n})^+}^{\theta}(e, g\check{n}) + \i(\beta_{g^-}^{\i(\theta)}(e, g)) + \i(\beta_{e^-}^{\i(\theta)}(e, \check{n}))\\
        & = \beta_{(g\check{n})^+}^{\theta}(e, g\check{n}) + \i(\beta_{g^-}^{\i(\theta)}(e, g)).
        \end{aligned}$$
    Since   $\cal G^{\theta}(g^+, g^-) = \beta_{g^+}^{\theta}(e, g) + \i(\beta_{g^-}^{\i(\theta)}(e, g))$,  we get
    $$\beta_{(g\check{n})^+}^{\theta}(e, g\check{n}) = \beta_{g^+}^{\theta}(e, g) + \cal G^{\theta}((g\check{n})^+, g^-) - \cal G^{\theta}(g^+, g^-)$$
proving the second identity.
\end{proof}

We say a metric $\d$ on $\Omega_\theta$ {\it admissible} if it extends to a metric of the one-point compactification of $\Omega_\theta$ (if $\Omega_\theta$ is compact, any metric is admissible).
Since $\Omega_{\theta}$ is a second countable locally compact Hausdorff space (Theorem \ref{thm.propdisckow}), there exists an admissible metric.

For $x \in \Omega_{\theta}$, we define $W^{ss}(x)$ (resp. $W^{su}(x)$)
    to be the set of all $y \in \Omega_{\theta}$ such that $\d(xa_t, ya_t) \to 0$ as $t \to + \infty$ (resp. $t\to -\infty$). They form strongly stable and unstable foliations in $\Omega_{\theta}$ with respect to the flow $\{a_t\}$ respectively. 

In turns out that with respect to any admissible metric $\d$ on $\Omega_\theta$,
the $N_\theta$ and $\check{N}_\theta$-orbits are contained in the stable and unstable foliations of the directional flow $\{a_t\}$ on $\Omega_{\theta}$ respectively. The following proposition is important in applying Hopf-type arguments; 
the observation that one can use an admissible metric in this context
is due to Blayac-Canary-Zhu-Zimmer \cite{BCZZ}.

\begin{prop} \label{prop.foliation} Let $g \in G$ be such that $[g] \in \tilde \Omega_{\theta}$. For any compact subsets ${\cal U} \subset N_{\theta}$ and $\check{\cal U} \subset \check{N}_{\theta}$, we have, as $t \to + \infty$
$$
\begin{aligned}
    \diam \left( \{ \Ga [g n] \in \Omega_{\theta} : n \in {\cal U} \} \cdot a_t \right) \to 0; \\
    \diam \left( \{ \Ga [g \check{n}]  \in \Omega_{\theta} : \check{n} \in \check{\cal U} \} \cdot a_{-t} \right) \to 0
\end{aligned}
$$
where the diameter is computed with respect to an admissible metric $\d$ on $\Omega_\theta$. In particular,
    
    \begin{enumerate}
        \item 
        $\{\Gamma [gn]\in \Omega_\theta: n\in N_\theta\}\subset W^{ss}(\Gamma [g])$;

\item     $\{\Gamma [g\check{n}]\in \Omega_\theta: \check{n}\in \check{N}_{\theta}\}\subset W^{su}(\Gamma [g])$.

        \end{enumerate}
\end{prop}

\begin{proof} Let $\spadesuit$ be the point at infinity in the one-point compactification of $\Omega_{\theta}$. For each $\e>0$,
set $\sq_\e = \Omega_{\theta}$ if $\Omega_{\theta}$ is compact and $\sq_\e = \{x \in \Omega_{\theta} : \d(x, \spadesuit ) \ge \varepsilon /2\}$ otherwise, and  choose a compact lift $\tilde{\mathsf{Q}}_\e \subset \tilde \Omega_{\theta}$ of $\mathsf Q_\e$. 
Let $[g] = (\xi, \eta, v) \in \tilde \Omega_{\theta}$.
To show the first claim, suppose not. Then there exist
 $\varepsilon > 0$, a sequence $t_i \to \infty$ and convergent sequences $n_i, n_i'\in N_{\theta} $ such that $[gn_i], [gn_i'] \in \tilde \Omega_{\theta}$ and $\d (\Ga [ gn_i] a_{t_i}, \Ga [g n_i'] a_{t_i}) > \varepsilon$ for all $i \ge 1$.
By passing to a subsequence and switching $n_i$ and $n_i'$ if necessary, we may assume that  for all $i \ge 1$, $\ga_i[gn_i]a_{t_i} \in \tilde{ \mathsf Q}_\e$ 
for some  $\ga_i \in \Ga$. After passing to a subsequence, we have the convergence 
\be \label{eqn.folconv1}
\ga_i[gn_i]a_{t_i} = (\ga_i \xi, \ga_i (g n_i)^-, v + \beta_{\xi}^{\theta}(\ga_i^{-1}, e) + t_i u) \to (\xi_0, \eta_0, v_0) \text{ as } i \to \infty,
\ee for some $(\xi_0, \eta_0, v_0) \in \tilde{\sq}_{\e}$.
In particular, for any linear form $\phi \in \fa_{\theta}^*$ positive on $\fa_{\theta}^+$, we must have
$\phi(\beta_{\xi}^{\theta}(\ga_i^{-1}, e)) \to -\infty$ as $i \to \infty$ and the sequence $\ga_i$ is unbounded.

Since the sequence $n_i \in N_{\theta}$ converges, the sequence $(\xi, (gn_i)^-) \in \La_{\theta}^{(2)}$ is convergent as well. Moreover, \eqref{eqn.folconv1} implies that the sequence $\ga_i(\xi, (gn_i)^-) \in \La_{\theta}^{(2)}$ is precompact.
By the argument as in the proof of \cite[Lem. 9.10, Prop. 9.11]{KOW_indicators}, for any compact subset $C \subset \La_{\i(\theta)}$ with $\{\xi\} \times C \subset \La_{\theta}^{(2)}$, we have $\ga_i C \to \eta_0$ as $i \to \infty$. Since $n_i' \in N_{\theta}$ is a convergent sequence and $(\xi, \{(gn_i')^-\}) \subset \La_{\theta}^{(2)}$,  we have $\ga_i (g n_i')^- \to \eta_0$. Since $[gn_i'] = (\xi, (g n_i')^-, v)$ by Lemma \ref{lem.transbyhoro}, we deduce from \eqref{eqn.folconv1} that $$\ga_i[gn_i']a_{t_i} = (\ga_i \xi, \ga_i (gn_i')^-, v + \beta_{\xi}^{\theta}(\ga_i^{-1}, e) + t_i u) \to (\xi_0, \eta_0, v_0) \text{ as } i \to \infty.$$ Therefore, two sequences $\ga_i[g n_i]a_{t_i}$ and $\ga_i[gn_i']a_{t_i}$ converge to the same limit, which is a contradiction to the assumption $\d(\Ga[g n_i]a_{t_i}, \Ga[gn_i']a_{t_i}) > \varepsilon$ for all $i \ge 1$. Hence the first claim is proved. 

For the second claim, suppose to the contrary that for some $ \varepsilon > 0$, there exist a sequence $t_i \to \infty$ and convergent sequences $\check{n}_i, \check{n}_i'\in \check{N}_{\theta} $ such that $[g \check{n}_i], [g \check{n}_i'] \in \tilde \Omega_{\theta}$ and $\d(\Ga[g \check{n}_i]a_{-t_i}, \Ga[g\check{n}_i']a_{-t_i}) > \varepsilon$ for all $i \ge 1$. As above, we may then assume that  for all $i \ge 1$, $\ga_i[g \check{n}_i] a_{-t_i} \in \tilde{\mathsf Q}_\e$ for some sequence $\ga_i \in \Ga$. 
By passing to a subsequence, we have the convergence $$\ga_i [g \check{n}_i] a_{-t_i} \to (\xi_1, \eta_1, v_1) \text{ as } i \to \infty$$
for some $(\xi_1, \eta_1, v_1) \in \tilde{\mathsf Q}_\e$.
By Lemma \ref{lem.transbyhoro}, we have for each $i \ge 1$ that
$$
\begin{aligned}
\ga_i[g \check{n}_i] & = \ga_i \left((g \check{n}_i)^+, \eta, 
    v    + \cal G^{\theta}((g \check{n}_i)^+,  \eta) - \cal G^{\theta} (  \xi,  \eta) 
\right) \\
& = \left(\ga_i (g \check{n}_i)^+, \ga_i \eta, 
    v  + \cal G^{\theta}((g \check{n}_i)^+,  \eta) - \cal G^{\theta} (  \xi,  \eta) + \beta_{(g \check{n}_i)^+}^{\theta}(\ga_i^{-1}, e)
\right),
\end{aligned}
$$
and therefore we have that as $i \to \infty$,
\be \label{eqn.folconv2}
\begin{aligned}
&\ga_i (g \check{n}_i)^+  \to \xi_1; \\
&\ga_i \eta  \to \eta_1; \\
&v  + \cal G^{\theta}((g \check{n}_i)^+,  \eta) - \cal G^{\theta} (  \xi,  \eta) + \beta_{(g \check{n}_i)^+}^{\theta}(\ga_i^{-1}, e) - t_i u  \to v_1.
\end{aligned}
\ee
Since the sequence $\check{n}_i \in \check{N}_{\theta}$ converges, the sequence $((g \check{n}_i)^+, \eta) \in \La_{\theta}^{(2)}$ is convergent as well. Hence $\cal G^{\theta}((g \check{n}_i)^+, \eta)$ is a bounded sequence in $\fa_{\theta}$. It then follows from \eqref{eqn.folconv2} that for any linear form $\phi \in \fa_{\theta}^*$ positive on $\fa_{\theta}^+$, we have 
$$\phi(\beta_{(g \check{n}_i)^+}^{\theta}(\ga_i^{-1}, e) )\to \infty \quad \text{ as } i \to \infty$$ 
and the sequence $\ga_i$ is unbounded.

Again, by the same argument as in the proof of \cite[Lem. 9.10, Prop. 9.11]{KOW_indicators}, we obtain that for any compact subset $C \subset \La_{\theta}$ such that $C \times \{\eta\} \subset \La_{\theta}^{(2)}$, we have $\ga_i C \to \xi_1$ as $i \to \infty$.
Since the sequence $((g \check{n}_i')^+, \eta) \in \La_{\theta}^{(2)}$ is convergent as mentioned above,
we also have $\ga_i (g \check{n}_i')^+ \to \xi_1$ as $i \to \infty$. It then follows from Lemma \ref{lem.transbyhoro} that
$$\begin{aligned}
    \ga_i[g \check{n}_i] & = \left(\ga_i (g \check{n}_i)^+, \ga_i \eta, 
    v + \beta_{\xi}^{\theta}(\ga_i^{-1}, e)
    + \cal G^{\theta}( \ga_i(g \check{n}_i)^+, \ga_i \eta) - \cal G^{\theta} ( \ga_i \xi, \ga_i \eta) 
\right);\\
\ga_i[g \check{n}_i'] & = \left(\ga_i (g \check{n}_i')^+, \ga_i \eta, 
    v + \beta_{\xi}^{\theta}(\ga_i^{-1}, e)
    + \cal G^{\theta}( \ga_i(g \check{n}_i')^+, \ga_i \eta) - \cal G^{\theta} ( \ga_i \xi, \ga_i \eta) 
\right).
\end{aligned}$$
Since both sequences $(\ga_i(g \check{n}_i)^+, \ga_i \eta)$ and $(\ga_i (g \check{n}_i')^+, \ga_i \eta)$ converge to $(\xi_1, \eta_1)$ and $\ga_i[g \check{n}_i]a_{-t_i} \to (\xi_1, \eta_1, v_1)$ as $i \to \infty$, it follows that $$\ga_i[g \check{n}_i']a_{-t_i} \to (\xi_1, \eta_1, v_1) \quad \text{as } i \to \infty.$$
Again, two sequences $\ga_i[g \check{n}_i] a_{-t_i}$ and $\ga_i [ g \check{n}_i'] a_{-t_i}$ converge to the same limit, contradicting the assumption that $\d(\Ga[g \check{n}_i]a_{-t_i}, \Ga[g\check{n}_i']a_{-t_i}) > \varepsilon$ for all $i \ge 1$. This proves (2).
\end{proof}

For a $(\Ga, \theta)$-proper form $\phi \in \fa_{\theta}^*$, the action of $A_u = \{a_t : t \in \R\}$ on $\Omega_{\theta}$ induces a right $A_u$-action on $\Omega_{\phi}$ via the projection $\Omega_{\theta} \to \Omega_{\phi}$ where $\Omega_\phi$ is defined in
 \eqref{opsi}. Note that when $u \in \inte \L_{\theta}$, the condition $\phi(u) > 0$ is satisfied for any $(\Ga, \theta)$-proper $\phi \in \fa_{\theta}^*$ \cite[Lem. 4.3]{KOW_indicators}.
\begin{prop} \label{fol} Let $\phi\in \fa_\theta^*$ be a $(\Ga, \theta)$-proper form such that $\phi(u) > 0$ and
 $g \in G$ be such that $[g]_{\phi} \in \tilde \Omega_{\phi}$. For any compact subsets ${\cal U} \subset N_{\theta}$ and $\check{\cal U} \subset \check{N}_{\theta}$, we have, as $t\to +\infty$,
$$
\begin{aligned}
    \diam \left( \{ \Ga [g n]_{\phi}  \in \Omega_{\phi} : n \in {\cal U} \} \cdot a_t \right) \to 0; \\
    \diam \left( \{ \Ga [g \check{n}]_{\phi} \in \Omega_{\phi} : \check{n} \in \check{\cal U} \} \cdot  a_{-t} \right) \to 0
\end{aligned} $$
where the diameter is computed with respect to 
an admissible\footnote{I.e., it extends to a metric on the one-point compactification of $\Omega_\phi$} metric $\d$ on $\Omega_\phi$. In particular,
we have $$
\begin{aligned}
        \{\Gamma [gn]_\phi\in \Omega_\phi: n\in N_\theta\}\subset W^{ss}(\Gamma [g]_\phi);\\ 
\{\Gamma [g\check{n}]_\phi\in \Omega_\phi: \check{n}\in \check{N}_{\theta}\}\subset W^{su}(\Gamma [g]_\phi) , \end{aligned} $$
where $W^{ss}(x)$ (resp. $W^{su}(x)$)
    is the set of all $y \in \Omega_{\phi}$ such that $\d(x a_t, ya_t) \to 0$ as $t \to + \infty$ (resp. $t\to -\infty$) for $x\in \Omega_{\phi}$. 
\end{prop}

\begin{proof}
The condition $\phi(u) >0$ 
ensures that the convergence of the sequences $\phi(\beta_{\xi}^{\theta}(\ga_i^{-1}, e)) + t_i \phi(u)$ in \eqref{eqn.folconv1} and $\phi(\beta_{(g \check{n}_i)^+}^{\theta}(\ga_i^{-1}, e)) - t_i \phi(u)$ in \eqref{eqn.folconv2} implies that $\phi(\beta_{\xi}^{\theta}(\ga_i^{-1}, e)) \to - \infty$ and $\phi(\beta_{(g \check{n}_i)^+}^{\theta}(\ga_i^{-1}, e)) \to + \infty$ respectively.
Given this, we can proceed exactly as in the proof of Proposition \ref{prop.foliation}, replacing $\Omega_\theta$
by $\Omega_{\phi}$.
\end{proof}

\subsection*{Conservativity of general actions}
Let $H$ be a connected subgroup of $A$. Denote by $dh$ the Haar measure on $H$.
Consider the dynamical system $(H, \Omega, \lambda)$
where $\Omega$ is a separable, locally compact and $\sigma$-compact topological space on which $H$ acts continuously and preserving a Radon measure $\lambda$ on $\Omega$.
A Borel subset $B\subset \Omega$ is called wandering
if $\int_H {\mathbbm 1}_B(h. w)dh <\infty$ for $\la$-almost all $w\in B$. If there is no wandering subset of positive measure, the system is called completely conservative. 
If $\Omega$ is a countable union of wandering subsets, then the system is called completely dissipative.  An ergodic system $(H, \Omega, \lambda)$ is either 
 completely conservative or completely dissipative by the Hopf decomposition theorem.

\begin{lemma} \label{lem.consandergodic}
If $(\Omega_\theta, A_\theta, \m)$ is completely conservative, then
    it is $A_\theta$-ergodic.
\end{lemma}
\begin{proof} Choose any $\phi\in \fa_\theta^*$ which is positive on $\fa_\theta^+$;
in particular, $\phi$ is $(\Ga, \theta)$-proper.
Consider $\tilde \Omega_\phi$, $\Omega_\phi$ and $\m^\phi=\m^\phi_{\nu, \nu_{\i}}$ as defined in
\eqref{opsi} and \eqref{omm}.
The conservativity of the $A_\theta$-action on $(\Omega_\theta, \m)$
then implies the conservativity of the $\br$-action on $(\Omega_\phi, \m^\phi)$, and
 the $A_\theta$-ergodicity on $(\Omega_\theta, \m)$ follows if we show the ergodicity of
$(\Omega_\phi, \R, \m^\phi)$. 

Let $f$ be a bounded $\m^{\phi}$-measurable $\R$-invariant function on $\Omega_{\phi}$.
We need to show that $f$ is constant $\m^{\phi}$-a.e. 
Choose any admissible metric on $\Omega_\phi$ which exists by Theorem \ref{thm:opsi} and apply Proposition \ref{fol}.
  By a theorem of Coud\'ene \cite[Sec. 2]{Coudene_Hopf},  it follows that there exists an $\m^{\phi}$-conull subset $W_0 \subset \Omega_{\phi}$ such that if $\Ga[g]_{\phi}, \Ga[gn]_{\phi} \in W_0$ for $g \in G$ and $n \in N_{\theta} \cup \check{N}_{\theta}$, then 
  $$f(\Ga[g]_{\phi}) = f(\Ga[gn]_{\phi}).$$
Let $\tilde{f} : \tilde\Omega_\phi \to \R$ and $\tilde{W}_0 \subset \tilde\Omega_\phi$ be $\Ga$-invariant lifts of $f$ and $W_0$ respectively. Since $f$ is $\R$-invariant, we may assume that $\tilde{W}_0$ is $\R$-invariant as well.
For any $[g]_\phi, [h]_\phi\in \tilde{\Omega}_{\phi}$ with $g^+ = h^+$, we can find $n\in N_\theta$ and $a \in A_{\theta}$ such that $[gna]_{\phi} = [h]_{\phi}$ by \eqref{op}. Similarly, if $g^- = h^-$, we can find $n \in \check{N}_{\theta}$ and $a \in A_{\theta}$ such that $[gna]_{\phi} = [h]_{\phi}$.
Hence, by the $\R$-invariance of $f$ and hence of $\tilde f$,
 for any $(\xi, \eta, s), (\xi', \eta', s') \in \tilde{W}_0$
 such that $\xi = \xi'$ or $\eta = \eta'$, we have $\tilde f(\xi, \eta, s) = \tilde f(\xi', \eta', s')$.

Let \begin{align*}
    W^+ &:= \{ \xi \in \La_{\theta} : (\xi, \eta', s) \in \tilde{W}_0 \text{ for all } s \in \R\text{ and } \nu_{\i}\text{-a.e. } \eta'\};\\
    W^- &:= \{ \eta \in \La_{\i(\theta)} : (\xi', \eta, s) \in \tilde{W}_0 \text{ for all } s \in \R\text{ and } \nu\text{-a.e. } \xi'\}.
\end{align*} Then $\nu(W^+) = \nu_{\i}(W^-) = 1$ by Fubini's theorem. Hence the set $W' := (W^+ \times W^-) \cap \La_{\theta}^{(2)}$ has full $\nu \otimes \nu_{\i}$-measure. We choose a $\nu \otimes \nu_{\i}$-conull subset $W \subset W'$ such that $W \times \R \subset \tilde{W}_0$.
Let $(\xi, \eta), (\xi', \eta') \in W$. Then there exists $\eta_1 \in \La_{\i(\theta)}$ so that $(\xi, \eta_1), (\xi', \eta_1) \in W$. Hence for any $s \in \R$, we get
$$\tilde f(\xi, \eta, s) = \tilde f(\xi, \eta_1, s) = \tilde f(\xi', \eta_1, s) = \tilde f(\xi', \eta', s).$$ Therefore, $\tilde{f}$ is constant on $W \times \R$, and hence $f$ is constant $\m^{\phi}$-a.e., completing the proof.
\end{proof}

\subsection*{Ergodicity of directional flows}
We now prove the following analog of the Hopf dichotomy:
\begin{prop} \label{prop.ergodic}
    The directional flow $(\Omega_{\theta}, A_u,  \m)$ is completely conservative if and only if $(\Omega_{\theta}, A_u, \m)$ is ergodic.
\end{prop}

\begin{proof}
    Suppose that $(\Omega_{\theta}, A_u, \m)$ is completely conservative.
    Since this implies that $(\Omega_{\theta}, A_{\theta}, \m)$ is completely conservative, we have
    $(\Omega_{\theta}, A_{\theta}, \m)$ is ergodic by Lemma \ref{lem.consandergodic}.
    Let $f : \Omega_{\theta} \to \R$ be a bounded measurable function which is $A_u$-invariant. 
    By the $A_\theta$-ergodicity, it suffices to prove that $f$ is $A_\theta$-invariant.
    
Choose any admissible metric on $\Omega_\theta$ which exists by Theorem \ref{thm.propdisckow}.
Similarly to the proof of Lemma \ref{lem.consandergodic},
Proposition \ref{prop.foliation} and
\cite{Coudene_Hopf} imply  that there exists an $\m$-conull subset $W_0 \subset \Omega_{\theta}$ such that if $\Ga[g], \Ga[gn] \in W_0$ for $g \in G$ and $n \in N_{\theta} \cup \check{N}_{\theta}$, then 
    $$f(\Ga[g]) = f(\Ga[gn]).$$ Consider the $\Ga$-invariant lifts $\tilde{f} : \tilde{\Omega}_{\theta} \to \R$ and the $\tilde \m$-conull subset $\tilde{W}_0 \subset \tilde{\Omega}_{\theta}$ of $f$ and $W_0$ respectively.
    Let $$W_1 := \{(\xi, \eta) \in \La_{\theta}^{(2)} : (\xi, \eta, b) \in \tilde{W}_0 \text{ for } db\text{-a.e. } b \in \fa_{\theta} \};$$  $$W := \{(\xi, \eta) \in W_1 : (\xi, \eta'), (\xi', \eta) \in W_1 \text{ for } \nu\text{-a.e. } \xi' \in \La_{\theta}, \nu_{\i}\text{-a.e. }\eta' \in \La_{\i(\theta)}\} .$$
    By Fubini's theorem, $W$ has the full $\nu \otimes \nu_{\i}$-measure and we may assume that $W$ is $\Ga$-invariant as well. For all small $\varepsilon > 0$, we define $\tilde{f}_{\varepsilon} : \tilde{\Omega} \to \R$ by $$\tilde{f}_{\varepsilon}([g]) = \frac{1}{\vol(A_{\theta, \varepsilon})} \int_{A_{\theta, \varepsilon}} \tilde{f}([g]b) db$$
 where $A_{\theta, \varepsilon} = \{a \in A_{\theta} : \| \log a\| \le \varepsilon\}$.
    Then for $g \in G$ and $n \in N_{\theta} \cup \check{N}_{\theta}$ such that $(g^+, g^-), ((gn)^+, (gn)^-) \in W$, we have $\tilde{f}_{\varepsilon}([g]) = \tilde{f}_{\varepsilon}([gn])$ and $\tilde{f}_{\varepsilon}$ is continuous on $[g]A_{\theta}$.

    Since $\tilde{f} = \lim_{\varepsilon \to 0} \tilde{f}_{\varepsilon}$ $\tilde{\m}$-a.e., it suffices to show that $\tilde{f}_\e$ is $A_\theta$-invariant. 
       Fix $g \in G$ such that $(g^+, g^-) \in W$. 
       By Proposition \ref{prop.transitivity} and the continuity of $\tilde f_\e$ on each
        $A_\theta$-orbit,  it is again sufficient to show that $\tilde f_\e$ is invariant under $\cal H_{\Ga}^{\theta}(g)$.  Let $a \in \cal H_{\Ga}^{\theta}(g)$. Then there exist $\ga \in \Ga$ and a sequence $n_1, \cdots, n_k \in N_{\theta} \cup \check{N}_{\theta}$ such that \begin{enumerate}
        \item $(g n_1 \cdots n_r)^+ \in \La_{\theta}$ and $(g n_1 \cdots n_r)^- \in \La_{\i(\theta)}$ for all $1 \le r \le k$; and
        \item $gn_1\cdots n_k = \ga gas$ for some $s \in S_{\theta}$.
    \end{enumerate}
    For each $i = 1, \cdots, k$, we denote by $N_i = N_{\theta}$ if $n_i \in N_{\theta}$ and $N_i = \check{N}_{\theta}$ if $n_i \in \check{N}_{\theta}$. We may assume that $N_i \neq N_{i+1}$ for all $1 \le i \le k -1$. Noting that $W$ is $\Ga$-invariant, we consider a sequence of $k$-tuples $(n_{1, j}, \cdots, n_{k, j}) \in N_1 \times \cdots \times N_k$ as follows:

    \medskip
    {\bf \noindent Case 1: $N_k = \check{N}_{\theta}$.} In this case, we have $$(\ga g)^+ = (gn_1\cdots n_k)^+ \quad \text{and} \quad (\ga g)^- = (gn_1 \cdots n_{k-1})^-.$$
   Take a sequence of $k$-tuples $(n_{1, j}, \cdots, n_{k, j}) \in N_1 \times \cdots \times N_k$ converging to $(n_1, \cdots, n_k)$ as $j \to \infty$ so that for each $j$, we have \begin{enumerate}
        \item $((gn_{1, j} \cdots n_{r, j})^+, (gn_{1, j} \cdots n_{r, j})^-) \in W$ for all $1 \le r \le k$;
        \item $(\ga g)^- = (gn_{1, j} \cdots n_{k-1, j})^-$; and
        \item $(\ga g)^+ = (gn_{1, j}\cdots n_{k, j})^+$.
    \end{enumerate}
    This is possible since $(g^+, g^-), ((\ga g)^+, (\ga g)^-) \in W$ and $W$ has the full $\nu \otimes \nu_{\i}$-measure. Since $n_{k, j} \in \check{N}_{\theta}$, we indeed have $(\ga g)^- = (g n_{1, j} \cdots n_{k, j})^-$ as well, and therefore $g n_{1, j} \cdots n_{k, j} = \ga g a_j s_j$ for some $a_j \in A_{\theta}$ and $s_j \in S_{\theta}$. In particular, we have $$[gn_{1, j} \cdots n_{k, j}] = [\ga g a_j] \in \tilde{\Omega}_{\theta} \quad \text{for all } j \ge 1.$$

    \medskip
    {\bf \noindent Case 2: $N_k = N_{\theta}$.} In this case, we have $$(\ga g)^+ = (gn_1\cdots n_{k-1})^+ \quad \text{and} \quad (\ga g)^- = (gn_1 \cdots n_{k})^-.$$
 We then take a sequence of $k$-tuples $(n_{1, j}, \cdots, n_{k, j}) \in N_1 \times \cdots \times N_k$ converging to $(n_1, \cdots, n_k)$ as $j \to \infty$ so that for each $j$, we have \begin{enumerate}
        \item $((gn_{1, j} \cdots n_{r, j})^+, (gn_{1, j} \cdots n_{r, j})^-) \in W$ for all $1 \le r \le k$;
        \item $(\ga g)^+ = (gn_{1, j} \cdots n_{k-1, j})^+$; and
        \item $(\ga g)^- = (gn_{1, j}\cdots n_{k, j})^-$.
    \end{enumerate}
    Since $n_{k, j} \in N_{\theta}$, we have $(\ga g)^+ = (g n_{1, j} \cdots n_{k, j})^+$ as well, and therefore $g n_{1, j} \cdots n_{k, j} = \ga g a_j s_j$ for some $a_j \in A_{\theta}$ and $s_j \in S_{\theta}$. In particular, we have $$[gn_{1, j} \cdots n_{k, j}] = [\ga g a_j] \in \tilde{\Omega}_{\theta} \quad \text{for all } j \ge 1.$$

    \medskip

    In either case, we have that for each $j\ge 1$,
    $$\tilde{f}_{\varepsilon}([\ga g a_j]) = \tilde{f}_{\varepsilon}([gn_{1, j} \cdots n_{k, j}]) = \tilde{f}_{\varepsilon} ([gn_{1, j} \cdots n_{k-1, j}]) = \cdots = \tilde{f}_{\varepsilon}([g]).$$ Since $\tilde{f}_{\varepsilon}$ is $\Ga$-invariant, it implies $$\tilde{f}_{\varepsilon}([ga_j]) = \tilde{f}_{\varepsilon}([g]) \quad \text{for all } j \ge 1.$$ Since $a_j$ converges to $a$ as $j \to \infty$, we get $\tilde{f}_{\varepsilon}([ga]) = \tilde{f}_{\varepsilon}([g])$ by the continuity of $\tilde{f}_{\varepsilon}$ on $gA_{\theta}$. 
    This shows that $\tilde f_\e$ is invariant under $\cal H_{\Ga}^{\theta}(g)$, finishing the proof of ergodicity.

    Now suppose that the flow $(\Omega_{\theta}, A_u, \m)$ is ergodic. Then by the Hopf decomposition theorem, it is either completely conservative or completely dissipative. Suppose to the contrary that $(\Omega_{\theta}, A_u, \m)$ is completely dissipative. Then it is isomorphic to a translation on $\R$ with respect to the Lebesgue measure.
 This yields a contradiction as in  proof of \cite[Thm. 10.2]{KOW_indicators}, as we recall for readers' convenience. 
 Since $(\Omega_{\theta}, A_u, \m)$ is isomorphic to a translation on $\R$,  $(\nu \times \nu_{\i})|_{\La_{\theta}^{(2)}}$ is supported on the single $\Ga$-orbit $\Ga(\xi_0, \eta_0)$
 by the ergodicity of $(\Ga, \La_{\theta}^{(2)}, \nu \times \nu_{\i})$.  Since $\nu$ and $\nu_{\i}$ also have atoms on $\xi_0$ and $\eta_0$ respectively, we have $$(\Ga \xi_0 \times \Ga \eta_0) \cap \La_{\theta}^{(2)} \subset \Ga(\xi_0, \eta_0).$$ 
    We deduce from the $\theta$-antipodality of $\Ga$ that $\Ga \xi_0 \subset \Ga_{\eta_0} \xi_0 \cup \{\eta_0'\}$ where $\Ga_{\eta_0} = \stab_{\Ga}(\eta_0)$ and $\eta_0'$ is the image of $\eta_0$ under the $\Ga$-equivariant homeomorphism $\La_{\i(\theta)} \to \La_{\theta}$ obtained in \cite[Lem. 9.5]{KOW_indicators}.
    Since $\Ga_{\eta_0} = \Ga_{\eta_0'}$, we have
\be \label{eqn.noneltcontra}
\Ga \xi_0 \subset \Ga_{\eta_0'} \xi_0 \cup \{\eta_0'\}.
\ee Recall that the $\Ga$-action on $\La_{\theta}$ is a non-elementary convergence group action \cite[Thm. 4.16]{KLP_Anosov} and hence there must be infintely many accumulation points of $\Ga \xi_0$.
On the other hand,
as $\Ga_{\eta_0'}$ is an elementary subgroup, the orbit  $\Ga_{\eta_0'} \xi_0$ accumulates at most at two points of $\La_{\theta}$ (\cite{Tukia_convergence}, \cite{Bowditch1999convergence}). This yields a contradiction, and therefore  $(\Omega_{\theta}, A_u, \m)$ is completely conservative.
\end{proof}

\subsection*{Proof of Theorem \ref{main}}
The equivalences between (1)-(3) follow from Proposition \ref{prop.conserv} and Proposition \ref{prop.ergodic}. Suppose that $\m$ is $u$-balanced.  Corollary \ref{thm.directionalpoincare} implies that $(1) \Leftrightarrow (4) \Leftrightarrow (5)$. That the first case occurs only when $\psi(u) = \psi_{\Ga}^{\theta}(u) > 0$ is a consequence of Lemma \ref{lem.tangent}.

\section{Ergodic dichotomy for subspace flows} \label{sec.codim}
In this section, we extend our ergodic dichotomy to the action of any connected subgroup of $A_\theta$ of arbitrary dimension. In fact, we deduce this from the ergodic dichotomy for directional flows.

\medskip

Let $\Ga$ be a Zariski dense $\theta$-transverse subgroup of $G$. Let $W < \fa_{\theta}$ be a non-zero linear subspace and set $A_W = \exp W$. We consider the subspace flow $A_W$ on $\Omega_{\theta}$ and explain how the proof of Theorem \ref{main} extends to this setting so that we obtain Theorem \ref{thm.subspacedichotomy}, adapting the argument of Pozzetti-Sambarino \cite{PS_metric} on relating the subspace flows with directional flows.

 For $R>0$, we set
$$\Ga_{W,R}=\{\gamma\in \Ga: \|\mu_\theta (\ga)-W\| < R\}.$$
If $W=\fa_\theta$, then $\Ga_{W, R}=\Ga$ for all $R>0$.
\begin{definition}[$W$-conical points] \label{def.subspaceconical}
We say that $\xi\in \F_\theta$ is a $W$-conical point of $\Ga$ if
there exist $R>0$ and a sequence $\ga_i \in \Ga_{W, R}$ such that
$\xi \in O_{R}^{\theta}(o, \ga_i o)$ for all $i \ge 1$.
We denote by $\La_\theta^W$ the set of all $W$-conical points of $\Ga$.
\end{definition}

Fix a $(\Ga, \theta)$-proper linear form $\psi \in \fa_{\theta}^*$. Let $\nu, \nu_{\i}$ be a pair of $(\Ga, \psi)$ and $(\Ga, \psi \circ \i)$-conformal measures on $\La_{\theta}$ and $\La_{\i(\theta)}$ respectively, and let $\m = \m_{\nu, \nu_{\i}}$ denote the associated BMS measure on $\Omega_{\theta}$.

If $W \cap \L_{\theta} = \{0\}$ or  $W \subset \ker \psi$, then 
 the $(\Ga,\theta)$-proper hypothesis on $\psi$ implies that $\Ga_{W, R}$ is finite for all $R > 0$, and hence $\La_{\theta}^W = \La_{\i(\theta)}^{\i(W)}= \emptyset$ and $(\Omega_{\theta} , A_W, \m)$ is completely dissipative and non-ergodic.

The rest of this section is now devoted to
proving Theorem \ref{thm.subspacedichotomy}, 
assuming that
\begin{itemize}
    \item $W\cap \L_{\theta}\ne \{0\}$;
    \item  $ W\not\subset \ker \psi$.
\end{itemize}

Recalling that $\psi \ge 0$ on $\L_{\theta}$ by \cite[Lem. 4.3]{KOW_indicators}, the intersection $W \cap \ker \psi$ has  codimension one in $W$ and intersects $\inte \L_{\theta}$ only at $0$.

Set $$W^\diamond=\fa_\theta/(W\cap \ker \psi)\quad\text{ and }\quad \tilde \Omega_{W^\diamond}:=\La_\theta^{(2)}\times {W^\diamond} .$$
Recalling the spaces $\tilde \Omega_{\psi}$ and $\Omega_{\psi}$  defined in \eqref{opsi},
the projection $\tilde{\Omega}_{\theta} \to \tilde{\Omega}_{\psi}$ factors through $\tilde{\Omega}_{{W^\diamond}}$ in a $\Ga$-equivariant way. Since
the $\Ga$-action on $\tilde{\Omega}_{\psi}$ is properly discontinuous (Theorem \ref{thm:opsi}),
the induced $\Ga$-action  on $\tilde{\Omega}_{{W^\diamond}}$ is also properly discontinuous.
Moreover, the trivial vector bundle $\Omega_\theta\to \Omega_{\psi}$ in \eqref{trivial} factors through
\be \label{eqn.V}
\Omega_{W^\diamond}:=\Ga\ba \tilde \Omega_{W^\diamond}.
\ee
Hence we have a $W\cap \ker \psi$-equivariant homeomorphism:
$$\Omega_\theta\simeq \Omega_{W^\diamond} \times (W\cap \ker \psi).$$

Denote by $\m'$ the $A_\theta$-invariant
Radon measure on $\Omega_{W^\diamond}$ such that $\m=\m' \otimes \op{Leb}_{W\cap \ker \psi}$.

The main point of the proof of Theorem \ref{thm.subspacedichotomy} is to relate the action of $A_W$ on $\Omega_\theta$
with that of a directional flow on $\Omega_{W^{\diamond}}$. Once we do that, we can proceed similarly to 
the proof of Theorem \ref{main}.

Since $W\not\subset\ker\psi$, there exists $u\in W$ with $\psi(u)\ne 0$. By replacing $u$ by $-u$ if necessary, we fix $u\in W$ such that
 $$\psi(u)>0.$$ 
Set  $A_u =A_{\br u}= \{a_{tu}=\exp (tu) : t \in \R\}$ and  consider the $A_u$-action on $(\Omega_{W^{\diamond}}, \m')$. Since $W=\br u +(W\cap \ker \psi)$, we have:
\begin{lem}\label{same}
  The $A_W$-action on  $(\Omega_\theta , \m)$ is ergodic
  (resp. completely conservative, non-ergodic, completely dissipative) if and only if
the $A_u$-action on $(\Omega_{W^{\diamond}}, \m')$ ergodic
  (resp. completely conservative, non-ergodic, completely dissipative).
\end{lem}

Among the ingredients for the proof of Theorem \ref{main}, Lemma \ref{lem.repeat} and Proposition \ref{lem.dircartan} were repeatedly used and played  basic roles in the proof. The following analogue of Lemma \ref{lem.repeat} can be proved by a similar argument as in the proof of
Lemma \ref{lem.repeat}:
\begin{lemma} \label{lem.analrepeat}
    Suppose that $d_i \in a_{t_i u} \exp(W \cap \ker \psi) B_{\theta}^+$, $t_i > 0$ and $\ga_i \in \Ga$ are sequences such that $\ga_i h_i m_i d_i$ is bounded for some bounded sequence $h_i \in G$ with $h_i P \in \La$ and $m_i \in M_{\theta}$. Then there exists  $w \in \cal W \cap M_{\theta}$ such that,
    after passing to a subsequence,  we have that for all $i \ge 1$, $$d_i \in w A^+ w^{-1} .$$
\end{lemma}

\begin{proof}
    As in the proof of Lemma \ref{lem.repeat}, there exists a Weyl element $w \in \cal W$ such that $d_i \in w A^+ w^{-1}$ for all $i \ge 1$ after passing to a subsequence, and moreover $w \in M_{\theta}$ or $w \in M_{\theta} w_0$. We claim that the latter case $w \in M_{\theta} w_0$ cannot happen. Suppose that $w \in M_{\theta} w_0$ and write $d_i = a_{t_i u} a_i b_i$ for $a_i \in \exp (W \cap \ker \psi)$ and $b_i \in B_{\theta}^+$. Since $w \in M_{\theta} w_0$, we get $\mu_{\i(\theta)}(d_i) = \log (w_0^{-1} a_{t_i u} a_i w_0)$ for all $i \ge 1$. In particular, $t_i u + \log a_i \in - \fa_{\theta}^+$. 

    Since the sequence $\ga_i h_i m_i d_i$ is bounded by the hypothesis, the sequence $\mu_{\i(\theta)}(\ga_i^{-1}) - \mu_{\i(\theta)}(d_i)$ is bounded as well by Lemma \ref{lem.cptcartan}. Since $\mu_{\i(\theta)}(\ga_i^{-1}) = -\op{Ad}_{w_0}(\mu_{\theta}(\ga_i))$ and $\mu_{\i(\theta)}(d_i) = \op{Ad}_{w_0}( t_i u + \log  a_i)$, it follows that $\mu_{\theta}(\ga_i) =- (t_i u + \log a_i) +q_i$ for some bounded sequence
    $q_i\in \fa_\theta$.
    Applying $\psi$, we get $\psi(\mu_\theta(\ga_i))= -t_i \psi (u)+\psi(q_i)$ since $\log a_i\in \ker \psi$. Since $\psi(u)>0$,  $\psi(\mu_\theta(\ga_i)) $  is uniformly bounded.
    The $(\Ga, \theta)$-properness of $\psi$ implies that $\ga_i$ is a finite sequence, yielding a contradiction.  Therefore, the case
    $w \in M_{\theta} w_0$ cannot occur; so  $w \in M_{\theta}$.
\end{proof}

Let $p: \fa_\theta \to W^{\diamond} $ denote the natural projection map. Choosing a norm $\|\cdot \|$ on
$W^{\diamond}$, the map $p$ is Lipschitz.
Then for a constant $c>1$ depending  on the Lipschitz constant of $p$ as well as norms on $\fa_{\theta}$ and $W^{\diamond}$, we have for all $R>0$,
$$ \{ \ga \in \Ga : \| p(\mu_{\theta}(\ga)) - \R u\| < R/c \} \subset 
\Ga_{W, R}\subset  \{ \ga \in \Ga : \| p(\mu_{\theta}(\ga)) - \R u\| < cR \} .$$ 
 Note also that $\psi(p(\mu_{\theta}(\ga))) = \psi(\mu_{\theta}(\ga))$ for all $\ga \in \Ga$.
 
Using this relation and
Lemma \ref{lem.analrepeat}, similar arguments as in Sections \ref{sec:rec}  and \ref{sec.dirpoincare} apply to the $A_u$-flow on $\Omega_{W^{\diamond}}$, replacing  $\Ga_{u, r}$ with $\Ga_{W, R}$. In particular,
applying Lemma \ref{lem.analrepeat} in place of  Lemma \ref{lem.repeat}, 
 the following analogs of Proposition \ref{lem.dircartan} and Lemma \ref{defdir}(2) respectively can be proved similarly. 

\begin{prop} \label{prop.analgromov}
Let $Q\subset \tilde \Omega_{W^{\diamond}}$ be a compact subset. 
There are positive constants $C_1=C_1(Q), C_2=C_2(Q)$ and $R=R(Q)$ such that if $[h] \in Q \cap \ga Qa_{-tu} $ for some $h \in G$, $\ga \in \Ga$ and $t > 0$, then the following hold:
    \begin{enumerate}
        \item   $\| p(\mu_{\theta}(\ga)) - t u \| < C_1;$
\item $(h^+, h^-) \in O_R^{\theta}(o, \ga o) \times O_R^{\i(\theta)}(\ga o, o) ;$
\item $\|\cal G^{\theta}(h^+, h^-) \| < C_2.$
\end{enumerate}
\end{prop}

\begin{lem} \label{lem.equivdefsubspace} The following are equivalent for any $\xi\in \La_\theta$:
\begin{enumerate}
  \item $\xi\in \La_\theta^W$;
\item $\xi = gP_\theta\in \F_\theta$ for some $g \in G$ such that $[g]\in \Omega_\theta$ and $\limsup [g](A_W\cap A^+) \ne \emptyset$;
\item the sequence $[(\xi, \eta, v)]a_{t_iu}$ is precompact in $\Omega_{W^{\diamond}}$ for some $\eta \in \La_{\i(\theta)}$, $v \in W^{\diamond}$ and $t_i \to \infty$.
\end{enumerate}
\end{lem}
In particular, a $W$-conical point of $\Ga$ is a $u$-conical point
for the action of $A_u$ on $\Omega_{W^\diamond}$ and vice versa.

Since the recurrence of the $A_u$-flow on $\Omega_{W^{\diamond}}$ is related to the $W$-conical set as stated in Lemma \ref{lem.equivdefsubspace}, the arguments in Section \ref{sec.conserg}  for the directional flow $(\Omega_{W^{\diamond}}, A_u, \m')$ yield the following equivalences: \be \label{eqn.finallysubspace}
\begin{aligned}
\max \left( \nu(\La_{\theta}^W), \nu_{\i}(\La_{\i(\theta)}^{\i(W)}) \right) > 0 &\Leftrightarrow (\Omega_{W^{\diamond}}, A_u, \m') \text{ is completely conservative} \\
& \Leftrightarrow (\Omega_{W^{\diamond}}, A_u, \m') \text{ is ergodic}; \\
\max \left(\nu(\La_{\theta}^W), \nu_{\i}(\La_{\i(\theta)}^{\i(W)}) \right) = 0 &\Leftrightarrow (\Omega_{W^{\diamond}}, A_u, \m') \text{ is completely dissipative} \\
& \Leftrightarrow (\Omega_{W^{\diamond}}, A_u, \m') \text{ is non-ergodic}.
\end{aligned}
\ee 
This proves the equivalence $(1)\Leftrightarrow(2) \Leftrightarrow (3)$ of Theorem \ref{thm.subspacedichotomy}.

\begin{definition}\label{Wb}  We say that $\m$ is $W$-balanced if there exists 
$u \in W $ with $\psi(u)>0$ 
such that $(\Omega_{W^\diamond}, \m')$ is $u$-balanced.
\end{definition}

To complete the proof of Theorem \ref{thm.subspacedichotomy}, it remains to prove the following:
\begin{theorem} \label{thm.subspacepoincare} Suppose that 
 $\m$ is $W$-balanced.
The following are equivalent:
\begin{enumerate}
    \item 
$\sum_{\ga \in \Ga_{W, R}} e^{-\psi(\mu_{\theta}(\ga))} = \infty$ for some $R > 0$;
\item $\nu(\La_{\theta}^W) = 1 = \nu_{\i}(\La_{\i(\theta)}^{\i(W)}).$ 
\end{enumerate}

Similarly, the following are also equivalent:
\begin{enumerate}
    \item 
 $\sum_{\ga \in \Ga_{W, R}} e^{-\psi(\mu_{\theta}(\ga))} <\infty$ for all $R > 0$;

\item  $\nu(\La_{\theta}^W) = 0= \nu_{\i}(\La_{\i(\theta)}^{\i(W)}).$
\end{enumerate}

\end{theorem}

In the rest of this section, we assume that
$\m$ is $W$-balanced, and choose $u\in W$ with $\psi(u) > 0$ so that $\m'$ is $u$-balanced. Following the proof of Proposition \ref{prop.upperbound} while applying Proposition \ref{prop.analgromov} in the place of Proposition \ref{lem.dircartan}, we get:

\begin{proposition} \label{prop.analupperlower}
Suppose that $\sum_{\ga \in \Ga_{W, R}} e^{-\psi(\mu_{\theta}(\ga))} = \infty$ for some $R > 0$. Set $\delta=\psi(u)>0$.
 \begin{enumerate}
     \item 
  For any compact subset $Q\subset \tilde \Omega_{W^{\diamond}}$, there exists $R=R(Q)>0$ such that for any $T > 1$, we have 
    $$\int_0^T \int_0^T \sum_{\ga, \ga' \in \Ga} \tilde{\m}'(Q \cap \ga Q a_{-t u} \cap \ga' Q a_{-(t+s)u}) dt ds \ll \left( \sum_{\substack{\ga \in \Ga_{W, R} \\ \psi(\mu_{\theta}(\ga)) \le \delta T}} e^{-\psi(\mu_{\theta}(\ga))} \right)^2.
    $$
\item 
    For any $R > 0$, there exists a compact subset $Q'=Q'(R)\subset \tilde \Omega_{W^{\diamond}}$ such that
     $$\int_0^T \sum_{\ga \in \Ga} \tilde{\m}'(Q' \cap \ga Q' a_{-tu}) dt  \gg \sum_{\substack{\ga \in \Ga_{W, R} \\ \psi(\mu_{\theta}(\ga)) \le \delta T}} e^{-\psi(\mu_{\theta}(\ga))}.
     $$ 
 \end{enumerate} 
 \end{proposition}

The proof of Theorem \ref{thm.convergencedir} works verbatim for $\La_{\theta}^W$ so that the convergence $\sum_{\ga \in \Ga_{W, R}} e^{-\psi(\mu_{\theta}(\ga))} < \infty$ for all $R > 0$ implies that $\nu(\La_{\theta}^W) = 0$. Using Proposition \ref{prop.analupperlower} together with the $W$-balanced condition,  Theorem \ref{thm.subspacepoincare} can now be proved by the same argument as in the proof of Corollary \ref{thm.directionalpoincare}. 

\begin{rmk}\label{fine}
 The  $W$-balanced condition on $\m$ was needed because $Q$ and $Q'$ in Proposition \ref{prop.analupperlower} may not be the same  in principle.
 However when $W=\fa_\theta$, we have
 $\Ga_{W, R}=\Ga$ for any $R>0$
 and  $Q$ and $Q'$ in Proposition \ref{prop.analupperlower}
   can be taken to be the same set, and hence
the $W$-balanced condition is not needed in the proof of Theorem \ref{thm.subspacedichotomy}.
\end{rmk}

Similarly to Corollary \ref{useful}, we have the following
 estimates  which reduce the divergence
    of the series $\sum_{\ga \in \Ga_{W, R} } e^{-\psi(\mu_{\theta}(\ga))}$ to the local mixing rate for the $a_t$-flow:
\begin{cor} \label{useful2}  For all sufficiently large $R>0$, there exist compact subsets $Q_1, Q_2$ of $\Omega_{W^\diamond}$ with non-empty interior such that
for all $T\ge 1$,
$$\left(\int_0^T  {\m'}(Q_1 \cap Q_1 a_{-t}) dt \right)^{1/2}
\ll  \sum_{\substack{\ga \in \Ga_{W, R} \\ \psi(\mu_{\theta}(\ga)) \le \delta T}} e^{-\psi(\mu_{\theta}(\ga))} \ll \int_0^T  {\m'}(Q_2 \cap Q_2 a_{-t}) dt  .$$
\end{cor}
\section{Dichotomy theorems for Anosov subgroups}

In this last section, we focus on Anosov subgroups and establish the codimension dichotomy for ergodicity of the subspace flow given by $\exp W$, for a linear subspace $W<\fa_\theta$. Using the local mixing theorem for directional flows for Anosov subgroups, we show that the Poincar\'e series associated to $W$ diverges  if and only if the codimension of the subspace $W$ in $\fa_\theta$ is at most 2.

\medskip

Let $\Ga < G$ be a Zariski dense $\theta$-Anosov subgroup defined as in the introduction. Recall that $\L_{\theta}\subset \fa_\theta^+$ denotes the $\theta$-limit cone of $\Ga$.
Denote by $\cal T_{\Ga}^{\theta}\subset \fa_\theta^*$ the set of all linear forms tangent to the growth indicator $\psi_{\Ga}^{\theta}$ and by $\cal M_{\Ga}^{\theta}$ the set of all $\Ga$-conformal measures on $\La_{\theta}$. There are one-to-one correspondences between the following sets (\cite[Coro. 1.12]{KOW_indicators}, \cite[Thm. A]{sambarino2022report}): $$\P(\inte \L_{\theta}) \longleftrightarrow \cal T_\Ga ^{\theta} \longleftrightarrow \cal M_{\Ga}^{\theta}.$$ Namely, for each unit vector $v \in \inte \L_{\theta}$, there exists a unique $\psi_v \in \fa_{\theta}^*$ which is tangent to $\psi_{\Ga}^{\theta}$ at $v$ and a unique $(\Ga, \psi_v)$-conformal measure $\nu_v$ supported on $\La_{\theta}$. The linear form $\psi_v \circ \i\in \fa_{\i(\theta)}^*$ 
is tangent to $\psi_\Ga^{\i(\theta)}$ at $\i(v)$ and the measure
$\nu_{\i(v)}$ is a ($\Ga, \psi_v \circ \i)$-conformal measure on $\La_{\i(\theta)}$.
Denote by $\m_v$ the BMS measure on $\Omega_{\theta}$ associated with the pair $(\nu_v, \nu_{\i(v)})$. 

What distinguishes $\theta$-Anosov subgroups from general $\theta$-transverse subgroups is that
$\Omega_{\psi_v}$ is a {\it compact} metric space (\cite{sambarino_orbit} and \cite[Appendix]{Carvajales}) and hence
 $\Omega_\theta$ is a vector bundle over a {\it compact} space $\Omega_{\psi_v}$ with fiber $\ker \psi_v\simeq \br^{\#\theta -1}$.
We use the the following local mixing for directional flows due to Sambarino.

\begin{theorem}[{\cite[Thm. 2.5.2]{sambarino2022report}, see also \cite{CS_local} for $\theta=\Pi$}] \label{thm.mixing}
Let $\Ga < G$ be a $\theta$-Anosov subgroup and $v \in \inte \L_{\theta}$. Then there exists $\kappa_v > 0$ such that for any $f_1, f_2 \in C_c(\Omega_{\theta})$, $$\lim_{t \to  \infty} t^{\frac{\# \theta - 1}{2}} \int_{\Omega_{\theta}} f_1(x) f_2(x \exp(tv)) d\m_v(x) = \kappa_v \m_v(f_1) \m_v(f_2).$$
\end{theorem}

In particular, for any $v \in \inte \L_{\theta}$, $\m_v$ is $v$-balanced.

\begin{corollary} \label{cor.div3}
    For any $v \in \inte \L_{\theta}$ and any bounded Borel subset $Q \subset \tilde{\Omega}_{\theta}$ with non-empty interior, we have for any $T>2$,
     $$\int_0 ^{T} \sum_{\ga \in \Ga} \tilde{\m}_v(Q \cap \ga Q \exp(-tv)) dt 
      \asymp\int_1^T  t^{\frac{1 - \# \theta}{2}} dt.$$
\end{corollary}

\begin{proof}
    Given a bounded Borel subset $Q \subset \tilde{\Omega}_{\theta}$ with non-empty interior, we choose $\tilde{f}_1, \tilde{f}_2 \in C_c(\tilde{\Omega}_{\theta})$ so that $0 \le \tilde{f}_1 \le \mathbbm{1}_Q \le \tilde{f}_2$ and $\tilde{\m}_v(\tilde{f}_1) > 0$. For each $i =1, 2$, we define the function $f_i \in C_c(\Omega_{\theta})$ by $f_i(\Ga[g]) = \sum_{\ga \in \Ga} \tilde{f}_i(\ga g)$. By Theorem \ref{thm.mixing}, for each $i=1,2$, we have  that for all $t \ge 1$, $$\begin{aligned}
        \int_{\tilde{\Omega}_{\theta}} \sum_{\ga \in \Ga} \tilde{f}_i(\ga [g] \exp(tv)) \tilde{f}_i([g]) d\tilde{\m}_v([g]) & = \int_{\Omega_{\theta}} f_i(x \exp(tv)) f_i(x) d\m_v(x) \\
        & \asymp t^{\frac{1 - \# \theta}{2}}.
    \end{aligned}$$
\end{proof}

By Corollary \ref{useful} and Corollary \ref{cor.div3}, we get:
\begin{prop} \label{prop.poincareintegral} Let $v\in \inte\L_\theta$ and $\delta=\psi_v(v)$. For all sufficiently large $r>0$, we have that for all $T>2$,
     \be\label{div0} \left(\int_1^T  t^{\frac{1 - \# \theta}{2}} dt\right)^{1/2} \ll 
     \sum_{\substack{\ga \in \Ga_{v, r}  \\ \psi_v(\mu_{\theta}(\ga)) \le \delta T}} e^{-\psi_v(\mu_{\theta}(\ga))}  \ll \int_1^T  t^{\frac{1 - \# \theta}{2}} dt .\ee
\end{prop}

\begin{theorem} \label{thm.thetarank}
    For any  $v \in \inte \L_{\theta}$ and $u \in \fa_{\theta}^+ - \{0\}$,
    the following are equivalent:
    \begin{enumerate}
        \item $\# \theta \le 3$ and $\R u = \R v$;
        \item $\sum_{\ga \in \Ga_{u, r}} e^{-\psi_v(\mu_{\theta}(\ga))} = \infty$ for some $r > 0$.
    \end{enumerate}
\end{theorem}

\begin{proof}
    Note that $\int_1^\infty t^{\frac{1-\# \theta}{2}} dt = \infty$ if and only if $\# \theta \le 3$. Hence (1) implies (2) by Proposition \ref{prop.poincareintegral}.
    To show the implication $(2) \Rightarrow (1)$, suppose that  $\sum_{\ga \in \Ga_{u, r}} e^{-\psi_v(\mu_{\theta}(\ga))} = \infty$ for some $r > 0$. By Lemma \ref{lem.tangent}, $\psi_v(u) = \psi_{\Ga}^{\theta}(u)$. It follows from the strict concavity of $\psi_{\Ga}^{\theta}$ \cite[Thm. 12.2]{KOW_indicators} that $\psi_v$ can be tangent to
    $\psi_\Ga^\theta$ only in the direction $\br v$. Therefore
    $\R u = \R v$. Now $\# \theta \le 3$ follows from Proposition \ref{prop.poincareintegral}.
\end{proof}

Here is the special case of Theorem \ref{main3} for $\dim W=1$:
\begin{theorem} 
\label{main2} Let $\Ga<G$ be a Zariski dense $\theta$-Anosov subgroup. For any $u \in \inte \L_{\theta}$,
the following are equivalent:
    \begin{enumerate}
        \item $\#\theta \le  3$ (resp. $\# \theta \ge 4$);
        \item $\nu_u(\La_\theta^u)=1$ (resp.  $\nu_{u}(\La_{\theta}^u) = 0 $);
       \item $(\Omega_\theta, A_u,  \mathsf m_u)$ is ergodic and completely conservative (resp. non-ergodic and completely dissipative);
        \item  $\sum_{\ga\in\Ga_{{u},R}}e^{-\psi_u(\mu_\theta(\ga))}=\infty$ for some $R>0$ (resp. $\sum_{\ga\in\Ga_{{u},R}}e^{-\psi_u(\mu_\theta(\ga))} < \infty$ for all $R > 0$).
 \end{enumerate}
\end{theorem}

\begin{proof}
Since $\m_u$ is $u$-balanced by Theorem \ref{thm.mixing}, the equivalences between (2)-(4) follow from Theorem \ref{main}. By Theorem \ref{thm.thetarank}, we have $(1) \Leftrightarrow (4)$.
\end{proof}

\subsection*{Codimension dichotomy for Anosov subgroups}
We now deduce Theorem \ref{main3}.
 We use the notation from Theorem \ref{main3} and set $\psi=\psi_u$. As in Section \ref{sec.codim}, we consider the quotient space ${W^\diamond} = \fa_{\theta}/(W \cap \ker \psi)$ and set $\Omega_{W^{\diamond}} = \Ga \ba \La_{\theta}^{(2)}\times W^\diamond $ (see \eqref{eqn.V}). We denote by $\m_u'$ the $A_{\theta}$-invariant Radon measure on $\Omega_{W^{\diamond}}$ such that $\m_u = \m_u' \otimes \Leb_{W \cap \ker \psi}$.
 As before, $\Omega_{W^\diamond}$ is a vector bundle over a {\it compact} metric space $\Omega_{\psi}$ with fiber $\br^{\dim W^{\diamond} -1}$, and
the local mixing theorem for the $\{a_{tu}\}$-flow on $\Omega_{W^{\diamond}}$ \cite[Thm. 2.5.2]{sambarino2022report} says that there exists $\kappa_u > 0$ such that for any $f_1, f_2 \in C_c(\Omega_{W^\diamond})$, \be\label{mmix} \lim_{t \to  \infty} t^{\frac{\dim {W^\diamond} - 1}{2}} \int_{\Omega_{{W^\diamond}}} f_1(x) f_2(x a_{tu}) d\m_u'(x) = \kappa_u \m_u'(f_1) \m_u'(f_2).\ee 
We then obtain the following  version of Proposition \ref{prop.poincareintegral}, using Corollary \ref{useful2} and \ref{mmix}: 
\begin{prop} \label{prop.poincareintegral2}  For $\delta = \psi(u) > 0$ and all sufficiently large $R>0$, we have that 
     \be \left(\int_1^T  t^{\frac{1 - \dim W^{\diamond}}{2}} dt\right)^{1/2} \ll 
     \sum_{\substack{\ga \in \Ga_{W, R}  \\ \psi(\mu_{\theta}(\ga)) \le \delta T}} e^{-\psi(\mu_{\theta}(\ga))}  \ll \int_1^T  t^{\frac{1 - \dim W^{\diamond}}{2}} dt\ee
     where the implied constants are independent of $T>2$.
\end{prop}

Since $\dim {W^\diamond} -1 = \codim W$ and hence
$\dim {W^\diamond} \le 3  \Leftrightarrow 
\codim W \le 2$, the following is immediate from Proposition \ref{prop.poincareintegral2}:
\begin{prop} \label{prop.subspaceanosov}
   If $\Ga$ is a Zariski dense $\theta$-Anosov subgroup of $G$, then 
   $$ 
\codim W \le 2  \Longleftrightarrow  \sum_{\ga \in \Ga_{W, R}} e^{-\psi(\mu_{\theta}(\ga))} = \infty \text{ for some } R > 0.$$
\end{prop}

Hence the equivalence $(1) \Leftrightarrow (4)$ in Theorem \ref{main3} follows.
Since the local mixing for $(\Omega_{W^{\diamond}}, \{a_{tu}\}, \m_u')$ implies that $\m_u'$ is $u$-balanced, and hence $\m_u$ is $W$-balanced, we can apply Theorem \ref{thm.subspacedichotomy} to obtain the equivalences (2)-(4) in Theorem \ref{main3}. Therefore Theorem \ref{main3} follows.

\bibliographystyle{plain} 

\end{document}